\documentclass[11pt]{article}

\usepackage{amsmath,amssymb,amsthm,color}
\usepackage{algorithm,algorithmicx,algpseudocode}
\usepackage{subcaption}
\usepackage[hidelinks,colorlinks=true,linkcolor=blue,citecolor=black]{hyperref}
\usepackage{fancybox,graphicx,bm,float,soul}               
\usepackage{lscape}

\newtheorem{definition}{Definition}
\newtheorem{assumption}{Assumption}

\newtheorem{theorem}{Theorem}
\newtheorem{lemma}[theorem]{Lemma}

\newtheorem{proposition}[theorem]{Proposition}

\newtheorem{remark}{Remark}[section]

\newcommand{\R}{\mathbb{R}}

\newcommand{\cO}{\mathcal{O}}
\newcommand{\cM}{\mathcal{M}}
\newcommand{\cA}{\mathcal{A}}
\newcommand{\cE}{\mathcal{E}}

\newcommand{\cT}{\mathcal{T}}
\newcommand{\cC}{\mathcal{C}}

\newcommand{\bpmat}{\begin{pmatrix}}
\newcommand{\epmat}{\end{pmatrix}}
\newcommand{\Symn}{\mathcal{S}^{n \times n}}

\newcommand{\Tr}{\text{Tr}}
\newcommand{\bx}{\bm{x}}
\newcommand{\bxkp}{\bm{x}^{k+1} }
\newcommand{\bxk}{\bm{x}^{k} }
\newcommand{\bykp}{\bm{y}^{k+1}}
\newcommand{\byk}{\bm{y}^{k}}

\newcommand{\bb}{\bm{b}}

\newcommand{\bq}{\bm{q}}
\newcommand{\by}{\bm{y}}
\newcommand{\bZ}{\bm{Z}}
\newcommand{\bc}{\bm{c}}

\newcommand{\bd}{\bm{d}}
\newcommand{\be}{\bm{e}}
\newcommand{\bu}{\bm{u}}

\newcommand{\bz}{\bm{z}}
\newcommand{\bX}{\bm{X}}
\newcommand{\bY}{\bm{Y}}

\newcommand{\bQ}{\bm{Q}}

\newcommand{\bA}{\bm{A}}
\newcommand{\bU}{\bm{U}}

\newcommand{\bC}{\bm{C}}

\newcommand{\bV}{\bm{V}}
\newcommand{\barf}{\bar{f}}
\newcommand{\barc}{\bar{c}}
\newcommand{\blambda}{\bm{\lambda}}

\newcommand{\bnabla}{\bm{\nabla}}

\newcommand{\bSigma}{\bm\Sigma}

\newcommand{\Retr}{\text{\rm Retr}}

\newcommand{\Diag}{\operatorname{Diag}}

\newcommand{\grad}{\operatorname{grad}}
\DeclareMathOperator*{\argmin}{argmin}

\newcommand{\Rmnum}[1]{\uppercase\expandafter{\romannumeral #1}} 
\usepackage[raggedright]{titlesec}
\titleformat{\chapter}{\centering\Huge\bfseries}{Chapter \Rmnum{\thechapter} }{1em}{} 

\usepackage{mathrsfs} 
\usepackage{enumerate} 

\usepackage{multirow}
\usepackage{stmaryrd}

\title{ 
General Constrained Matrix Optimization
}

\author{ Casey Garner\thanks{School of Mathematics, University of Minnesota (\href{mailto:garne214@umn.edu}{garne214@umn.edu}, \href{mailto:lerman@umn.edu}{lerman@umn.edu}) }
\hspace{1cm}
Gilad Lerman\footnotemark[1]
\hspace{1cm}
Shuzhong Zhang\thanks{Department of Industrial and Systems Engineering (\href{mailto:zhangs@umn.edu}{zhangs@umn.edu})}}

\date{\today}

\begin{document}
\maketitle

\vspace{-0.2in}
\begin{abstract}
{
This paper presents and analyzes the first matrix optimization model which allows general coordinate and spectral constraints. 
The breadth of problems our model covers is exemplified by a lengthy list of examples from the literature, including semidefinite programming, matrix completion, and quadratically constrained quadratic programs (QCQPs), and we demonstrate our model enables completely novel formulations of numerous problems.
Our solution methodology leverages matrix factorization and constrained manifold optimization to develop an equivalent reformulation of our general matrix optimization model for which we design a feasible, first-order algorithm. 
We prove our algorithm converges to $(\epsilon,\epsilon)$-approximate first-order KKT points of our reformulation in $\mathcal{O}(1/\epsilon^2)$ iterations. 
The method we developed applies to a special class of constrained manifold optimization problems and is one of the first which generates a sequence of feasible points which converges to a KKT point.
We validate our model and method through numerical experimentation. 
Our first experiment presents a generalized version of semidefinite programming which allows novel eigenvalue constraints, and our second numerical experiment compares our method to the classical semidefinite relaxation approach for solving QCQPs. 
For the QCQP numerical experiments, we demonstrate our method is able to dominate the classical state-of-the-art approach, solving more than ten times as many problems compared to the standard solution procedure.

\vspace{3mm}

\vspace{3mm}
}
\end{abstract}

\section{Introduction}
The ubiquity of matrix optimization is unquestionable and the utility of the models is even more so. 
From matrix completion \cite{mat_completion,matrix_completion_with_noise}, various forms of principal component analysis \cite{jiang2015tensor,zou2006sparse}, max-cut \cite{goemans1995improved}, robust subpsace recovery \cite{robust_subspace_recovery_lerman,rsr}, covariance estimation \cite{fan2016overview}, semidefinite programming \cite{vandenberghe1996semidefinite}, eigenvalue optimization \cite{lewis2003mathematics,lewis1996eigenvalue}, Markov processes \cite{boyd2004fastest}, etc., we see the applications of minimizing and maximizing functions whose arguments are matrices subject to various forms of constraints. 
These constraints come in two varieties: coordinate constraints and spectral constraints. 
The coordinate constraints are the most readily visible and most prominently leveraged. 
The spectral constraints which require the eigenvalues and/or singular values to satisfy some set of conditions are present but to a lesser degree. 
One of the most common forms of spectral constraints is given by enforcing a matrix to be positive definite or semidefinite; however, this is a relatively tame condition when we consider the vast possibilities of restricting the spectrum.

The study of how to solve matrix optimization problems with general spectral constraints is still in its infancy.
Our prior work \cite{garnerspec} was the first model allowing general linear inequality constraints on the spectrum of a symmetric matrix. 
Shortly following our first paper, Ito and Louren{\c{c}}o came out with the second paper on this problem. In \cite{ito2023eigenvalue}, they utilized the work of a relatively unknown paper by Gowda \cite{gowda2019optimizing} to enable solving matrix optimization problems with general convex constraints on the spectrum; however, neither of these works addressed how to solve matrix optimization problems with both general coordinate and spectral constraints. 

This paper presents the first model and algorithm which enables general constraints on both the coordinates and spectrum of a matrix. 
Namely, in these pages we develop and analysis a method  
to find approximate solutions to models of the form: 
\begin{align}\label{eqn:gen_spec_coord}
\min&\; F(\bX) \\ 
\text{s.t.}&\; \bm{G}(\bX) = \bm{0}, \nonumber \\
          \;& \bm{g}(\bm{\sigma}(\bX)) \leq \bm{0}, \nonumber \\
          \;& \bX \in \mathbb{S}, \nonumber 
\end{align}
where $\mathbb{S}$ can be any one of many sets including: $\R^{m \times n}$, $\Symn$, and $\R^{n_1 \times \cdots \times n_k}$, i.e., real $m \times n$ matrices, real $n \times n$ symmetric matrices, and real tensors, and $\bm{\sigma}:\mathbb{S} \rightarrow \R^p$ maps the matrix or tensor $\bX$ to its ordered vector of singular values, eigenvalues, or generalized singular values. 
So, the method we develop extends beyond even matrices to tensors, enabling a novel modeling power not yet leveraged in the literature.

\begin{remark}
It is worth noting the only reason we choose to present \eqref{eqn:gen_spec_coord} with equality constraints only on the coordinates, i.e., $\bm{G}(\bX) = \bm{0}$, and not an additionally set of inequality constraints, e.g., $\bm{H}(\bX) \leq \bm{0}$, was for ease of presentation through a reduction of notation and terms. 
The analysis to come admits additionally inequality constraints with minor alterations to the assumptions, 
and our numerical experiments on quadratically constraints quadratic programs contain inequality constraints, further supporting this. 
\end{remark}

\section{Motivating Examples}
The generality of \eqref{eqn:gen_spec_coord} is vast, including multiple important applications and problem classes. 
To showcase this fact, let us investigate some examples and see how \eqref{eqn:gen_spec_coord} enables new constraints on the spectrum which before had been impossible, while also extending essential models in the literature.  

\subsection{Generalized Semidefinite Programs}
The impact of semidefinite programming (SDP) is nearly impossible to overstate, though it took some time before the salience of the model was truly realized and established. 
Today, we have seen applications of SDP extend from control theory to economics and most places in-between \cite{wolkowicz2012handbook,vandenberghe1996semidefinite}.
The standard form of a semidefinite program can be expressed as
\begin{align}\label{eq:SDP}
\min&\; \langle \bm{C}_0, \bX \rangle \\
\text{s.t.}&\; \langle \bm{C}_i, \bX \rangle = b_i , \; i=1,\hdots, p, \nonumber \\
           &\; \bX \succeq \bm{0},\nonumber 
\end{align}
where $\bm{C}_i \in \Symn$ for $i=0,1,\hdots,p$. 
By appropriately defining the functions in \eqref{eqn:gen_spec_coord} and letting $\mathbb{S} = \Symn$, we see the standard SDP model is a special instance of our general matrix model, making \eqref{eqn:gen_spec_coord} a generalization of semidefinite programming. 
We can extend SDPs and define the following {\it generalized SDP} model:
\begin{align}\label{eq:gen_SDP}
\min&\; \langle \bm{C}_0, \bX \rangle \\
\text{s.t.}&\; \langle \bm{C}_i, \bX \rangle = b_i , \; i=1,\hdots, p, \nonumber \\
           \;& \bm{A} \blambda(\bX) \leq \bb, \nonumber \\
          \;& \bX \in \Symn, \nonumber 
\end{align}
where $\blambda(\bX) := (\lambda_1(\bX), \hdots, \lambda_n(\bX))^\top$ is the vector of eigenvalues of $\bX$ listed in descending order, i.e., $\lambda_i(\bX) \geq \lambda_{i+1}(\bX)$ for $i=1,\hdots, n-1$.
So, letting $\bb = \bm{0}$ and $\bA = - \bm{I}$, we see \eqref{eq:gen_SDP} becomes \eqref{eq:SDP}. 
Thus, our general matrix model enables entirely new classes of constraints on the eigenvalues not yet attempted, while extending an essential problem class.  
{\color{black} Clearly, \eqref{eq:gen_SDP} is no longer a convex formulation. 
For example, a constraint as simple as $\lambda_1(\bX) \ge c_1$ and $\lambda_{n-1}(\bX) \le c_{n-1}$ with $0 < c_{n-1} < c_1$ would be nonconvex. 
In the case when $\bX$ is a Laplacian matrix, the condition $\lambda_{n-1}(\bX) \le c_{n-1}$ relates to limiting the connectivity of a graph which might well be desirable in certain design circumstances. Furthermore, the eigenvectors associated with the second smallest eigenvalue are known as Fiedler vectors and they have important applications to combinatorial optimization problems; see \cite{de2007old} and the references therein. 
}

\subsection{Matrix Preconditioning}
The subject of matrix preconditioning often arises in practical applications \cite{chen_2005,wathen_2015}. 
The topic is broached from the desire to find a relatively simple invertible matrix to improve the stability and solvability of various numerical procedures, e.g., solving a system of linear equations or minimizing a quadratic form. 
One of many preconditioning models is to locate good feasible solutions to
\begin{align}\label{eq:precond}
\min&\; \frac{1}{2}\| \bX \bA \bX^\top - \bm{I}\|_F^2  \\
\text{s.t.} \;& \bX \in \mathcal{C}, \nonumber 
\end{align}
where $\mathcal{C}$ enforces some structure upon the matrix $\bX$. 
Assuming $\bA \in \Symn$, the pre- and post-multiplication by $\bX$ ensures symmetry is maintained.
This model relates directly to \eqref{eqn:gen_spec_coord} through the constraint set $\cC$. 
For example, it would be reasonable to desire $\bX$ to a be a well-conditioned block diagonal matrix. 
Writing this explicitly, we have the model 
\begin{align}\label{eq:gen_precond}
\min&\; \frac{1}{2}\| \bX \bA \bX^\top - \bm{I}\|_F^2  \\
\text{s.t.} \;& \bX  = \begin{pmatrix} \bX_1 & & \bm{0} \\ & \ddots & \\ \bm{0} & & \bX_p \end{pmatrix} \in \Symn, \nonumber \\
  \;& \lambda_1(\bX_i) - \kappa \lambda_n(\bX_i) \leq 0, \nonumber \\
  \;& \lambda_n(\bX_i) \ge \delta, \nonumber \\
  \;& \bX_i \in \mathcal{S}^{s \times s}, i=1,\hdots, p, \nonumber 
\end{align}
where $\delta$ and $\kappa$ are nonnegative and $sp = n$. 
Note, the eigenvalue constraint ensures the condition number of $\bX$ is bounded above by $\kappa$ and each matrix forming $\bX$ is well-conditioned. 
This single example demonstrates the versatility of \eqref{eqn:gen_spec_coord} for modeling optimization models related to matrix preconditioning. 

\subsection{Inverse Eigenvalue and Singular Value Problems and Constrained PCA}
A large body of literature exists on what are known as inverse eigenvalue and singular value problems \cite{chu2005inverse,chu1998inverse}.
The general concern is to form a matrix from prescribed spectral information. 
For example, the entire spectrum of the matrix may be provided or some subset of it. 
Then, the goal is to construct a matrix that has this structure, possibly subject to additional structural constraints. 
The prevalence of problems such as this are hard to overstate as applications of them include: seismic tomography, remote sensing, geophysics, circuit theory, and more; see the aforementioned works by Chu for further discussion. 
These problems naturally relate to \eqref{eqn:gen_spec_coord} as it models both coordinate and spectral constraints. 
For example, a simple instance of an inverse singular value problem is the following projection problem, 
\begin{align}\label{eq:projection_inv_eig}
\min&\; \frac{1}{2}\| \bX  - \bA\|_F^2  \\
\text{s.t.} \;& \bX \in \mathcal{C}_1, \nonumber \\
 \;& \bm{\sigma}(\bX) \in \mathcal{C}_2, \nonumber
\end{align}
where $\bA \in \R^{m \times n}$, $\bm{\sigma}(\bX) := (\sigma_1(\bX), \hdots, \sigma_{\min(m,n)}(\bX))$ is the vector of singular values of $\bX$ in descending order, and $\cC_1$ and $\cC_2$ correspond to coordinate and spectral constraints respectively.
So, \eqref{eq:projection_inv_eig} seeks to project $\bA$ onto general matrix constraints. 
Hence, if $\bA$ satisfies the constraints it is a solution to the inverse singular problem described by $\cC_1$ and $\cC_2$, otherwise the model seeks to find the closet matrix to $\bA$ which solves the inverse problem. 

The relevance of this problem is further demonstrated by the fact \eqref{eq:projection_inv_eig} becomes the principal component analysis (PCA) model by the appropriate selection of the constraints. 
Letting $\cC_1 = \R^{m \times n}$ and $C_2 = \{ \bm{z} \in \R^{\min(m,n)}_\geq \; | \; z_1 >0, z_2 = \cdots = z_{\min(m,n)} = 0 \}$, the model becomes 
\begin{align}
\min&\; \frac{1}{2}\| \bX  - \bA\|_F^2  \nonumber \\
\text{s.t.} \;& \text{rank}(\bX) = 1, \; \bX \in \R^{m \times n}, \nonumber 
\end{align}
whose solution is the first principal component of $\bA$. However, our model extends beyond this to enable constrained versions of PCA. For example, we have the following nonnegative PCA model,
\begin{align}
\min&\; \frac{1}{2}\| \bX  - \bA\|_F^2  \nonumber  \\
\text{s.t.} \;& \text{rank}(\bX) = 1, \nonumber \\
            \;& \bX \geq \bm{0},\;  \bX \in \R^{m \times n}, \nonumber 
\end{align}
where $\bX \geq \bm{0}$ enforces $X_{ij} \geq 0$ for all $i$ and $j$. 
We can also move beyond simple rank constraints to enforce far more general constraints on the singular values. 
For instance, we could constrain the top $k$-singular values to account for between $p_1$ and $p_2$ percent of the total information of the matrix, that is, 
\begin{align}\label{eq:pca_II}
\min&\; \frac{1}{2}\| \bX  - \bA\|_F^2  \\
\text{s.t.} \;\;\;& p_1 \leq \frac{\sigma_1(\bX) + \cdots + \sigma_k(\bX) }{ 
 \sigma_1(\bX) + \cdots + \sigma_{\min(m,n)}(\bX) } \leq p_2, \nonumber \\
            \;& \bX \geq \bm{0},\;  \bX \in \R^{m \times n}, \nonumber 
\end{align}
where we note \eqref{eq:pca_II} is a nonconvex model due to the spectral constraint.
Thus, \eqref{eqn:gen_spec_coord} enables novel extensions of PCA which have not yet been investigated and could be of much interest for future investigations. 

\subsection{Quadratically Constrained Quadratic Programs}\label{sec:QCQP_intro}
A problem of immense importance in signal processing and communications is that of solving quadratically constrained quadratic programs (QCQP), which can be formulated as,
\begin{align}\label{eq:QCQP}
\min&\; \bx^\top \bm{C} \bx \\
\text{s.t.}&\; \bx^\top \bm{F}_i \bx \geq g_i, \; i=1,\hdots, p, \nonumber \\
           &\; \bx^\top \bm{H}_i \bx = l_i, \; i=1,\hdots, q, \nonumber \\
           &\; \bx \in \R^n, \nonumber 
\end{align}
where $\bm{C}, \bm{F}_1, \hdots, \bm{F}_p, \bm{H}_1, \hdots, \bm{H}_q$ are general real symmetric matrices. To highlight one example, let us consider the Boolean quadratic program (BQP),
\begin{align}\label{eq:BQP}
\min&\; \bx^\top \bm{C} \bx \\
\text{s.t.}&\; x_i^2 = 1, \; i=1,\hdots, n, \nonumber \\
           &\; \bx \in \R^n. \nonumber 
\end{align}
The BQP is known to be NP-Hard problem and has applications in multiple-input-multiple-output
(MIMO) detection and multiuser detection.
And, a well-known important instance of BQP is the Max-Cut problem, which seeks to locate a maximum cut of a graph, i.e., a partition of the vertices of a graph into two sets $S_1$ and $S_2$ such that the number of edges connecting vertices in $S_1$ to vertices in $S_2$ is maximized.   

A powerful approach taken to try and solve \eqref{eq:QCQP} is to first reformulate it is as the equivalent matrix optimization model,
\begin{align}\label{eq:matQCQP}
\min&\; \langle \bm{C}, \bX \rangle \\
\text{s.t.}&\; \langle \bm{F}_i, \bX \rangle \geq g_i , \; i=1,\hdots, p, \nonumber \\
           &\; \langle \bm{H}_i, \bX \rangle = l_i, \; i=1,\hdots, q, \nonumber \\
           &\; \bX \succeq \bm{0}, \; \text{rank}(\bX)= 1,\nonumber 
\end{align}
then form and solve a semidefinite program relaxation (SDR) of \eqref{eq:matQCQP} by dropping the rank-1 condition, and then finally applying randomization techniques to find feasible solutions to \eqref{eq:QCQP}. 
This approach of convexification plus randomization for QCQPs has received much attention in recent years \cite{luo2010semidefinite}, and numerous results have been published proving the accuracy of SDR for certain QCQPs.
For example, the influential 0.8756 approximation result for the Max-Cut problem  demonstrates the solution to the SDR can sometimes be good \cite{goemans1995improved}; 
however, dropping the rank-1 condition in \eqref{eq:matQCQP} is not the only way to relax the problem. 
Our new general matrix optimization model admits several alternatives which we demonstrate can significantly outperform the state-of-the-art SDR approach.  
For example, we can approximate the rank-1 condition with different eigenvalue constraints, for example,
\begin{equation}\label{eq:rank_1_relax_1}
\{\bX \in \Symn \;| \; \lambda_1(\bX) \geq \delta, \; \lambda_i(\bX) \in [0,\delta],\; i=2,\hdots, n\},
\end{equation}
\begin{equation}\label{eq:rank_1_relax_2}
\{\bX \in \Symn \;| \; \lambda_i(\bX) \geq \alpha_i \lambda_{i+1}(\bX),\; i=1,\hdots, n-1\},
\end{equation}
where $\delta \geq 0$ and $\alpha_i \geq 1$, for $i=1,\hdots, n-1$.
In \cite{garnerspec}, we saw relaxations of this form applied to solving systems of quadratic equations was able to best the classical SDR approach. 
Later in Section~\ref{sec:QCQP_exp}, we demonstrate through a large number of numerical experiments that our general matrix optimization framework using \eqref{eq:rank_1_relax_1} to relax the rank constraints dominates the SDR approach to solving a set of QCQPs.

\subsection{Structured Linear Regression}
Across the scientific disciplines few mathematical paradigms are utilized more than linear regression.
The novel framework we are proposing can also be applied to perform novel structured linear regression between vector spaces. 
For example, let $\{(\bx^i,\bb^i)\}_{i=1}^{N} \subset \R^n \times \R^m$ be a set of observed data points. 
Our goal is to locate a constrained linear transform $\mathcal{L}:\R^n \rightarrow \R^m$ which solves or nearly solves $\mathcal{L}(\bx^i) = \bb^i$ for all $i$. 
Since all such linear transforms are representable by matrices $\bA \in \R^{m \times n}$, this leads to the following general matrix optimization model, 
\begin{align}\label{eq:linear_regression}
\min&\;\; \frac{1}{2} \sum_{i=1}^{N} \| \bA \bx^i - \bb^i\|^2 \\
\text{s.t.}&\; \bA \in \mathcal{C}_1, \nonumber \\  
           &\; \lambda(\bA) \in \mathcal{C}_2, \nonumber 
\end{align}
where $\cC_1$ and $\cC_2$ represent the coordinate and spectral constraints on $\bA \in \R^{m \times n}$ respectively.
For example, say we know the process generating our data set $\{(\bx^i,\bb^i)\}_{i=1}^{N}$ is given by a symmetric tri-diagonal matrix with a particular spectral structure, e.g., $\text{rank}(\bA) \leq r$, then we can leverage this information to form the structured linear regression model 
\begin{align}\label{eq:linear_regression_example_1}
\min&\;\; \frac{1}{2} \sum_{i=1}^{N} \| \bA \bx^i - \bb^i\|^2 \\
\text{s.t.}&\; \bA \in \mathcal{T}_n, \nonumber \\
           &\;  \text{rank}(\bA) \leq r, \nonumber 
\end{align}
where $\mathcal{T}_n := \{ \bX \in \Symn \; | \; X_{ij} = 0 \text{ if } j \leq i-1 \text { or } j \geq i+1\}$.
This framework directly relates to the work in robust subspace recovery \cite{robust_subspace_recovery_lerman,lerman2015robust} where our framework enables a much greater freedom to leverage spectral information which might be present to a practitioner.  

%
%

\subsection{Matrix Completion}
The task of filling-in the missing entries of a matrix, also known as, matrix completion, has found immense success in both image recovery \cite{zhang2012matrix} and recommender systems \cite{chen2022review}, among other applications \cite{fazel2002matrix}. 
Formally, the matrix completion problem can be formulated as: {\it given a partially known matrix $\bm{M} \in \R^{m\times n}$ such that $M_{ij}$ for all $(i,j) \in \Omega$ are known, recover the unknown entries in $\bm{M}$}. A classic formulation of this problem is the rank minimization problem, 
\begin{align}\label{eq:rank_min}
\min&\;\; \text{rank}(\bX) \\
\text{s.t.}&\; X_{i,j} = M_{i,j}, \; (i,j) \in \Omega, \nonumber \\
&\; \bX \in \R^{m \times n}. \nonumber 
\end{align}
Due to the fact \eqref{eq:rank_min} is an extremely challenging problem, it is common to replace the rank function with a convex approximation of it to obtain, 
\begin{align}\label{eq:rank_min_cnvx}
\min&\;\; \|\bX\|_* \\
\text{s.t.}&\; X_{i,j} = M_{i,j}, \; (i,j) \in \Omega, \nonumber \\
&\; \bX \in \R^{m \times n}, \nonumber 
\end{align}
where $\|\bX\|_*$ is the nuclear norm of $\bX$.
The main benefit of \eqref{eq:rank_min_cnvx} is the fact it is convex and in the seminal work of  Cand{\`e}s and Recht \cite{candes2012exact} they proved if enough entries of $\bm{M}$ were known then the matrix could be perfectly recovered with high probability. Central to this approach is the assumption the matrix being recovered has a low-rank structure, which very well may not be the case. 
In \cite{garnerspec}, our experiments demonstrated if spectral information was known about the matrix being recovered than our methods could nearly double the chance at recover when compared to solving \eqref{eq:rank_min_cnvx}.

Thus, we are motivated to apply \eqref{eqn:gen_spec_coord} to propose the following matrix completion model with spectral constraints, 
\begin{align}\label{eq:mat_comp_spec}
\min&\;\; F(\bX) \\
\text{s.t.}&\; X_{i,j} = M_{i,j}, \; (i,j) \in \Omega, \nonumber \\
           &\; \bm{\sigma}(\bX) \in \cC, \nonumber \\
           &\; \bX \in \R^{m \times n}, \nonumber 
\end{align}
where $\cC$ is derived from the spectral information at our disposal and $F:\R^{m \times n} \rightarrow \R$ is an objective function chosen to bias the matrix as dictated by the user.  
For example, it might be the case that a subset of our known entries are noisy, then it would make sense to let $F(\bX) := \frac{1}{2}\| P_{\Omega'}(\bX - \bm{M})\|_F^2$, where 
\[
\left[P_{\Omega'}(\bX)\right]_{ij} = \begin{cases} 0 &,\; (i,j) \in \Omega' \\ X_{ij} &,\; \text{otherwise}   \end{cases},   
\]
for all $i,j$, and $\Omega'$ is the subset of the known entries which are noisy. 
Hence, in this example we would replace $\Omega$ in the constraint with $\Omega''$ to be the set of known noiseless entries, yielding the noisy matrix completion model with spectral constraints,
\begin{align}\label{eq:mat_comp_spec_noisy}
\min&\;\; \frac{1}{2}\| P_{\Omega'}(\bX - \bm{M})\|_F^2 \\
\text{s.t.}&\; X_{i,j} = M_{i,j}, \; (i,j) \in \Omega'', \nonumber \\
           &\; \bm{\sigma}(\bX) \in \cC, \nonumber \\
           &\; \bX \in \R^{m \times n}. \nonumber 
\end{align}
It is worth noting that learning structural information about the spectrum for particular problems is very possible in practice; see for example Theorem 1 of \cite{kasten2019algebraic}. 

\subsection{And More...}
The prior examples should hopefully be sufficient to see the breadth of problems modeled and extended by \eqref{eqn:gen_spec_coord}. Further examples include: low-rank matrix optimization \cite{zhu2018global,nguyen2019low}, tensor completion \cite{gandy2011tensor}, and synchronization problems \cite{lerman2022robust,shi2020message}, among others; however, we do not have the space here to elaborate on these important applications.  

\section{An Equivalent Reformulation}\label{sec:reformulation}
The primary approach we take in this paper to solve \eqref{eqn:gen_spec_coord} is to form a decomposed model by utilizing standard matrix and tensor factorizations. 
For example, in the case $\mathbb{S} = \Symn$ in \eqref{eqn:gen_spec_coord} we utilize the eigenvalue decomposition. 
It is well known a symmetric matrix $\bX$ has an eigenvalue/spectral decomposition such that $\bX = \bm{Q} \Diag(\blambda(\bX)) \bm{Q}^\top$ where $\bm{Q} \in \mathcal{O}(n,n)$, the matrix manifold of orthogonal matrices, and $\Diag(\blambda(\bX))$ is the real diagonal matrix formed by the eigenvalues of $\bX$. 
Thus, using the eigenvalue decomposition, we can equivalently rewrite  
\begin{align}\label{eq:sym_model}
\min&\; F(\bX) \\ 
\text{s.t.}&\; \bm{G}(\bX) = \bm{0}, \nonumber \\
&\; \bm{g}(\blambda(\bX)) \leq \bm{0},\nonumber \\
&\bX \in \Symn, \nonumber
\end{align}
as the decomposed model 
\begin{align}\label{eq:decomp_model}
\min&\; f(\bQ, \blambda):= F(\bQ \Diag(\blambda) \bQ^\top) \\ 
\text{s.t.}&\; \bm{G}(\bQ \Diag(\blambda) \bQ^\top) = \bm{0}, \nonumber \\
&\; \bm{\tilde{g}}(\blambda) \leq \bm{0}, \nonumber \\
&\bQ \in \mathcal{O}(n,n),\; \blambda \in \R^n, \nonumber
\end{align}
where ${\bm{\tilde{g}}}$ additionally enforces the condition the elements of $\blambda$ are in decreasing order, i.e., $\lambda_i \geq \lambda_{i+1}$ for all $i$. 
Now, it is clear the global minimums of \eqref{eq:sym_model} and \eqref{eq:decomp_model} are equivalent, but there might be some concern that the decomposed model introduces new local minimums which do not correspond to local minimums in the original model;
however, we prove this is not the case for local minimizers of \eqref{eq:sym_model} which have no repeated eigenvalues. 
Since a minor perturbation of the constraint in \eqref{eq:sym_model} ensures all feasible matrices have unique eigenvalues, our result demonstrates the strong tie between the local minimums of the decomposed model and the original which supports our decomposition approach to solving \eqref{eqn:gen_spec_coord}. 
%
%

Before stating and proving our main result, we first introduce a few definitions and prove a set of lemmas. 
Let $\R^n_{\geq}:=\{ \bx \in \R^n \; | \; x_1 \geq  x_2 \geq \cdots \geq x_n\}$, $\|\cdot\|_F$ and $\|\cdot\|$ denote the Frobenius and Euclidean norms respectively, and given $\bX \in \Symn$, $\bQ \in \mathcal{O}(n,n)$, $\blambda \in \R^n_{\geq}$ and $\delta > 0$ we define
\begin{equation}\label{eq:spec_ball}
\bm{B}\left( (\bQ,\blambda), \delta \right) := \left\{\bar{\bQ}\Diag(\bar{\blambda}) \bar{\bQ}^\top \; | \; \|\bar{\bQ} - \bQ \|_F \leq \delta, \; \|\bar{\blambda} - \blambda\| \leq \delta, \; \bar{\bQ} \in \mathcal{O}(n,n), \;\bar{\blambda}\in \R^n_{\geq}    \right\}, 
\end{equation}
\begin{equation}\label{eq:ball}
\bm{B}\left(\bX, \delta \right) := \left\{\bar{\bX} \in \Symn \; | \; \|\bar{\bX} - \bX \|_F \leq \delta \right\}.
\end{equation}
\begin{lemma}\label{lem:lemma_1}
Given $\bX \in \Symn$, $\delta>0$, and a spectral decomposition of $\bX$ such that $\bX = \bQ_x \Diag(\blambda(\bX)) \bQ_x^\top$ with $\bQ_x \in \mathcal{O}(n,n)$, there exists $\delta'>0$ such that
\[
\bm{B}\left( (\bQ_x,\blambda(\bX)), \delta' \right) \subseteq \bm{B}\left(\bX, \delta \right).
\]
\end{lemma}
\begin{proof}
Define 
\[
\delta':= \min \left\{ \frac{\delta}{2}, \; \frac{\delta}{2(\delta+2\|\blambda(\bX)\|)} \right\}. 
\]
Then, for all $\bQ \in \mathcal{O}(n,n)$ such that $\|\bQ - \bQ_x\|_F \leq \delta'$ and $\blambda \in \R^n_{\geq}$ such that $\|\blambda - \blambda(\bX)\| \leq \delta'$
\begin{align}
\| \bQ \Diag(\blambda) \bQ^\top - \bX\|_F &= \| \bQ \Diag(\blambda) \bQ^\top - \bQ_x \Diag(\blambda(\bX)) (\bQ_x)^\top\|_F \nonumber \\
&\leq \| (\bQ - \bQ_x) \Diag(\blambda) \bQ^\top \|_F + \|\Diag(\blambda) - \Diag(\blambda(\bX))\|_F \nonumber \\
&\hspace{1.0in}+ \| \bQ_x \Diag(\blambda(\bX))(\bQ-\bQ_x)^\top\|_F \nonumber \\
&\leq \|\bQ - \bQ_x\|_F \|\blambda\| + \|\blambda - \blambda(\bX)\| + \| \blambda(\bX)\|\cdot \|\bQ - \bQ_x\|_F  \nonumber \\
&\leq \delta'(2\|\blambda(\bX)\| + \delta') + \delta' \nonumber \\
&\leq \frac{\delta}{2(\delta+2\|\blambda(\bX)\|)} \cdot \left( 2\|\blambda(\bX)\|+ \delta\right) + \frac{\delta}{2} \nonumber \\
&= \delta, \nonumber  
\end{align}
and the result follows by the definition of $\bm{B}\left( (\bQ_x,\blambda(\bX)), \delta' \right)$.
\end{proof}
This lemma is crucial to be able to claim that any local minimum in \eqref{eq:sym_model} generates a local minimum in \eqref{eq:decomp_model}. We next move to a lemma which demonstrates an equivalence between different representations of the same symmetric matrix. 
Note, in Lemma~\ref{lem:lemma_2}, we let $\bm{e}_i$ denote the vector of the of all zeros except with a single one in the $i$-th entry. 

\begin{lemma}\label{lem:lemma_2}
Given $\bX \in \Symn$ with unique eigenvalues,  $\delta \geq 0$, and $\bm{E} \in \left[ \pm 1 \bm{e}_1 \;| \; \cdots \;| \; \pm 1 \bm{e}_n \right]$, we have for any spectral decomposition of $\bX$, i.e., $\bX = \bQ_x \Diag(\blambda(\bX)) \bQ_x^\top$ with $\bQ_x \in \mathcal{O}(n,n)$, that
\[
\bm{B}\left( (\bQ_x,\blambda(\bX)), \delta \right) 
= 
\bm{B}\left( (\bQ_x\bm{E},\blambda(\bX)), \delta \right).
\]
\end{lemma}
\begin{proof}
Let $\bQ \Diag(\blambda) \bQ^\top \in \bm{B}\left( (\bQ_x,\blambda(\bX)), \delta \right)$. Then, $\|\bQ - \bQ_x\|_F \leq \delta$, $\|\lambda - \lambda(\bX)\| \leq \delta$, and
\[
 \|\bQ\bm{E} - \bQ_x\bm{E}\|_F = \|\bQ - \bQ_x\|_F \leq \delta
\]
which implies
\[
\bQ \bm{E} \Diag(\blambda) \bm{E}^\top \bQ^\top =\bQ \Diag(\blambda)\bQ^\top  \in \bm{B}\left( (\bQ_x\bm{E},\blambda(\bX)), \delta \right).
\]
On the other hand, letting $\tilde{\bQ} \Diag(\blambda)\tilde{\bQ}^\top \in 
\bm{B}\left( (\bQ_x\bm{E},\blambda(\bX)), \delta \right)$ and setting $\bQ = \tilde{\bQ}\bm{E}$ we see
\[
\|\bQ - \bQ_x\|_F = \| \tilde{\bQ}\bm{E} - \bQ_x \|_F = \| \tilde{\bQ} - \bQ_x \bm{E} \|_F \leq \delta.
\]
Thus, 
\[
 \tilde{\bQ} \Diag(\blambda) \tilde{\bQ}^\top = \bQ \Diag(\blambda) \bQ^\top\in 
\bm{B}\left( (\bQ_x,\blambda(\bX)), \delta \right). 
\]
\end{proof}
This lemma demonstrates if we are in the setting where $\bX$ has only unique eigenvalues then any of the neighborhoods about the different representations of $\bX$ create the same image in the space of symmetric matrices. 
From this we shall be able to conclude that a local minimum at any one of the representations shall imply a local minimum at all of the representations. 
Finally, we want to demonstrate that we can generate a ball in the space of symmetric matrices which shall be contained inside the image of the balls over decompositions.   
Note, in the event the argument is a matrix, $\|\cdot\|_2$ denotes the matrix 2-norm. 

\begin{lemma}\label{lem:lemma_3}
Given $\bX \in \Symn$ with unique eigenvalues, $\delta'>0$, and the spectral decomposition $\bX = \bQ_x \Diag(\blambda(\bX)) \bQ_x^\top$, there exists $\delta''>0$ such that
\[
\bm{B}\left(\bX, \delta'' \right) \subset \bm{B}\left( (\bQ_x,\blambda(\bX)), \delta' \right).
\]
\end{lemma}
\begin{proof}
Define $\tau:=\min_{1 \leq i \leq n-1}\; | \lambda_i(\bX) - \lambda_{i+1}(\bX)|$ and let 
\[
\delta'' = \min\left\{1, \tau/2, \delta', \frac{\tau^2}{\sqrt{4n\tau^2 + 16n}}\delta'\right\}.
\]
Then, given $\bZ \in \Symn$ such that $\|\bZ - \bX\|_F \leq \delta''$ it follows 
\[
\|\blambda(\bZ) - \blambda(\bX)\| \leq \|\bZ - \bX\|_2 \leq \|\bZ - \bX\|_F \leq \delta'',
\]
which implies the eigenvalues of $\bZ$ are simple. 
By Lemma~\ref{lem:lemma_2}, the result shall follow provided $\bZ \in B( (\bQ_x \bm{E}), \blambda(\bX)), \delta')$ for some $\bm{E}$ as defined in Lemma~\ref{lem:lemma_2}. 
Now, by Theorem~3.3.7 in \cite{Ortega_Num_Analysis} it follows there exists normalized eigenvectors $\bm{u}_j$ and $\bm{v}_j$ of $\bZ$ and $\bX$ corresponding to $\lambda_j(\bZ)$ and $\lambda_j(\bX)$ respectively for $j=1,\hdots, n$ such that 
\begin{equation}\label{eq:lem3_help1}
\| \bu_j - \bm{v}_j \| \leq \gamma(1+\gamma^2)^{1/2},
\end{equation}
where $\gamma = \| \bZ - \bX\|_2/ (\tau - \| \bZ - \bX\|_2)$. 
Since $\bZ$ and $\bX$ both have simple eigenvalues there exists $\bm{E}$ and $\bm{Q}_z$ such that 
\[
\bm{Q}_x \bm{E} = \left[ \bm{v}_1 \;|\; \cdots \;|\; \bm{v}_n \right] \text{ and } \bm{Q}_z = \left[ \bm{u}_1 \;|\; \cdots \;|\; \bm{u}_n \right] 
\]
with $\bZ = \bm{Q}_z \Diag(\blambda(\bZ)) \bm{Q}_z^\top$ and $\bm{Q}_z \in \mathcal{O}(n,n)$. 
Noting $\gamma \leq {2\delta''}/{\tau}$ we see 
\begin{align}
\| \bm{Q}_z - \bm{Q}_x\bm{E}\|_F^2 &\leq \sum_{i=1}^n \|\bm{u}_i - \bm{v}_i \|^2 \nonumber \\
&\leq n \gamma^2 (1+\gamma^2) \nonumber \\
&\leq \frac{4n}{\tau^2}(\delta'')^2 \left(1 + \frac{4 (\delta'')^2}{\tau^2}\right) \nonumber \\
&\leq \left( \frac{4n\tau^2 + 16n}{\tau^4}\right)(\delta'')^2 \nonumber  \\
&\leq (\delta')^2, \nonumber 
\end{align}
where the second inequality followed from \eqref{eq:lem3_help1} and the remaining inequalities were a result of the definition of $\delta''$.
\end{proof}

We now combine our lemmas to prove our main result, but first we need to define a local minimum for our models. 
\begin{definition}\label{def:local_min}
Define the following feasible regions for \eqref{eq:sym_model} and \eqref{eq:decomp_model}: 
\[
\mathcal{X}_1:= \left\{\bX \in \Symn \; | \; \bm{G}(\bX) = \bm{0}, \; \blambda(\bX) \in \mathcal{C} \right\},
\]
\[
\mathcal{X}_2:=\left\{ (\bQ, \blambda) \in \mathcal{O}(n,n) \times \R^n_{\geq} \; | \; \bm{G}(\bQ \Diag(\blambda) \bQ^\top) = \bm{0}, \; \blambda \in \mathcal{C} \right\},
\]
where $\cC$ is a closed subset of $\R^n$. 
Then, the matrix $\bX^* \in \Symn$ is a local minimum of \eqref{eq:sym_model} if there exists $\delta_1 \geq 0$ such that $F(\bX^*) \leq F(\bX)$  for all $\bX \in \bm{B}(\bX^*,\delta_1)\cap \mathcal{X}_1$. 
Similarly, $(\bQ^*,\blambda^*) \in \mathcal{O}(n,n) \times \R^n_{\geq}$ is a local minimum of \eqref{eq:decomp_model} provided there exists $\delta_2 \geq 0$ such that $f(\bQ^*,\blambda^*) \leq f(\bQ,\blambda)$ for all\\ 
\[
(\bQ,\blambda) \in \left\{ (\bar{\bQ}, \bar{\blambda}) \in \mathcal{O}(n,n) \times \R^n_{\geq} \; | \; \|\bar{\bQ} - \bQ^* \|_F \leq \delta_2, \; \|\bar{\blambda} - \blambda^*\| \leq \delta_2\right\} \cap \mathcal{X}_2.
\]
\end{definition}
We now present our main theorem for this section. 
\begin{theorem}\label{thm:local_mins}
A matrix $\bX^* \in \Symn$ with unique eigenvalues is a local minimum of \eqref{eq:sym_model} if and only if one of its spectral decompositions is a local minimum of \eqref{eq:decomp_model}. Furthermore, if one of the spectral decompositions of $\bX^*$ is a local minimum of \eqref{eq:decomp_model}, then all of its spectral decompositions form local minimums. 
\end{theorem}
\begin{proof}
Assume $\bX^*$ is a local minimum of \eqref{eq:sym_model}. Thus, there exists $\delta_1 \geq 0$ such that $F(\bX^*) \leq F(\bX)$  for all $\bX \in \bm{B}(\bX^*,\delta_1)\cap \mathcal{X}_1$. 
Let $\bQ^* \Diag(\blambda(\bX^*)) \bQ^*$ be any spectral decomposition of $\bX^*$. Then by Lemma~\ref{lem:lemma_1} there exists $\delta_2 \geq 0$ such that 
$
\bm{B}\left( (\bQ^*,\blambda(\bX^*),\delta_2 \right) \subseteq \bm{B}(\bX^*,\delta_1).
$
Therefore, it follows that $f(\bQ^*,\blambda(\bX^*)) \leq f(\bQ,\blambda)$ for all  
\[
(\bQ,\blambda) \in \left\{ (\bar{\bQ}, \bar{\blambda}) \in \mathcal{O}(n,n) \times \R^n \; | \; \|\bar{\bQ} - \bQ^* \|_F \leq \delta_2, \; \|\bar{\blambda} - \blambda(\bX^*)\| \leq \delta_2\right\} \cap \mathcal{X}_2.
\]
Since no particular spectral decomposition of $\bX^*$ was selected, any spectral decomposition of $\bX^*$ generates a local minimum of \eqref{eq:decomp_model}. To prove the other direction, assume $(\bQ^*,\blambda^*) \in \mathcal{O}(n,n) \times \R^n_{\geq}$ is a local minimum of \eqref{eq:decomp_model}. Then there exists $\delta_1 \geq 0$ such that $f(\bQ^*,\blambda^*) \leq f(\bQ,\blambda)$ for all 
\[
(\bQ,\blambda) \in \left\{ (\bar{\bQ}, \bar{\blambda}) \in \mathcal{O}(n,n) \times \R^n_{\geq} \; | \; \|\bar{\bQ} - \bQ^* \|_F \leq \delta_2, \; \|\bar{\blambda} - \blambda^*\| \leq \delta_2\right\} \cap \mathcal{X}_2
\]
which is equivalent to stating $F(\bQ^*\Diag(\blambda^*)(\bQ^*)^\top) \leq F(\bar{\bX})$ for all $\bar{\bX} \in \bm{B}\left((\bQ^*,\lambda^*), \delta_2\right) \cap \mathcal{X}_1$. By Lemma~\ref{lem:lemma_3}, it follows there exists $\delta_1 \geq 0$ such that $\bm{B}(\bQ^*\Diag(\blambda^*)(\bQ^*)^\top, \delta_1) \cap \mathcal{X}_1 \subseteq  \bm{B}\left((\bQ^*,\lambda^*), \delta_2\right) \cap \mathcal{X}_1$ which implies $\bX^* = \bQ^*\Diag(\blambda^*)(\bQ^*)^\top$ is a local minimum of \eqref{eq:sym_model}. 
\end{proof}

Thus, we have demonstrated the local minimums of the two models are identical in the case of local minimizers without repeated eigenvalues. 
This might seem like a strong restriction, 
but in reality it is always possible to use the eigenvalue constraints to ensure 
no repeated eigenvalues occur. 
Additionally, the continuity of the eigenvalues of a symmetric matrix with respect to perturbations of the coordinates of the matrix ensure that such a restriction only alters the value of the objective function in a boundable fashion, assuming some degree of smoothness in the objective function; 
see Theorem~4 in \cite{garnerspec} for a more substantial discussion on this point. 

\section{A General Block Manifold Model} 
The decomposition approach outlined in Section~\ref{sec:reformulation} leads us to formulate and study a new optimization model, which generalizes beyond the decomposition approach for symmetric matrices presented in \eqref{eq:decomp_model}. 
Thus, the central model guiding the remainder of our discussion is the block manifold model:
\begin{align}\label{eq:main_model}
\min&\; f(\bx, \by)  \\
\text{s.t.}&\; \bc(\bx, \by) = \bm{0}, \nonumber \\
&\; \bm{g}(\by) \leq \bm{0}, \nonumber \\
&\; \bx \in \cM, \; \by \in \R^n, \nonumber 
\end{align}
where $f:\cM \times \R^n \rightarrow \R$, $\bc:\cM \times \R^n \rightarrow \R^p$, and $\bm{g}: \R^n \rightarrow \R^s$ are differentiable and $\cM$ is a differentiable submanifold in a Euclidean space.
Depending on the choice of $\mathbb{S}$ in \eqref{eqn:gen_spec_coord}, the corresponding manifold in \eqref{eq:main_model} will vary, as dictated by the decomposition utilized. 
For example, if $\mathbb{S} = \R^{m \times n}$, then we shall apply a singular value decomposition to form the variables $\bx$ and $\by$ in \eqref{eq:main_model} and $\cM$ shall be a Cartesian product of Stiefel manifolds.
In this fashion, \eqref{eq:main_model} can be utilized to model all presented versions of \eqref{eq:main_model}. 
In the remainder of this paper, we shall develop, analyze,  and test an algorithm which converges to approximate first-order KKT points of \eqref{eq:main_model} when viewed as a constrained Riemannian optimization problem.  

\subsection{Riemannian Manifold Optimization Revisited}
As optimization over smooth manifolds shall be a major aspect of our discussion, we begin by presenting a brief introduction to the major geometric objects to be leveraged. 
Our presentation of Riemannian geometry for optimization is restricted to differentiable submanifolds of Euclidean spaces, i.e., vector spaces with a defined inner product, as this covers our examples of interest; however, our later analysis holds for general smooth Riemannian manifolds. 
The notation and descriptions presented in this section are inspired and drawn heavily from Chapters 3 and 10 of \cite{boumal2023introduction}. 
For readers interested in further details, the authors highly recommend Boumal's excellent text.

At a high level, a smooth manifold is a generalization of linear spaces in such a way calculus operations such as differentiation, integration, etc., are well-defined and made possible. 
In our presentation, we let $\mathcal{E}$ refer to a general finite dimensional vector space. If $\cE$ additionally has a defined inner product, then we call $\cE$ an Euclidean space.  
The simplified definition of an embedded submanifold of $\cE$ we choose to present is provided by Boumal: 
\begin{definition}\label{def:manifold}[Definition 3.6, \cite{boumal2023introduction}] 
Let $\cM$ be a subset of $\cE$. 
We say $\cM$ is a (smooth) embedded submanifold of $\cE$ if either of the following holds: 
\vspace{-0.1in}
\begin{enumerate}
\item $\cM$ is an open subset of $\cE$. Then we call $\cM$ an open submanifold. If $\cM = \cE$, we also call it a linear manifold.
\item For a fixed integer $k\geq 1$ and for each $\bx \in \cM$ there exists a neighborhood $U$ of $\bx$ in $\cE$ and a smooth function $h:U \rightarrow \R^k$ such that 
\begin{enumerate}
    \item If $\by \in  U$, then $h(\by) = \bm{0}$ if and only if $\by \in \cM$; and 
    \item rank $Dh(\bx) = k$.
\end{enumerate}
Such a function $h$ is called a local defining function of $\cM$ at $\bx$. 
\end{enumerate}
\vspace{-0.1in}
The set $\cE$ is often referred to as the embedding space or the ambient space of $\cM$. 
\end{definition}
An experienced student of Riemannian geometry shall note from $U$ and $h$ via the inverse function theorem charts in the more traditional presentation of general smooth manifolds can be obtained.

\begin{definition}\label{def:tangent_vector}[Definitions 3.7 \& 3.10, \cite{boumal2023introduction}]
Let $\cM$ be an embedded submanifold of $\mathcal{E}$. For all $\bx \in \cM$, the tangent space of $\cM$ at $\bx$ is defined as the set 
\begin{equation}\label{eq:TxM}
\mathcal{T}_{\bx}\cM := \{c'(0) \; | \; c: \mathcal{I} \rightarrow \cM \; \text{ is smooth around $0$ and } c(0) = \bx \}, 
\end{equation}
and elements $\bm{v} \in \mathcal{T}_{\bx}\cM$ are called tangent vectors to $\cM$ at $\bx$.  
So, $\bm{v} \in \mathcal{T}_{\bx}\cM$ if and only if there is a smooth curve on $\cM$ through $\bx$ such that $c'(0) = \bm{v}$.
\end{definition}

In the case of embedded submanifolds of a vector space, a local defining function at a point $\bx \in \cM$ presents a simple manner to define the set of tangent vectors at $\bx$. 

\begin{theorem}\label{thm:ker_to_tangent_vector}[Theorem 3.8, \cite{boumal2023introduction}]
Let $\cM$ be an embedded submanifold of $\mathcal{E}$. Consider $\bx \in \cM$ and the set $\mathcal{T}_{\bx}\cM$ \eqref{eq:TxM}. If $\cM$ is an open submanifold of $\mathcal{E}$, then $\mathcal{T}_{\bx}\cM = \mathcal{E}$. Otherwise, $\mathcal{T}_{\bx}\cM = \ker\left(Dh(\bx)\right)$ with $h$ any local defining function at $\bx$.  
\end{theorem}

The tangent space of $\cM$ at a point $\bx$ is intended to be a reasonable approximation of the manifold locally. 
This idea is seen most clearly from the description of the tangent space in Theorem~\ref{thm:ker_to_tangent_vector} as the kernal of the differential of a local defining function. 

We have new spaces over which we are defining functions, namely embedded submanifolds of vector spaces. 
Thus, it is prudent for us to consider what it means for such functions to be smooth over these manifolds.
This is clarified in the following definition.

\begin{definition}\label{eq:smooth_f_M}[Definition 3.23, \cite{boumal2023introduction}]
Let $\cM$ and $\cM'$ be embedded submanifolds of vector spaces $\cE$ and $\cE'$ respectively.
A map $F:\cM \rightarrow \cM'$ is smooth at $\bx \in \cM$ if there exists a function $\bar{F}: U \rightarrow \cE'$ which is smooth (in the usual sense) on a neighborhood $u$ of $\bx$ in $\cE$ and such that $\bar{F}$ restricted to $\cM \cap U$ is equal to $F$. The function $\bar{F}$ is called a local smooth extension of $F$ at $\bx$. We say $F$ is smooth if it is smooth for all $\bx \in \cM$.   
\end{definition}

The set of all tangent spaces over the manifold can be bundled together to form another manifold known as the tangent bundle. 

\begin{definition}\label{def:tangent_bundle}[Definition 3.35, \cite{boumal2023introduction}]
The tangent bundle of a manifold $\cM$ is the disjoint union of the tangent spaces of $\cM$: 
\begin{equation}\label{eq:TM}
T\cM := \{(\bx,\bm{v}) \; | \; \bx \in \cM, \; \bm{v} \in T_{\bx}\cM \}.
\end{equation}
\end{definition}

Retraction operations on manifolds, to be defined shortly, and mapping elements of the manifold $\cM$ to tangent vectors naturally take place over the tangent bundle. 
A particularly important class of maps are known as vector fields. 
\begin{definition}\label{def:vector_fields}[Definition 3.37, \cite{boumal2023introduction}]
A vector field on a manifold $\cM$ is a map $V:\cM \rightarrow T\cM$ such that $V(\bx) \in T_{\bx}\cM$ for all $\bx \in \cM$. 
If $V$ is a smooth map, we say it is a smooth vector field.
\end{definition}

In order to define distances and gradients on a manifold an inner product must be introduced over the tangent spaces of the manifold. 

\begin{definition}\label{def:inner_product}[Definition 3.43, \cite{boumal2023introduction}]
An inner product  on $T_{\bx}\cM$ is a bilinear, symmetric, positive definite function $\langle \cdot, \cdot \rangle_{\bx} : T_{\bx}\cM \times T_{\bx}\cM \rightarrow \R$. It induces a norm for tangent vectors: $\|\bm{v}\|_{\bx} := \sqrt{\langle \bm{v},\bm{v}\rangle_{\bx}}$. 
\end{definition}

A metric on $\cM$ follows from a particular choice of inner product $\langle \cdot, \cdot \rangle_{\bx}$ for each $\bx \in \cM$. 
As the inner product over the manifold depends on the point on the manifold where the inner product is taken, an important property for an inner product is that it smoothly varies over the manifold. 
Such inner products are known are Riemannian metrics. 
More precisely,
\begin{definition}\label{def:R_metric}[Definition 3.44, \cite{boumal2023introduction}]
A metric $\langle \cdot, \cdot \rangle_{\bx}$ on $\cM$ is a Riemannain metric if it varies smoothly with $\bx$, in the sense that if $V$, $W$, are two smooth vector fields on $\cM$ then the function $\bx \mapsto \langle V(\bx), W(\bx) \rangle_{\bx}$ is smooth from $\cM$ to $\R$. 
\end{definition}

The presence of a Riemannian metric produces a Riemannian manifold and from this metric we can define a notion of distance between connected components, i.e., points which can be connected by a continuous curve on the manifold, on a Riemannian manifold.
\begin{definition}\label{def:R_distance}
Let $\cM$ be a Riemannian manifold. The length of a piecewise smooth curve, see Definition~10.2, \cite{boumal2023introduction}, $c:[a,b] \rightarrow \cM$ is defined as 
\[
L(c) = \int_{a}^{b} \| c'(t)\|_{c(t)} \; dt. 
\]
The Riemannian distance between $\bx,\by \in\cM$ is then $\text{dist}(\bx,\by) = \inf_c L(c)$ where the infimum is taken over all piecewise regular curve segments on $\cM$ connecting $\bx$ and $\by$.
\end{definition}
An interesting by-product of this notation of distance on a Riemannian manifold is the Riemannian distance in tandem with $\cM$ forms a metric space (Theorem 10.3, \cite{boumal2023introduction}).   
Now,  for the case of submanifolds of Euclidean spaces, that is, vector spaces with a defined inner product, e.g., $\cE = \R^n$ with the standard inner $\langle \bx,\by \rangle = \bx^\top \by$, a particular choice of inner product is of great importance.   

\begin{definition}\label{def:R_manifold}[Definition 3.47, \cite{boumal2023introduction}]
Let $\cM$ be an embedded submanifold of a Euclidean space $\cE$. Equipped with the Riemannian metric obtained by restriction of the metric of $\cE$, we call $\cM$ a Riemannian submanifold of $\cE$.  
\end{definition}

Riemannian submanifolds of Euclidean spaces are of immense practical interest as many applications of optimization over manifolds consists of optimizing over such spaces. 
For example, the sphere $S^{n-1} := \{\bx \in \R^n \; | \; \|\bx\| = 1 \}$, the Stiefel manifold $\text{St}_{n,k} :=  \{ \bX \in  \R^{n\times k} \; | \; \bX^\top \bX = \bm{I}_{k\times k}\}$, and low-rank spectrahedron $\{ \bX \in \R^{n\times n} \; | \; \Tr(\bX) = 1, \bX \succeq \bm{0}, \text{rank}(\bX) \leq r \leq n\}$, and any Cartesian product of these are all Riemannain submanifolds of Euclidean spaces.  

In the case of a function $\bar{F}:\cE \rightarrow \cE'$ which is a smooth map between two vector spaces, the differential of $\bar{F}$ at $\bx \in \cE$ is a linear operator from $\cE$ to $\cE'$, denoted $D\bar{F}(\bx):\cE \rightarrow \cE'$, defined as 
\[
D\bar{F}(\bx)[\bm{v}] = \lim_{t\rightarrow 0} \frac{\bar{F}(\bx + t\bm{v}) - \bar{F}(\bx)}{t}.
\]
From a smooth curve $c$ over $\cM$, a smooth curve over $\cM'$ can be formed as the map $t\mapsto F(c(t))$. From this idea, we can define the differential of a map $F$ between two embedded submanifolds of vector spaces. 

\begin{definition}\label{def:differential}[Definition 3.27, \cite{boumal2023introduction}]
Let $\cM$ and $\cM'$ be two embedded submanifolds of vectors spaces $\cE$ and $\cE'$ respectively.
The differential of $F:\cM \rightarrow \cM'$ at $\bx$ is a linear operator $DF(\bx): T_{\bx}\cM \rightarrow T_{F(\bx)}\cM'$ defined by
\[
DF(\bx)[\bm{v}] = \frac{d}{dt} F(c(t))\bigg|_{t=0},
\]
where $c$ is a smooth curve on $\cM$ passing through $\bx$ at $t=0$ with $c'(0) = \bm{v}$. 
\end{definition}
With the definition of a differential and Riemannian metrics presented, we can now introduce the definition of the Riemannian gradient of a smooth function $f:\cM \rightarrow \R$ over a Riemannian manifold. 
\begin{definition}\label{def:gradf}[Definition 3.50, \cite{boumal2023introduction}]
Let $f:\cM \rightarrow \R$ be a smooth function on a Riemannian manifold $\cM$. The Riemannian gradient of $f$ is the vector field $\grad f: \cM \rightarrow T\cM$ uniquely defined by the identities: 
\[
\forall (\bx,\bm{v}) \in T \cM, \;\;\; Df(\bx)[\bm{v}] = \langle \bm{v}, \grad f(\bx) \rangle_{\bx},
\]
where $Df(\bx)$ is the differential of $f$ as defined above and $\langle \cdot, \cdot \rangle_{\bx}$ is the Riemannian metric.
\end{definition}

In the event $\cM$ is a Riemannian submanifold of a Euclidean space, then computing the Riemannian gradient of $f$ is easily accomplished through a smooth extension of $f$. 

\begin{proposition}\label{prop:gradf}[Proposition 3.53, \cite{boumal2023introduction}]
Let $\cM$ be a Riemannian submanifold of $\cE$ endowed with the metric $\langle \cdot, \cdot \rangle$ and let $f:\cM \rightarrow \R$ be a smooth function. The Riemannian gradient of $f$ is given by 
\begin{equation}\label{eq:R_gradf}
\grad f(\bx) = \text{\rm Proj}_{\bx} \left(\grad \bar{f}(\bx)\right),
\end{equation}
where $\bar{f}$ is any smooth extension of $f$ in a neighborhood of $\bx \in \cM$ and $\text{\rm Proj}_{\bx}(\cdot)$ is the projection operator from $\cE$ to $T_{\bx}\cM$, orthogonal with respect to the inner product over $\cE$.  
\end{proposition}

Another crucial operation over manifolds are retractions.

\begin{definition}\label{def:retraction}[Definition 3.41, \cite{boumal2023introduction}]
A retraction on $\cM$ is a smooth map $\Retr_{\cM}:T\cM \rightarrow \cM$ with the following properties. For each $\bx \in \cM$, let $\Retr_{\cM,\bx}: T_{\bx} \cM \rightarrow \cM$ be the restriction of $\Retr_{\cM}$ at $\bx$, such that $\Retr_{\cM,\bx}(\bm{v}) = \Retr_{\cM}(\bx,\bm{v})$. Then, 
\begin{enumerate}
\item $\Retr_{\cM,\bx}(\bm{0}) = \bx$ for all $\bx \in \cM$, and 
\item $D\Retr_{\cM,\bx}(\bm{0}): T_{\bx}\cM \rightarrow T_{\bx}\cM$ is the identity map: $D\Retr_{\cM,\bx}(\bm{0})[\bm{v}] = \bm{v}$. 
\end{enumerate}
Equivalently, each curve $c(t) = \Retr_{\cM,\bx}(t\bm{v})$ satisfies $c(0)=\bx$ and $c'(0)=\bm{v}$.
\end{definition}

From an optimization perspective, a retraction is a way of returning to the manifold if one has moved away from the manifold in the ambient space. A very special retraction available on each manifold is known as the Riemannian exponential map whose retraction curves are geodesics.

From classical results on differential equations, we know for all $(\bx,\bm{v}) \in T \cM$ there exists a unique maixmal geodesics on $\cM$, $\gamma_{\bm{v}}:I \rightarrow \cM$ such that $\gamma_{\bm{v}}(0) = \bx$ and $\gamma_{\bm{v}}'(0) = \bm{v}$. The Riemannian exponential mapping is a special retraction which aligns with this unique geodesic. 
\begin{definition}\label{def:exp_map}[Definition 10.16,\cite{boumal2023introduction}]
Consider the following subset of the tangent bundle:
\[
\mathcal{O} = \{(\bx,\bm{v}) \in T \cM \; | \; \gamma_{\bm{v}} \text{ is defined on an interval containing $[0,1]$} \}. 
\]
The exponential map $\text{\rm Exp}:\mathcal{O} \rightarrow \cM$ is defined by
$
\text{\rm Exp}(\bx,\bm{v}) = \gamma_{\bm{v}}(1).  
$
\end{definition}
Thus, the exponential map has $\text{Exp}(\bx,0) = \bx$ and $\frac{d}{dt}\left(\text{Exp}(\bx,t\bm{v})\right)|_{t=0} = \bm{v}$. 
The exponential map is smooth over its domain and is a diffeomorphism over its image over a potentially restricted subspace of its domain. 
The injectivity radius of a manifold captures the size of domain on which the expoential map is smoothly invertible.   
\begin{definition}\label{def:complete}[Definitions 10.19 \& 10.28  in \cite{boumal2023introduction}]
The injectivity radius of a Riemannian manifold $\cM$ at a point $\bx$, denoted by $\text{\rm inj}(\bx)$, is the supremum over radii $r>0$ such that $\text{\rm Exp}(\bx,\cdot)$ is defined and is a diffeomorphism on the open ball
\[
B(\bx,r) = \left\{ \bm{v} \in \mathcal{T}_{\bx}\cM  \; | \; \| \bm{v}\|_{\bx} \leq r \right\}.
\]
The injectivity radius $\text{\rm inj}(\cM)$ of a Riemannian manifold $\cM$ is the infimum of $\text{\rm inj}(\bx)$ over $\bx \in \cM$. 
\end{definition}
The injectivity radius shall become important in the analysis of our algorithm as it shall determine a ball on which we can always ensure a tangent vector exists which moves us between points on the manifold $\cM$ which are sufficiently close together. For compact Riemannian submanifolds, e.g., the Stiefel manifold or products of Stiefel manifolds, the injectivity radius over the manifold is strictly positive. 
Another useful property over compact Riemannian submanifolds is a Lipschitz-type condition over any defined retraction. 
\begin{proposition}\label{prop:retraction_lip}
For a retraction $\Retr_{\cM}:T\cM \rightarrow \cM$ on a compact Riemannian submanifold of $\cE$ there exists $L_1 > 0$ such that
\begin{equation}\label{eq:retraction_bdd}
\| \Retr_{\cM}(\bx,\bm{v}) - \bx \|_{\cM} \leq L_1 \|\bm{v}\|_{\bx}
\end{equation}
for all $\bx \in \cM$ and $\bm{v} \in T_{\bx}\cM$, where $\|\cdot\|_{\cM}$ is the induced norm from the inner product of the embedding space $\cE$.
\end{proposition}
This result was proven as an intermediate step of an argument in Appendix B of \cite{boumal2019global} and was concisely stated as Proposition 2.7 in \cite{zhang2018cubic}.
As a note, Proposition~\ref{prop:retraction_lip} holds for any valid retraction, though the constant $L_1$ shall change dependent upon the retraction. 
Lastly, through the use of retractions a natural extension of the traditional gradient Lipschitz condition was proposed by Boumal et al. \cite{boumal2019global}. 

\begin{definition}\label{def:lip_grad_retraction}
The function $f:\cM \rightarrow \R$ has a gradient Lipschitz retraction with constant $L_R \geq 0$ with respect to the retraction $\Retr_{\cM}: T\cM \rightarrow \cM$ if for all $\bx \in \cM$ and $\bm{v} \in T_{\bx}\cM$ we have
\begin{equation}\label{eq:grad_lip_retraction}
f(\Retr_{\cM}(\bx,\bm{v})) \leq f(\bx) + \langle \grad f(\bx), \bm{v} \rangle_{\bx} + \frac{L_R}{2} \|\bm{v}\|_{\bx}^2.  
\end{equation}
\end{definition}
This condition enables many convergence results over $\R^n$ to be extended, relatively easily, to Riemannian manifolds. 
Our analysis makes heavy use of this condition, and, in the case of compact Riemannain submanifolds, Lemma 2.7 of \cite{boumal2019global} shows this condition follows with minor additional assumptions from the classic gradient Lipschitz condition.  

\subsection{Additional Definitions and Notation}
To facilitate our discussion, some additional notation and conventions must be set in order.
For $\bc$ and $\bm{g}$ we denote their component functions as $\bc(\bx,\by) = [c_1(\bx,\by), \hdots, c_p(\bx,\by)]^\top$ and $\bm{g}(\by) = [g_1(\by), \hdots, g_s(\by)]^\top$, where $c_i: \cM \times \R^n \rightarrow \R$ and $g_j:\R^n \rightarrow \R$ for $i=1,\hdots, p$ and $j=1,\hdots,s$. We let $\nabla_y f$ and $\nabla_y c_i$ denote the Euclidean gradients of $f$ and $c_i$ with respect to $\by$. The Jacobian of $\bc$ is denoted  
\[
\nabla_y \bc (\bx, \by) = \bigg[ \nabla_y c_1(\bx,\by) \; | \cdots \; | \; \nabla_y c_p(\bx,\by) \bigg] \in \R^{n \times p}. 
\]
The feasible region of \eqref{eq:main_model} is defined as 
\begin{equation}\label{eq:C}
\cC :=
\left\{ (\bx,\by) \in \cM \times \R^n \; | \; \bc(\bx,\by) = \bm{0},\; \bm{g}(\by) \leq \bm{0}\right\}.
\end{equation}
We further define the following slices of the feasible region. For all $\bx \in \cM$ and $\by \in \R^n$ we define the sets:
\begin{equation}\label{eq:Cy}
\cC_y(\bx) :=
\left\{ \by \in \R^n \; | \; \bc(\bx,\by) = \bm{0}, \; \bm{g}(\by) \leq \bm{0} \right\},
\end{equation}
\begin{equation}\label{eq:Cx}
\cC_x(\by) :=
\left\{ \bx \in \cM \; | \; \bc(\bx,\by) = \bm{0}\right\}.
\end{equation}
The level set of $f$ at the point $(\bar{\bx},\bar{\by}) \in \cC$ is given by
\[
L(\bar{\bx},\bar{\by}) :=
\{ (\bx,\by) \in \cC \; | \; f(\bx,\by) \leq f(\bar{\bx},\bar{\by}) \}.
\]
We define $\cM':= \cM \times \R^n$ and shall often let $\bz$ denote elements of $\cM'$. 
Letting $\bar{z} = (\bar{\bx},\bar{\by})$, we will frequently abbreviate $L(\bar{\bx},\bar{\by})$ as $L(\bar{\bz})$.
The set of active inequality constraints at $(\bx,\by)$ is the set $\cA(\by) := \{ i \;|\; g_i(\by) = 0,\;i=1,\hdots, s\}$.

\subsection{A Constrained Manifold Reformulation}\label{sec:manifold_reform}
To present a clear formulation of optimality conditions for \eqref{eq:main_model}, we recast the model as a nonlinear programming model over a Riemannian manifold. 
Remembering $\cM' = \cM \times \R^n$, we let $i_1: \cM' \rightarrow \cM$ be the projection onto $\cM$, i.e., $i_1(\bx,\by) = \bx$, and $i_2: \cM' \rightarrow \R^n$ be the projection onto $\R^n$, i.e., $i_2(\bx,\by) = \by$. 
Then, with $\bz = (\bx, \by) \in \cM'$, we define the functions 
\[
\barf(\bz) := f(i_1(\bz),i_2(\bz)) = f(\bx,\by),\;\;\; \barc_j(\bz):= c_j(i_1(\bz),i_2(\bz)), \; \text{ and } \; \bar{g}_k(\bz) := g_k(i_1(\bz),i_2(\bz)),
\]
for $j=1,\hdots,p$ and $k=1,\hdots, s$.
We then equivalently rewrite \eqref{eq:main_model} as
\begin{align}\label{eq:manifold_model}
\min&\; \barf(\bz)  \\
\text{s.t.}&\; \bar{\bc}(\bz) = \bm{0},\nonumber \\
&\; \bar{\bm{g}}(\bz) \leq \bm{0}, \nonumber \\
&\; \bz \in \cM'.\nonumber 
\end{align}
This is a nonlinear program over a Riemannian manifold, and programs of this forms have only recently begun to be studied in the literature.
Necessary and sufficient conditions for this problem class were derived in 2014 by Yang~et al.\ \cite{yang2014optimality}. 
Since this paper, other investigations have been undertaken developing theory and algorithms
\cite{bergmann2019intrinsic,liu2020simple,lai2022riemannian,yamakawa2022sequential,schiela2021sqp,obara2022sequential}. 
From these developments, we state the first-order necessary conditions for \eqref{eq:manifold_model}.

\begin{definition}[First-Order KKT Points \cite{yang2014optimality}] \label{def:first_order_necessary_conditions}
The point $\bz^* \in \cM'$ such that $\bar{\bc}(\bz^*)=\bm{0}$ and $\bar{\bm{g}}(\bz^*) \leq \bm{0}$ is a first-order KKT point of \eqref{eq:manifold_model} if there exists $\blambda^* \in \R^p$ and $\bm{\mu}^* \in \R^s_+$ such that 
\begin{align}\label{eqn:FOKKT}
&\grad \barf(\bz^*) + \sum_{i=1}^{p} \lambda_j^* \cdot \grad \barc_i(\bz^*) + \sum_{j=1}^{s}\mu^*_j \cdot \grad \bar{g}_j(\bz^*)  = \bm{0}, \nonumber \\
&\mu_j^*\cdot \bar{g}_j(\bz^*) = 0, \; j=1,\hdots, s. \; 
\end{align}
Similarly, a feasible point $\bz^* \in \cM'$ is an $\epsilon$-approximate first-order KKT point of \eqref{eq:manifold_model} provided there exists $\blambda^* \in \R^p$ and $\bm{\mu}^* \in \R^s_+$ such that 
\begin{align}\label{eqn:FOKKT_approx}
&\bigg\|\grad\barf(\bz^*) + \sum_{i=1}^{p} \lambda_j^* \cdot \grad \barc_i(\bz^*) + \sum_{j=1}^{s}\mu^*_j \cdot \grad\bar{g}_j(\bz^*)\bigg\|_{\bz^*}  \leq  \epsilon,  \\
&\mu_j^* \cdot \bar{g}_j(\bz^*) = 0, \; j=1,\hdots, s.\nonumber  \; 
\end{align}
\end{definition}
Often constraint qualifications are necessary to strength results. 
In this work, we shall assume the linear independence constraint qualification (LICQ) for manifold optimization, which is analogous to the standard property by the same name in Euclidean space. 
\begin{definition}[LICQ \cite{yang2014optimality}]\label{def:LICQ}
The linear independent constraint qualification (LICQ) is said to hold at $\bz \in \cM'$ with $\bar{\bm{c}}(\bz)=\bm{0}$ and $\bar{\bm{g}}(\bz) \leq \bm{0}$ provided $\{\text{grad}\; \bar{c}_i(\bz)\}_{i=1}^{p} \cap \{\text{grad}\;\bar{g}_j(\bz)\}_{j \in \mathcal{A}(\by)}$ forms a  linearly independent set on $\mathcal{T}_{\bz^*}\cM'$. 
\end{definition}

Under the LICQ, we obtain we obtain the fact \eqref{eqn:FOKKT} is a first-order necessary condition.

\begin{proposition}[Theorem 4.1, \cite{yang2014optimality}]\label{prop:first_prop}
If $\bz^*$ is a local minimum of \eqref{eq:manifold_model} and the LICQ property holds at $\bz^*$, then there exists multipliers $\blambda^*\in \R^p$ and $\bm{\mu}^* \in \R^s_+$ such that \eqref{eqn:FOKKT} holds. 
\end{proposition}

In Section~\ref{sec:alg} we present our new algorithm, which from the lens of constrained Riemannian optimization can escape stationary points of \eqref{eq:main_model} and converge to first-order KKT points of the equivalent reformulation \eqref{eq:manifold_model}. 
Since finding exact stationary points and KKT points is in-general an intractable task, we further develop notions of approximate stationarity and KKT points in the upcoming section. 

\section{Approximate Stationary and KKT Points}
From investigating \eqref{eq:main_model} and \eqref{eq:manifold_model}, we recognize there are multiple manners of measuring approximate optimality, 
and our algorithm development takes advantage of these different approaches.
Here we outline different types of approximate solutions and our proposed schemes for measuring them. 

\begin{definition} [Approximate Stationary Points in $\bx$]\label{def:x_fos}
The point $(\bx^*,\by^*)\in \cC$ is a first-order stationary point of \eqref{eq:main_model} with respect to $\bx$ if there exists $\blambda^* \in \R^p$ such that
\begin{equation}\label{eq:x_fos}
\grad f(\bx^*,\by^*) + \sum_{i=1}^{p} \lambda_i^*\cdot \grad c_i(\bx^*,\by^*) = \bm{0}.
\end{equation}
We say $(\bx^*,\by^*) \in \cC$ is an $\epsilon$-approximate first-order stationary point of \eqref{eq:main_model} with respect to $\bx$ if there exists $\blambda^* \in \R^p$ such that
\begin{equation}\label{eq:x_fos_approx}
\bigg\|\grad f(\bx^*,\by^*) + \sum_{i=1}^{p} \lambda_i^*\cdot \grad c_i(\bx^*,\by^*)\bigg\|_{\bx^*} \leq \epsilon.
\end{equation}
\end{definition}

\begin{definition} [Approximate Stationary Points in $\by$]\label{def:y_fos}
The point $(\bx^*,\by^*) \in \cC$ is a first-order stationary point of \eqref{eq:main_model} with respect to $\by$ if there exists $\blambda^* \in \R^p$ and $\bm{\mu}^* \in \R^s_+$ such that
\begin{align}\label{eq:y_fos}
& \nabla_y f(\bx^*,\by^*) + \sum_{i=1}^{p} \lambda_i\cdot \nabla_y c_i(\bx^*,\by^*) + \sum_{j=1}^{s}  \mu_j^*\cdot \nabla g_j(\by^*)= \bm{0}, \\
&\mu_j^* \cdot g_j(\by^*)= 0, \; j=1,\hdots, s. \nonumber  \; 
\end{align}
We say $(\bx^*,\by^*) \in \cC$ is an $\epsilon$-approximate first-order stationary point of \eqref{eq:main_model} with respect to $\by$ provided there exists $\blambda^* \in \R^p$ and $\bm{\mu}^* \in \R^s_+$ such that
\begin{align}\label{eq:y_fos_approx}
&\bigg\|\nabla_y f(\bx^*,\by^*) + \sum_{i=1}^{p} \lambda_i\cdot \nabla_y c_i(\bx^*,\by^*) + \sum_{j=1}^{s}  \mu_j^*\cdot \nabla g_j(\by^*)\bigg\| \leq \epsilon, \\
&\mu_j^* \cdot g_j(\by^*)= 0, \; j=1,\hdots, s. \nonumber  
\end{align}
\end{definition}
Combining these definitions, we obtain our notion of stationary points for \eqref{eq:main_model}.
\begin{definition} [Approximate Stationary Points of \eqref{eq:main_model}]\label{def:fos}
The point $(\bx^*,\by^*) \in \cC$ is a first-order stationary point of \eqref{eq:main_model} if both \eqref{eq:x_fos} and \eqref{eq:y_fos} hold for some set of multipliers; the point $(\bx^*,\by^*) \in \cC$ is an $\epsilon$-approximate first-order stationary point of \eqref{eq:main_model} if both \eqref{eq:x_fos_approx} and \eqref{eq:y_fos_approx} hold. 
\end{definition}
Clearly, it is possible a point could be an approximate stationary point of \eqref{eq:main_model} and fail to be a local minimum. 
So, we consider the stronger constrained Riemannian KKT formulation given in Definition~\ref{def:first_order_necessary_conditions}. 
However, in-order to generate algorithms which provide guarantees, we must be able to measure these conditions; 
we introduce the next result for this purpose. 
\begin{proposition}\label{prop:measure_opt}
Define the following measures: 
\begin{multline}\label{eq:x_measure}
m_{x}(\bx,\by) := \bigg|\; \underset{\bd \in \mathcal{T}_{\bx}\cM}{\min} \{ \langle \grad f(\bx,\by), \bd \rangle_{\bx} \; | \;  \langle \grad\bm{c}_i(\bx,\by), \bd \rangle_{\bx}  =  0, \; i=1,\hdots,p, \; \| \bd \|_{\bx} \leq 1 \}  \bigg|,
\end{multline}
\begin{multline}\label{eq:y_measure}
m_{y}(\bx,\by) :=  \bigg|\; \underset{\bd \in \R^n}{\min} \{ \langle \nabla_y f(\bx,\by), \bd \rangle \; | \;\\ \langle \nabla_y c_i(\bx,\by), \bd \rangle = 0, \; i=1,\hdots,p,\; \langle \nabla_y g_j(\bx,\by), \bd \rangle \leq 0, \; j \in \mathcal{A}(\by),\; \| \bd\|\leq 1 \} \bigg|,   
\end{multline}
\vspace{-0.6in}

\begin{multline}\label{eq:z_measure}
m_{KKT}(\bx,\by) := \bigg|\; \underset{\bd \in \mathcal{T}_{\bz}\cM'}{\min} \{ \langle \grad \barf(\bz), \bd \rangle_{\bz} \; | \;\\  \langle \grad \bar{c}_i(\bz), \bd \rangle_{\bz}  =  0, \; i=1,\hdots,p, \langle \grad \bar{g}_j(\bz), \bd \rangle_{\bz} \leq 0, \; j \in \mathcal{A}(\bz),\; \| \bd \|_{\bz} \leq 1\}  \bigg|.
\end{multline}
Assume the LICQ holds at $(\bx^*,\by^*) \in \cC$. Then, if $\max\left\{m_x(\bx^*,\by^*), m_y(\bx^*,\by^*) \right\} \leq \epsilon$, then $(\bx^*,\by^*)$ is an $\epsilon$-approximate first-order stationary point of \eqref{eq:main_model}, and if $m_{KKT}(\bx^*,\by^*)\leq \epsilon$ then $(\bx^*,\by^*)$ is an $\epsilon$-approximate first-order KKT point of \eqref{eq:main_model}, viewing the problem as the constrained Riemannian optimization problem \eqref{eq:manifold_model}. 
\end{proposition}
\begin{proof}
The proof of this claim is an exercise in duality with the utilization of the LICQ, i.e., Definition~\ref{def:LICQ}. 
Under this assumption, we see from the form of the product manifold $\cM'$ that $\{\text{grad}\; c_i(\bx^*,\by^*)\}_{i=1}^{p}$ forms a linearly independent set over $\mathcal{T}_{\bx^*}\cM$; 
therefore, forming the dual problem defining \eqref{eq:x_measure}, we obtain
\begin{multline}
\underset{\blambda \in \R^p}{\max}\; \underset{\|\bd\|_{\bx^*}\leq 1}{\min}\; \bigg\langle \text{grad}\; f(\bx^*,\by^*)  + \sum_{i=1}^{p}\; \lambda_i\; \text{grad}\; c_i(\bx^*,\by^*), \bd \bigg\rangle_{\bx^*} \\  = \underset{\blambda \in \R^p}{\min}\; \bigg\|  \text{grad}\; f(\bx^*,\by^*)  + \sum_{i=1}^{p}\;\lambda_i\; \text{grad}\; c_i(\bx^*,\by^*) \bigg\|_{\bx^*} = m_x(\bx^*,\by^*), \nonumber 
\end{multline}
where the last equality follows by strong duality. Thus, $m_x(\bx,\by)$ provides an exact measure of the approximate stationarity of $\bx$ per Definition~\ref{def:x_fos}. In a similar fashion, since the LICQ implies the Mangasarian-Fromovitz (MFCQ) over manifolds \cite{bergmann2019intrinsic}, we see the dual problem of the convex program defining \eqref{eq:z_measure} is
\begin{align}
&\underset{\blambda \in \R^p,\; \bm{\mu} \in \R^{|\mathcal{A}(\by^*)|}_+}{\max}\;\; \underset{\|\bd\|_{\bz^*}\leq 1}{\min}\; \bigg\langle \text{grad}\; \bar{f}(\bz^*)  + \sum_{i=1}^{p}\;\lambda_i \; \text{grad}\; \bar{c}_i(\bz^*)  + \sum_{j \in \mathcal{A}(\by^*)} \; \mu_j\; \text{grad}\; \bar{g}_j(\bz^*)   \; , \bd \bigg\rangle_{\bz^*} \nonumber \\  
&= \underset{\blambda \in \R^p,\; \bm{\mu} \in \R^{|\mathcal{A}(\by^*)|}_+}{\min}\; \bigg\| \text{grad}\; \bar{f}(\bz^*)  + \sum_{i=1}^{p}\;\lambda_i \; \text{grad}\; c_i(\bz^*)  + \sum_{j \in \mathcal{A}(\by^*)} \; \mu_j\; \text{grad}\; \bar{g}_j(\bz^*) \bigg\|_{\bz^*}  \nonumber \\
&= m_{KKT}(\bx^*,\by^*), \nonumber 
\end{align}
where as before the last equality follows by strong duality due to the constraint qualification. Note, throughout we let $\bz^*=(\bx^*,\by^*)$. Looking at Definition~\ref{def:first_order_necessary_conditions}, we then see by setting $\mu_j = 0$ for all $j\in \{1,\hdots, s\}\setminus\mathcal{A}(\by^*)$ that $m_{KKT}(\bx^*,\by^*)$ is equivalent to \eqref{eqn:FOKKT_approx}. The argument to prove the equivalence of \eqref{eq:y_measure} and \eqref{eq:y_fos_approx} is similar and we leave it to reader. 
\end{proof}

Proposition~\ref{prop:measure_opt} provides a practical manner of measuring the optimality of a proposed feasible point without directly computing any dual variables.
Now, in order to provide a convergence rate for our proposed method, we require a slightly altered version of the approximate stationary and KKT points. 
This new definition is only required for the $y$-subproblem and the $z$-subproblem due to the inequality constraints and requires a notation of ``almost" active constraints. 

\begin{definition}\label{def:eps_kkt}
Given a set of inequality constraints $g_j(\by) \leq 0$ for $j=1,\hdots, s$, the set of $\delta \geq 0$ almost active constraints is 
$
\mathcal{A}^{\delta}(\bx):= \{ j \in \{1,\hdots, s\} \; | \; g_j(\by) \in [-\delta, 0] \}.
$
Under this definition, we define the following variants of $m_y$ and $m_{\text{KKT}}$: 
\begin{multline}\label{eq:y_delta_measure}
m^{\delta}_{y}(\bx,\by) := \bigg|\; \underset{\bd \in \R^n}{\min} \{ \langle \nabla_y f(\bx,\by), \bd \rangle \; | \;\\ \langle \nabla_y c_i(\bx,\by), \bd \rangle = 0, \; i=1,\hdots,p,\; \langle \nabla_y g_j(\bx,\by), \bd \rangle \leq 0, \; j \in \mathcal{A}^{\delta}(\by),\; \| \bd\|\leq 1 \} \bigg|,   
\end{multline}
\vspace{-0.2in}
\begin{multline}\label{eq:z_delta_measure}
m_{KKT}^{\delta}(\bx,\by):= \bigg|\; \underset{\bd \in \mathcal{T}_{\bz}\cM'}{\min} \{ \langle \grad \barf(\bz), \bd \rangle_{\bz} \; | \;\\  \langle \grad\bar{c}_i(\bz), \bd \rangle_{\bz}  =  0, \; i=1,\hdots,p, \langle \grad \bar{g}_j(\bz), \bd \rangle_{\bz} \leq 0, \; j \in \mathcal{A}^{\delta}(\bz),\; \| \bd \|_{\bz} \leq 1\}  \bigg|.
\end{multline}
We then say a point is an $(\delta,\epsilon)$-approximate stationary point with respect to $\by$ and an $(\delta,\epsilon)$-approximate first-order KKT point if $m_{y}^{\delta}(\bx,\by) \leq \epsilon$ and $m_{KKT}^{\delta}(\bx,\by) \leq \epsilon$ respectively. Note, the original definitions of $m_y$ and $m_{KKT}$ are retrieved if $\delta$ are set to zero.
\end{definition}

There is no fundamental difference between computing \eqref{eq:y_delta_measure} and \eqref{eq:z_delta_measure} as compared to the prior definitions \eqref{eq:y_measure} and \eqref{eq:z_measure}, though it is possible more conditions must be satisfied.
This definition is necessary for us to clearly obtain a convergence rate of our method because this 
definition enables us to determine the stepsize we can take in our algorithm as to not violate our error bound conditions for the inequality constraints. 

\begin{remark}
We shall be computing minimizers of the form represented by \eqref{eq:x_measure}, \eqref{eq:y_measure}, and \eqref{eq:z_measure} in-order to compute directions to decrease the value of the objective function. 
It is crucial to recognize all of these models are convex programs with linear objectives and therefore are tractable subproblems which admit numerous efficient solution procedures maintained in standard optimization solvers. 
\end{remark}

\begin{algorithm}
\caption{{{Staged-BCD for \eqref{eq:main_model} }
}}  
\label{alg:Staged_bcd}
\begin{flushleft} 
  \vspace{-0.1in}
  {\bf Input:}  Initial point $(\bx^0,\by^0)\in \cC$; positive tolerances, $\epsilon_1$, $\epsilon_2$; nonegative tolerances, $\delta_1, \delta_2$; constants, $\alpha, \gamma \in (0,1)$; base step-length, $t_b > 0$ \\
  \vspace{0.05in}
 
  {\bf Steps:}\\ \vspace{0.1in}

{\bf \small 1:} {\bf For:} $k=0,1,\hdots$ {\bf do} \\
{\bf \small 2:} \hspace{0.2in} $t_k = t_b$ \\ \vspace{0.05in}
{\bf \small 2:} \hspace{0.2in} {\bf If} $m^{\delta_1}_y(\bx^k,\by^k) > \epsilon_1$  {\bf do} \\ \vspace{0.05in}
{\bf \small 3:} \hspace{0.5in} Set $\bd_y^k$ to be a minimizer which defines $m^{\delta_1}_y(\bx^k,\by^k)$\\ \vspace{0.05in} 
{\bf \small 4:} \hspace{0.5in} {\bf While } $f(\bx^k, \Pi_{\cC_y}(\by^k + t_k d_y^k;\bxk)) > f(\bx^k,\by^k) - \alpha t_k m^{\delta_1}_y(\bx^k,\by^k)$ {\bf do} \\ \vspace{0.05in}
{\bf \small 5:}\hspace{0.85in} $t_k = \gamma\cdot t_k$  \\ \vspace{0.05in} 
{\bf \small 4:} \hspace{0.5in} {\bf End While } \\ \vspace{0.05in}
{\bf \small 4:} \hspace{0.5in} Set $\bxkp = \bxk$ and $\bykp = \Pi_{\cC_y}(\by^k + t_k d_y^k; \bxk)$.\\ \vspace{0.15in}
{\bf \small 2:} \hspace{0.2in} {\bf ElseIf} $m_x(\bx^k,\by^k) > \epsilon_1$  {\bf do} \\ \vspace{0.05in}
{\bf \small 3:} \hspace{0.5in} Set $\bd_x^k$ to be a minimizer which defines $m_x(\bx^k,\by^k)$\\ \vspace{0.05in} 
{\bf \small 4:} \hspace{0.5in} {\bf While } $f(\Pi_{\cC_x}(\text{Retr}_{\cM}(\bx^k,t_k\bd_x^k);\byk),\byk) > f(\bx^k,\by^k) - \alpha t_k m_x(\bx^k,\by^k)$ {\bf do} \\ \vspace{0.05in}
{\bf \small 5:}\hspace{0.85in} $t_k = \gamma\cdot t_k$  \\ \vspace{0.05in} 
{\bf \small 4:} \hspace{0.5in} {\bf End While } \\ \vspace{0.05in}
{\bf \small 4:} \hspace{0.5in} Set $\bxkp = \Pi_{\cC_x}(\text{Retr}_{\cM}(\bx^k,t_k\bd_x^k);\byk)$ and $\bykp = \by^k$.\\ \vspace{0.15in}
{\bf \small 7:} \hspace{0.2in}{\bf ElseIf} $m_{KKT}^{\delta_2}(\bxk,\byk) > \epsilon_2$ {\bf do} \\ \vspace{0.05in}
{\bf \small 3:} \hspace{0.5in} Set $\bd^k$ to be a minimizer which defines $m_{KKT}^{\delta_2}(\bx^k,\by^k)$\\ \vspace{0.05in} 
{\bf \small 9:} \hspace{0.5in} {\bf While } $f(\Pi_{\cC}(\text{Retr}_{\cM'}(\bz^{k},t_k\bd^{k})) >  f(\bxk,\byk)  -\alpha t_k m_{KKT}^{\delta_2}(\bx^{k},\by^{k}) $ {\bf do} \\ 
{\bf \small 10:} \hspace{0.85in} $t_k = \gamma\cdot t_k$ \\ 
{\bf \small 11:} \hspace{0.5in} {\bf End While} \\ \vspace{0.05in}
{\bf \small 12:} \hspace{0.5in} $(\bx^{k+1}, \by^{k+1}) = \Pi_{\cC}(\Retr_{\cM'}(\bz^{k},t_k\bd^{k}))$ \\ 
{\bf \small 13:} \hspace{0.2in} {\bf Else} \\
{\bf \small 14:} \hspace{0.5in} Return $(\bxk, \byk)$ as solution. \\
{\bf \small 15:} \hspace{0.2in} {\bf End If} \\
{\bf \small 16:} {\bf End For}
  \end{flushleft} 
  \end{algorithm}

\section{Algorithm}\label{sec:alg}
We now present our algorithm for solving the block manifold optimization model \eqref{eq:main_model} and describe how it proceeds to compute approximate first-order KKT points reformulated as constrained Riemannian optimization model \eqref{eq:manifold_model}. 
Our procedure presented in Algorithm~\ref{alg:Staged_bcd} has three different updating phases: a $\by$-update, $\bx$-update, and a joint-update. 

The first updating phase in the algorithm is the $\by$-updating phase. 
We let $\bxk \in \cM$ be fixed and seek to improve the value of the objective function with respect to $\by$ by computing a descent direction and then projecting back onto the feasible region. 
The projection in this step is defined as
\begin{equation}\label{eq:y_projection}
\Pi_{\cC_y}(\by;\bx) \in  \argmin_{\bar{\by} \in \cC_{\by}(\bx)} \; \left\{ \| \bar{\by} - \by\|  \right\}.
\end{equation}
There are a few reasons we do the $y$-update first. In the event $\{\by \in \R^n \;| \; \bc(\bx,\by) = \bm{0}, \; \bm{g}(\by) \leq \bm{0}\}$ is a convex set for any fixed $\bx \in \cM$, then the projection operation is relatively simple because it becomes the projection onto a convex subset of $\R^n$. 
We shall see important examples, e.g., QCQPs, where this occurs. 
Second, there is no explicit reference of a manifold in this projection, which reduces the computational difficulty in most cases. 
Third, given our applications of interest, we have $\cM$ as either the Stiefel manifold or a product of Stiefel manifolds; thus, the dimension of $\cM$ is larger than $\R^n$. 
So, the memory cost of operating over $\by$ can be significantly less than over $\bx$. 
Finally, given the form of Algorithm~\ref{alg:Staged_bcd}, placing the $y$-update first means it will likely be undertaken the most, further amplifying the cost savings. 

The second phase in the algorithm is the $\bx$-updating phase. This phase is similar to the $\by$-update in the sense that $\byk$ is fixed and we seek to decrease $f$ by updating $\bx$ through the computation of a descent direction and performing line-search in that direction with projections back onto the feasible region. The projection in this phase is defined as
\begin{equation}\label{eq:x_projection}
\Pi_{\cC_x}(\bx;\by) \in  \argmin_{\bar{\bx} \in \cC_{\bx}(\by)} \; \left\{ \| \bar{\bx} - \bx \|_{\cM}\right\}.
\end{equation}
This phase shall only be entered if $m_y^*(\bxk,\byk)\leq \epsilon_1$, so it is unlikely to be as frequently visited as the $\by$-update step. 
A discussion about this projection operation shall be held in the coming sections.

The third and final phase of the algorithm updates $\bx$ and $\by$ jointly. 
This phase is only entered when it appears progress cannot be made by further updating only $\bx$ and $\by$ sequentially. 
Similar to the prior two phases, a joint descent direction is computed by treating the original problem as a constrained Riemannian optimization problem over the product manifold $\cM'$. 
In this phase, a point $\bz \in \cM'$ is projected onto both the functional constraints and the manifold and is defined as 
\begin{equation}\label{eq:z_projection}
\Pi_{\cC}(\bx) \in  \argmin_{\bar{\bz} \in \cC} \; \left\{ \|\bar{\bz} - \bz \|_{\cM'}\right\}.
\end{equation}
This projection operations shall be discussed as well in the forthcoming sections.

\subsection{A Feasible Constrained Riemannian Optimization Method}\label{sec:projections}
By design, Algorithm~\ref{alg:Staged_bcd} is a feasible method, i.e., every iterate generated by the procedure satisfies the original constraints.
Other methods such as penalty, augmented Lagrangian, sequential quadratic programming, primal-dual algorithms, and interior-point methods for constrained Riemannian optimization cannot return a feasible solution until the algorithm converges to a solution \cite{zhang2020primal,liu2020simple,lai2022riemannian,yamakawa2022sequential,schiela2021sqp,obara2022sequential,andreani2024}; however, this might never occur or could potentially take a prohibitively long period of time. 
Our algorithm is one of the few feasible constrained Riemannian optimization methods. 
Recently, Weber and Sra \cite{weber2023} developed a Frank-Wolfe method for a constrained Riemannian optimization model, but the constraint set they considered requires the strong conditions of being compact and geodesically convex.
So, our algorithm, as far as we are aware, presents the first feasible constrained Riemannian optimization method for problems without geodesic convexity. 

Developing such an approach comes with natural pros and cons. 
On the positive side, our staged block descent method can be terminated at any time with a feasible solution in-hand. 
%
%
On the negative side, in-order for our method to maintain feasibility, projections onto $\cC_{\bx}$ and $\cC$ must be computed, which could potentially prove difficult. 
The next sections describe some potential methods to approximate optimal projections onto $\cC_{\bx}$ and $\cC$ with provable guarantees for our applications of interest.  

\subsection{Approximate Alternating Projection Approach}
The method of alternating projections has a long history and more recently investigations have been made into dealing with alternating projections onto manifolds \cite{andersson2013alternating,lewis2008alternating,song2020fast}.
Our algorithm asks for the projections onto the  sets $\cC_x$ and $\cC$, which can be highly nonconvex; this could pose a problem. Recently, however, Drusvyatskiy and Lewis \cite{drusvyatskiy2018inexact} proposed an approximate alternating projection method for finding feasible points of the constraint set
\begin{equation}\label{eq:Lewis_setting}
\begin{cases}
G(\bx) \leq \bm{0} \\
H(\bx) = \bm{0} \\
\bx \in \mathcal{Q},
\end{cases}
\end{equation}
where $G$ and $H$ are $C^2$-smooth and $\mathcal{Q}$ is a closed subset of Euclidean space.
In the case of attempting to project onto $\cC_x$ or $\cC$, our constraint sets of interest are of the same form as \eqref{eq:Lewis_setting}, with $\mathcal{Q}$ either the Stiefel manifold, a product of Stiefel manifolds, or a product of Stiefel manifolds and $\R^n$, which are all closed subsets of Euclidean spaces. 
Thus, computing unique and exact projections onto $\mathcal{Q}$ for points near $\mathcal{Q}$ are possible; see Proposition 7 in \cite{absil2012projection} and Appendix A in \cite{balashov2022error} for a full description of optimal and unique projections onto the Stiefel manifold.
This then implies optimal and unique projections onto our product manifolds of interest. 
Thus, by \cite{drusvyatskiy2018inexact}, we have the following result: 
\begin{theorem}\label{thm:lewis_projection}[Theorem 1, \cite{drusvyatskiy2018inexact}]
If $\mathcal{Q} = \text{St}_{n,k}, \; \otimes_{i=1}^{r} \text{St}_{n_i,k_i}, \; \text{St}_{n,k}\times \R^n$ or $ \otimes_{i=1}^{r} \text{St}_{n_i,k_i} \times \R^m$ intersects a closed set $M$ separably (see Definition 2, \cite{drusvyatskiy2018inexact}) at a point $\bar{\bx}$ and $\Phi$ is an inexact projection onto $M$ (see Definition 3, \cite{drusvyatskiy2018inexact}) around $\bar{\bx}$, then starting from a point nearby $\bm{z} \in \mathcal{Q}$, the inexact alternating projection scheme 
\begin{equation}\label{eq:alg_proj}
\bm{z} \leftarrow P_{\mathcal{Q}}\left( \Phi(\bm{z})\right)
\end{equation}
converges linearly to a point in the intersection $M \cap \mathcal{Q}$.
\end{theorem}
Furthermore, the authors in \cite{drusvyatskiy2018inexact} propose an inexact projection method onto constraint sets defined by equality and inequality constraints which satisfy the conditions of Theorem~\ref{thm:lewis_projection} under the LICQ. 
Thus, we can approximate the exact projections in the $\bx$-updating phase and the joint-updating phase in Algorithm~\ref{alg:Staged_bcd} with \eqref{eq:alg_proj}. 
A potential critique is that the computed point in the intersection might not be the optimal projection onto the feasible region. 
This is possible; however, for the implementation of our algorithm to be successful only a sufficient improvement of the objective function is necessary and this is certainly obtainable with almost optimal projections.
In Section~\ref{sec:experiments}, we shall see that we are able to compute global minimums to nonconvex problems using Algorithm~\ref{alg:Staged_bcd} with inexact projections. 

\subsection{A Simple Penalty Approach to Approximate Projections}
Another method for approximating $\Pi_{\cC_x}$ and $\Pi_{\cC}$ is via a localized penalty approach. 
For example, let $(\hat{\bx},\hat{\by}) \in \cC$ and  $\bd \in \mathcal{E}$, with $\cE$ the Euclidean space in which $\cM$ is embedded. 
Then, we can approximate the projection of $(\Retr_{\cM}(\hat{\bx},t \bd),\hat{\by})$ onto $\cC_x(\hat{\by})$ by solving the manifold optimization model 
\begin{align}\label{eq:simple_projection}
\min&\;{P}(c(\bx,\hat{\by})) \\
\text{s.t.}&\; \bx \in \cM, \nonumber
\end{align}
by initializing procedures to solve \eqref{eq:simple_projection} at $\Retr_{\cM}(\hat{\bx},t \bd)$, where ${P}:\R^p \rightarrow [0,\infty)$ is a penalty function such that ${P}(\bm{0}) = 0$.
Ideally, initializing algorithms to solve \eqref{eq:simple_projection} at points which are nearly feasible, e.g., $(\Retr_{\cM}(\hat{\bx},t \bd),\hat{\by})$, shall not only compute feasible points but points which are nearly optimal projections.
In testing, we see for small values of $t$, feasible solutions can readily be computed which are nearer to $(\Retr_{\cM}(\hat{\bx},t \bd),\hat{\by})$ 
than the original feasible point $(\hat{\bx}, \hat{\by})$.
This does not guarantee optimal projections but this is sufficient to ensure feasibility and general improvement of the objective function. 

\section{Convergence Analysis}\label{sec:analysis}
In this section, we prove under reasonable assumptions Algorithm~\ref{alg:Staged_bcd} converges to an $(\epsilon,\epsilon)$-approximate first-order KKT point, as defined in Definition~\ref{def:eps_kkt}, in $\mathcal{O}(1/\epsilon^2)$ iterations. 
We begin our analysis by defining several constants which depend upon the initialization of Algorithm~\ref{alg:Staged_bcd}. 
For all $(\bx^0,\by^0) \in \cC$ we define the following terms: 
\begin{equation}\label{eq:Gfynew}
\mathcal{G}_{f,y}(\bx^0,\by^0) := \sup \left\{ \| \nabla_y f(\bx, \by + \bd) \| \; | \; (\bx,\by) \in L(\bx^0,\by^0), \; \| \bd \| \leq 1 \right\},
\end{equation}
\begin{equation}\label{eq:Ggy}
\mathcal{G}_{g_i,y}(\bx^0,\by^0) := \sup \left\{ \| \nabla_y g_i(\by + \bd) \| \; | \; (\bx,\by) \in L(\bx^0,\by^0), \; \| \bd \| \leq 1 \right\},
\end{equation}
\begin{equation}\label{eq:Ggmax}
\mathcal{G}_{g,y}(\bx^0,\by^0) :=  \max_{1 \leq i \leq s} \mathcal{G}_{g_i,y}(\bx^0,\by^0), 
\end{equation}
\begin{equation}\label{eq:Hgy}
\mathcal{H}_{g_i,y}(\bx^0,\by^0) := \sup \left\{ \| \nabla_y^2 g_i(\by + \bd) \|_2 \; | \; (\bx,\by) \in L(\bx^0,\by^0), \; \| \bd \| \leq 1 \right\},
\end{equation}
\begin{equation}\label{eq:Hgmax}
\mathcal{H}_{g,y}(\bx^0,\by^0) := \max_{1 \leq i \leq s} \mathcal{H}_{g_i,y}(\bx^0,\by^0), 
\end{equation}
\begin{multline}\label{eq:Gfxnew}
\mathcal{G}_{f,x}(\bx^0,\by^0) := \sup\bigg\{ \| \grad f(\Retr_{\cM}(\bx,\bd),\by) \|_{\Retr_{\cM}(\bx,\bd)} \; \\ \bigg| \; (\bx,\by) \in L(\bx^0,\by^0), \; \|\bd\|_{\bx} \leq 1, \; \bd \in T_{\bx}\cM \bigg\},
\end{multline}
\begin{equation}\label{eq:Gfbarx}
\mathcal{G}_{\bar{f},z}(\bz^0) := \sup\left\{ \| \grad \bar{f}(\Retr_{\cM'}(\bz,\bd)) \|_{\bz} \; | \; \bz \in L(\bz^0), \; \|\bd\|_{\bz} \leq 1, \; \bd \in T_{\bz}\cM' \right\},
\end{equation}
\begin{multline}\label{eq:Gcix}
\mathcal{G}_{c_i,x}(\bx^0,\by^0) := \sup \bigg\{ \text{abs}\left(\frac{d^2}{dt^2}\left(c_i(\Retr_{\cM}(\bx, t\bd),\by)\right)\right)\big|_{t=a} \; \\  \bigg| \; a \in [0,1],\; (\bx,\by) \in L(\bx^0,\by^0),\; \|\bd\|_{\bx} \leq 1, \; \bd \in T_{\bx}\cM \bigg\},
\end{multline}
\begin{equation}\label{eq:Gc}
\mathcal{G}_{c,x}(\bx^0,\by^0) := \sqrt{\mathcal{G}_{c_1,x}(\bx^0,\by^0)^2 + \cdots + \mathcal{G}_{c_p,x}(\bx^0,\by^0)^2}, 
\end{equation}

\begin{multline}\label{eq:Gbarcix}
\mathcal{G}_{\bar{c}_i,z}(\bx^0,\by^0) :=  \sup \bigg\{ \text{abs}\left(\frac{d^2}{dt^2}\left(\bar{c}_i(\Retr_{\cM'}(\bz, t\bd)\right)\right)\big|_{t=a} \; \\  \bigg| \; a \in [0,1],\; \bz = (\bx,\by) \in L(\bx^0,\by^0),\; \|\bd\|_{\bz} \leq 1, \; \bd \in T_{\bz}\cM' \bigg\},
\end{multline}
\begin{equation}\label{eq:Gbarc}
\mathcal{G}_{\bar{c},z}(\bx^0,\by^0) := \sqrt{\mathcal{G}_{\bar{c}_1,x}(\bx^0,\by^0)^2 + \cdots + \mathcal{G}_{\bar{c}_p,x}(\bx^0,\by^0)^2},
\end{equation}

\begin{equation}\label{eq:Hci}
\mathcal{H}_{c_i,y}(\bx^0,\by^0) :=  \sup\left\{ \| \nabla_y^2 c_i(\bx, \by+\bd)\|_2 \; | \; (\bx, \by) \in L(\bx^0,\by^0),\; \|\bd\| \leq 1 \right\},
\end{equation}
\begin{equation}\label{eq:Hc}
\mathcal{H}_{c,y}(\bx^0,\by^0) :=  \frac{1}{2}\sqrt{ \mathcal{H}_{c_1,y}(\bx^0,\by^0)^2 + \cdots + \mathcal{H}_{c_p,y}(\bx^0,\by^0)^2}.
\end{equation}
Under the assumptions we shall impose, it can be simply argued the defined constants above are all positive and finite. 
Each of these constants focuses on being able to bound the first and second-order derivatives of the objective and constraint functions. 
This is crucial in-order to produce a convergence rate for our algorithm. 
Note, we shall almost always drop the dependence on $(\bx^0,\by^0)$ when we use these constants in the remainder of the paper to reduce notational clutter.  

We now present the assumptions we utilize in our analysis. 
The assumptions can be divided into four categories: smoothness assumptions on the objective and functional constraints, error-bound conditions on the constraints, constraint qualifications, and manifold conditions.  

\begin{assumption}\label{a:smooth1}
The function $f:\cM \times \R^n \rightarrow \R$ is two-times continuously differentiable and $\bc:\cM \times \R^n \rightarrow \R^p $ and $\bm{g}:\cM \times \R^n \rightarrow \R^s$ are three-times continuously differentiable.
\end{assumption}

\begin{assumption}\label{a:smooth2}
The functions $f:\cM \times \R^n \rightarrow \R$ and $\bc:\cM\times \R^n \rightarrow \R^p$ satisfy the following Lipschitz conditions: 
\vspace{-0.15in}
\begin{enumerate}
\item There exists $L_y \geq 0$ such that for all $(\bx,\by^1)$ and $(\bx,\by^2)$ in $\cM \times \R^n$
\begin{equation}\label{eq:f_grad_y}
\| \nabla_y f(\bx,\by^1) - \nabla_y f(\bx,\by^2) \| \leq L_y \| \by^1 - \by^2\|.
\end{equation}
\item $f(\cdot,\by)$ has a gradient Lipschitz retraction with respect to both the exponential map over $\cM$ and the retraction operator utilized in Algorithm~\ref{alg:Staged_bcd} with constant $L_x \geq 1$ for all $\by \in \R^n$.
\item There exists $L_{xy} \geq 0$ such that for all $(\bx^1,\by)$ and $(\bx^2,\by)$ in $\cM \times \R^n$ 
\begin{equation}\label{eq:f_grad_xy}
\| \nabla_y f(\bx^1,\by) - \nabla_y f(\bx^2,\by) \| \leq L_{xy} \| \bx^1 - \bx^2\|_{\cM}.
\end{equation}
\item There exists $L_c \geq 0$ such that for all $(\bx^1,\by)$ and $(\bx^2,\by)$ in $\cM \times \R^n$ we have 
\begin{equation}\label{eq:c_lipgrad}
\| \nabla_y \bc(\bx^1,\by) - \nabla_y \bc(\bx^2, \by) \|_2 \leq L_{c} \| \bx^1 - \bx^2 \|_{\cM}. 
\end{equation}
\end{enumerate}
\end{assumption}

\begin{assumption}\label{a:LICQ}
The LICQ holds for all iterates generated by Algorithm~\ref{alg:Staged_bcd} in the Euclidean and Riemannian senses for $\bc$ and $\bm{g}$ and $\bar{\bc}$ and $\bar{\bm{g}}$ respectively. 
\end{assumption}

\begin{assumption}\label{a:error_bound_conditions}
There exists constants $\delta_x, \delta_y$, and $\delta_z > 0$ and $\tau_x, \tau_y$, and $\tau_z > 0$ such that
\vspace{-0.15in}
\begin{enumerate}
\item For all $\by \in \R^n$ such that $\text{\rm dist}(\by, \cC_y(\bx)) \leq \delta_y$, independent of $\bx\in \cM$, it follows 
\begin{equation}\label{eq:error_y}
\text{\rm dist}(\by, \cC_y(\bx)) \leq \tau_y \left(\| \bc(\bx,\by)\| + \max(\bm{g}(\by),0) \right),
\end{equation}
where $\text{\rm dist}(\by, \cC_y(\bx)):= \underset{ \hat{\by} \in \cC_y(\bx)}{\inf} \| \hat{\by} - \by\|$ and $\max(\bm{g}(\by),0):= \max(g_1(\by), \hdots, g_s(\by),0)$.

\item For all $\bx \in \cM$ such that $\text{\rm dist}(\bx, \cC_x(\by)) \leq \delta_x$, independent of $\by\in \R^n$, it follows 
\begin{equation}\label{eq:error_x}
\text{\rm dist}(\bx, \cC_x(\by)) \leq \tau_x \| \bc(\bx,\by)\|, 
\end{equation}
where $\text{\rm dist}(\bx, \cC_x(\by)) := \inf \left\{ \| \bar{\bx} - \bx \|_{\cM} \; | \; \bar{\bx} \in \cC_x(\by)\right\}$.

{\color{black}
\item For all $\bz = (\bx,\by) \in \cM \times \R^n$ such that $\text{\rm dist}((\bx,\by), \cC) \leq \delta_z$, we have that
\begin{equation}\label{eq:error_z}
\text{\rm dist}((\bx,\by), \cC) \leq \tau_z\left(\| \bc(\bx,\by)\| + \max(\bm{g}(\by),0) \right),
\end{equation}
}
where $\text{\rm dist}((\bx,\by), \cC) := \inf \left\{ \| \bar{\bx} - \bx \|_{\cM} + \| \bar{\by} - \by\| \; | \; (\bar{\bx}, \bar{\by}) \in \cC\right\}$.
\end{enumerate}
\end{assumption}

\begin{assumption}\label{a:manifold}
The injectivity radius of $\cM$ is strictly positive and Proposition~\ref{prop:retraction_lip} holds for the exponential mapping over $\cM$ and the retraction operation utilized in Algorithm~\ref{alg:Staged_bcd}.
\end{assumption}

\begin{assumption}\label{a:norm_equivalence}
The Riemannian distance and the norm inherited from the Euclidean space in which $\cM$ is embedded are locally equivalent, i.e., there exist constants $\rho > 0$ and $\Gamma \geq 1$ such that for all $\bx, \by \in \cM$, if $\| \bx - \by \|_{\cM} \leq \rho$, then $\| \bx - \by \|_{\cM} \leq \text{\rm dist}(\bx, \by) \leq \Gamma \| \bx - \by \|_{\cM}$.
\end{assumption}

It is worth noting that these assumptions are fairly standard. The first assumption asks for sufficient differentiability and is readily satisfied by all of the motivating examples. 
The Lipschitz conditions in the second assumption can follow from a simple Lipschitz condition on the original problem before decomposing; see Appendix~\ref{sec:appendix_lip}.
Constraint qualifications are a staple in the analysis of optimization algorithms, and error bound conditions are a common occurrence in constrained problems.
Finally, Assumption~\ref{a:manifold} holds for any compact Riemannian submanifold, for example, the Stiefel manifold, and Assumption~\ref{a:norm_equivalence} is expected due to the locally Euclidean nature of manifolds.

\subsection{Y-Updating Phase}
\begin{lemma}\label{lem:y_improvement}
Let $\delta_1 > 0$. 
If $m_y^{\delta_1}(\bxk,\byk) > \epsilon_1$, then the y-updating phase is entered and shall terminate with
\begin{equation}\label{eq:y_improve}
f(\bxkp,\bykp) \leq f(\bxk,\byk) - t_{\text{low},y}\cdot \alpha_1 \cdot \epsilon_1
\end{equation}
where 
\begin{equation}\label{eq:ty_low}
t_{\text{low},y} :=  \gamma \min \left\{ 1, \delta_y, \frac{\delta_1}{\mathcal{G}_{g,y}}, \frac{2(1-\alpha_1)\epsilon_1}{{L_y} + 2\tau_y \mathcal{G}_{f,y}(\mathcal{H}_{c,y}+\frac{1}{2}\mathcal{H}_{g,y}) + L_y\tau_y^2(\mathcal{H}_{c,y}+\frac{1}{2}\mathcal{H}_{g,y})}\right\}. 
\end{equation}
\end{lemma}
\begin{proof}
Let $(\bxk,\byk) \in \cC$ and assume $m_y^{\delta_1}(\bxk,\byk) > \epsilon_1$. Let $\bd_y^k$ be a minimizer of the subproblem defining $m^{\delta_1}_y(\bxk,\byk)$. Then by the gradient Lipschitz condition on $f$ with respect to $\by$, 
\begin{align}\label{eq:y_help1}
f(\bxk,\byk + t \bd_y^k) &\leq f(\bxk,\byk) + t \langle \nabla_y f(\bxk,\byk), \bd_y^k\rangle + \frac{L_y}{2} t^2 \nonumber \\
&= f(\bxk,\byk) - t m^{\delta_1}_y(\bxk,\byk) +  \frac{L_y}{2} t^2. 
\end{align}
From the definition of $\bd_y^k$ and the feasibility of $(\bxk,\byk)$ it follows 
\[
\bc(\bxk,\byk+t\bd_y^k) = 
\frac{t^2}{2} \begin{pmatrix} \langle \nabla_y^2 c_1(\bxk,\bar{\by})[\bd_y^k], \bd_y^k \rangle  \\ \vdots \\ \langle \nabla_y^2 c_p(\bxk,\bar{\by})[\bd_y^k], \bd_y^k \rangle  \end{pmatrix}
\]
where  $\bar{\by}$ lies between $\byk$ and $\byk + t \bd_y^k$. 
Thus, 
\begin{align}
\| \bc ( \bxk, \byk + t \bd_y^k) \| &\leq \frac{t^2}{2} \left[ \|\nabla_y^2 c_1 (\bxk,\bar{\by}) \|_2^2 + \cdots 
+ \|\nabla_y^2 c_p (\bxk,\bar{\by}) \|_2^2\right]^{1/2} \nonumber \\
&\overset{\eqref{eq:Hci}}{\leq} \frac{t^2}{2} \left[\mathcal{H}_{c_1,y}^2 + \cdots + \mathcal{H}_{c_p,y}^2 \right]^{1/2} \nonumber \\
&\overset{\eqref{eq:Hc}}{=} t^2 \mathcal{H}_{c,y} . \nonumber 
\end{align}

{\color{black}
We must similarly bound the potential violation of the inequality constraints. 
We begin by dealing with those constraints $i$ such that $g_i(\byk) < -\delta_1$. 
Therefore, for some $\tilde{\by}$ between $\byk$ and $\byk + t \bd_y^k$ 
\[
g_i(\byk + t \bd_y^k) = g_i(\byk) + t \langle \nabla_y g_i(\tilde{\by}),\bd_y^k \rangle  < -\delta_1 + t \mathcal{G}_{g_i,y}.
\]
Thus, if $t \leq \delta_1/\mathcal{G}_{g_i,y}$, then we know the $i$-th constraint shall not be violated. Next, we bound the error of the $\delta_1$-active constraints.  
} Expanding $g_i$ for $i\in \mathcal{A}^{\delta_1}(\byk)$ we have
\begin{align}\label{eq:gy_bound}
g_i(\byk+t\bd_y^k) &= g_i(\byk) + t\langle \nabla_y g_i(\byk), \bd_y^k \rangle + \frac{t^2}{2} \langle \nabla_y^2 g_i(\tilde{\by})[\bd_y^k], \bd_y^k \rangle  \nonumber  \\
&\leq  \frac{t^2}{2} \langle \nabla_y^2 g_i(\tilde{\by})[\bd_y^k], \bd_y^k\rangle \nonumber \\
&\overset{\eqref{eq:Hgy}}{\leq} \frac{t^2}{2} \mathcal{H}_{g_i,y}.  
\end{align}
So, assuming $t \leq \delta_y$ and $t\leq \delta_1/\mathcal{G}_{g,y}$, it then follows by the error bound condition, \eqref{eq:error_y}, that
\begin{align}\label{eq:y_help2}
\text{dist}(\byk + t \bd_y^k, \cC_y(\bxk)) &\leq \tau_y \left(\| \bc(\bxk,\byk + t \bd_y^k) \| + \max\big(\bm{g}(\byk+t \bd_y^k),\bm{0}\big) \right)  \nonumber \\
&\leq \tau_y\left( \mathcal{H}_{c,y} + \frac{1}{2}\mathcal{H}_{g,y}\right) t^2. 
\end{align}
Hence, it follows 
\begin{align}
&f(\bxk, \Pi_{\cC_y(\bxk)}(\byk + t \bd_y^k))  \nonumber \\
&\leq f(\bxk, \byk + t \bd_y^k) + \langle \nabla_y f(\bxk,\byk + t \bd_y^k), \Pi_{\cC_y(\bxk)}(\byk + t \bd_y^k) - (\byk + t \bd_y^k) \rangle + \frac{L_y}{2} \text{dist}(\byk + t \bd_y^k, \cC_y(\bxk))^2 \nonumber \\
&\overset{\eqref{eq:y_help1}}{\leq} f(\bxk,\byk) - t m^{\delta_1}_y(\bxk,\byk) +  \frac{L_y}{2} t^2 + \langle \nabla_y f(\bxk,\byk + t \bd_y^k), \Pi_{\cC_y(\bxk)}(\byk + t \bd_y^k) - (\byk + t \bd_y^k) \rangle \nonumber \\
&\hspace{2in}+ \frac{L_y}{2} \text{dist}(\byk + t \bd_y^k, \cC_y(\bxk))^2 \nonumber \\
&\overset{\eqref{eq:y_help2}}{\leq} f(\bxk,\byk) - t m^{\delta_1}_y(\bxk,\byk) +  \frac{L_y}{2} t^2 + \langle \nabla_y f(\bxk,\byk + t \bd_y^k), \Pi_{\cC_y(\bxk)}(\byk + t \bd_y^k) - (\byk + t \bd_y^k) \rangle \nonumber \\
&\hspace{2in}+ \frac{1}{2}L_y \tau_y^2(\mathcal{H}_{c,y}+\frac{1}{2}\mathcal{H}_{g,y})^2 t^4 \nonumber \\ 
&\overset{\eqref{eq:Gfynew}}{\leq} f(\bxk,\byk) - t m^{\delta_1}_y(\bxk,\byk) +  \frac{L_y}{2} t^2 + \mathcal{G}_{f,y}\text{dist}(\byk + t \bd_y^k, \cC_y(\bxk)) + \frac{1}{2}L_y \tau_y^2(\mathcal{H}_{c,y}+\frac{1}{2}\mathcal{H}_{g,y})^2 t^4 \nonumber \\ 
&\overset{\eqref{eq:y_help2}}{\leq} f(\bxk,\byk) - t m^{\delta_1}_y(\bxk,\byk) +  \frac{L_y}{2} t^2 + \mathcal{G}_{f,y}\tau_y(\mathcal{H}_{c,y}+\frac{1}{2}\mathcal{H}_{g,y})t^2 + \frac{1}{2}L_y \tau_y^2(\mathcal{H}_{c,y}+\frac{1}{2}\mathcal{H}_{g,y})^2 t^4. \nonumber 
\end{align}
Thus, assuming $t\leq 1$ it follows 
\begin{align}\label{eq:y_help3}
&f(\bxk, \Pi_{\cC_y(\bxk)}(\byk + t \bd_y^k)) \nonumber \\ 
&\hspace{0.1in}\leq f(\bxk,\byk) - m^{\delta_1}_y(\bxk,\byk) t + \left( \frac{L_y}{2} + \tau_y \mathcal{G}_{f,y}(\mathcal{H}_{c,y}+\frac{1}{2}\mathcal{H}_{g,y}) + \frac{1}{2}L_y\tau_y^2(\mathcal{H}_{c,y}+\frac{1}{2}\mathcal{H}_{g,y})^2 \right)t^2.
\end{align}
Therefore, a sufficient decrease shall be obtained provided 
\[
t \leq \min \left\{ 1, \delta_y, \frac{\delta_1}{\mathcal{G}_{g,y}}, \frac{(1-\alpha_1) m^{\delta_1}_y(\bxk,\byk)}{\frac{L_y}{2} + \tau_y \mathcal{G}_{f,y}(\mathcal{H}_{c,y}+\frac{1}{2}\mathcal{H}_{g,y}) + \frac{1}{2}L_y\tau_y^2(\mathcal{H}_{c,y}+\frac{1}{2}\mathcal{H}_{g,y})^2}\right\}.
\]
Given the structure of the line-search procedure in the y-updating phase, it follows the computed stepsize $t_k$ shall be lower bounded by
\[
t_{\text{low},y} := \gamma \min \left\{ 1, \delta_y, \frac{\delta_1}{\mathcal{G}_{g,y}}, \frac{2(1-\alpha_1)\epsilon_1}{{L_y} + 2\tau_y \mathcal{G}_{f,y}(\mathcal{H}_{c,y}+\frac{1}{2}\mathcal{H}_{g,y}) + L_y\tau_y^2(\mathcal{H}_{c,y}+\frac{1}{2}\mathcal{H}_{g,y})^2}\right\}, 
\]
and by combining this lower bound with the sufficient decrease inequality we obtain our result.  
\end{proof}

It is worth remarking the bound $t_{\text{low},y}$ is very pessimistic and follows from the worst possible scenario which might occur given the bounds on the norms of the Hessians and gradients for the objective and constraint functions. 

\subsection{X-Updating Phase}
For our analysis of the x-updating phase, we begin by proving a useful lemma. 
%
%
%
%
%
%
\begin{lemma}\label{lem:retraction}
Assume $\cM$ is a compact Riemannian manifold with positive injectivity radius $\kappa = \text{\rm inj}(\cM) >0$. 
Then for any two points $\bx, \by \in \cM$ with 
$\text{\rm dist}(\bx,\by) \leq \kappa/2$, 
then there exists $\bd \in \mathcal{T}_{\bx}\cM$ such that $\text{\rm Exp}(\bx,\bd)=\by$ and $\|\bd\|_{\bx} < \kappa$.  
\end{lemma}
\begin{proof}
Define the function $h:\cM \rightarrow \R$ such that 
$
h(\bx) := \min_{\|\bm{v}\|_{\bx} = \kappa/2, \; \bm{v} \in \mathcal{T}_{\bx}\cM} \; \text{\rm dist}(\bx,\text{Exp}(\bx,\bm{v})).
$
By Proposition 10.22 in \cite{boumal2023introduction}, since $\|\bm{v}\|_{\bx} = \kappa/2 < \text{inj}(\cM)$ it follows $h(\bx) = \| \bm{v}\|_{\bx} = \kappa/2$. Thus, $h$ is a constant function over $\cM$. Therefore, 
$
\min_{\bx \in \cM} h(\bx) = \kappa/2. 
$
Since $\mathcal{U}_{\bx}:= \{\text{Exp}(\bx,\bm{v}) \;|\; \|\bm{v}\|_{\bx} < \kappa, \; \bm{v} \in \mathcal{T}_{\bx}\cM\}$ is a neighborhood of $\bx$, it follows
$
\{ \tilde{\bx} \in \cM \; | \; \text{dist}(\tilde{\bx},\bx) < \kappa/2 \} \subset \mathcal{U}_{\bx}
$
yielding the result. 
\end{proof}

\begin{lemma}\label{lem:x_improvement}
If $m_x(\bxk,\byk) > \epsilon_2$ and the x-updating phase
is entered, then the update shall terminate with
\begin{equation}\label{eq:x_improve}
f(\bxkp,\bykp) \leq f(\bxk,\byk) - t_{\text{low},x}\cdot \alpha_2 \cdot  \epsilon_2
\end{equation}
where 
\begin{equation}\label{eq:x_step2}
t_{\text{low},x} :=  \gamma \min 
\left\{1, \frac{\delta_x}{L_1},\frac{\rho}{L_1}, \sqrt{\frac{\kappa_{\cM}}{\mathcal{G}_{c,x} \Gamma \tau_x}}, \frac{8(1-\alpha_2)\epsilon_2}{L_x \mathcal{G}_{c,x}^2 \Gamma^2 \tau_x^2 + 4(L_x + \mathcal{G}_{f,x} \mathcal{G}_{c,x} \Gamma \tau_x)} \right\}.
\end{equation}
\end{lemma}
\begin{proof}
Let $(\bxk,\byk) \in \cC$ satisfy $m^{\delta_1}_y(\bxk,\byk) \leq \epsilon_1$ and $m_x(\bxk,\byk) > \epsilon_2$ 
and let $\bd_x^k$ be a minimizer defining $m_x(\bxk,\byk)$. By the gradient Lipschitz retraction assumption,
\begin{align}\label{eq:x_help1}
f(\Retr_{\cM}(\bxk, t \bd_x^k),\byk) &\leq f(\bxk,\byk) + t \langle \grad f(\bxk,\byk), \bd_x^k \rangle_{\bxk} + \frac{L_x t^2}{2} \|\bd_x^k\|_{\bxk}^2 \nonumber \\ 
&\leq f(\bxk,\byk) - t m_x(\bxk,\byk) + \frac{L_x t^2}{2} \|\bd_x^k\|_{\bxk}^2. 
\end{align}
For $i=1,\hdots,p$, define the functions $h_i(t) = c_i(\Retr_{\cM}(\bxk,t \bd_x^k),\byk)$. By Taylor's theorem,
\[
h_i(t) = h_i(0) + t\langle \grad c_i(\bxk,\byk), \bd_x^k\rangle_{\bxk} + \frac{t^2}{2}h_i''(a)
\]
for some $a \in (0,1)$. 
Therefore, it follows by the definition of $\bd_x^k$ and \eqref{eq:Gcix} that 
\[
|c_i(\Retr_{\cM}(\bxk,t\bd_x^k),\byk)| \leq \frac{t^2}{2}\mathcal{G}_{c_i,x}
\]
for all $i$,
which implies $\|\bc(\Retr_{\cM}(\bxk, t \bd_x^k), \byk)\| \leq \frac{t^2}{2}\mathcal{G}_{c,x}$. 
Thus, by Assumption~\ref{a:error_bound_conditions} and \eqref{eq:error_x}, for a small enough stepsize $t$, to be specified shortly, the error bound condition applies and
\begin{align}\label{eq:x_help2}
\text{\rm dist}(\Retr_{\cM}(\bxk, t \bd_x^k), \cC_x(\byk)) \leq \tau_x \|\bc(\Retr_{\cM}(\bxk, t \bd_x^k), \byk)\|
\leq  \frac{\tau_x}{2}\mathcal{G}_{c,x}t^2.
\end{align}
Since $\cM$ is either $\text{St}(n_1,m_1)$ or $\text{St}(n_1,m_1) \times \cdots \times \text{St}(n_p,m_p)$, we can precisely specify how small the stepsize needs to be
so that 
the error bound condition applies.
For example, for the Stiefel manifold there exists a global constant $L_1$ such that
\begin{equation}\label{eq:retraction_cst}
\| \Retr_{\cM}(\bxk, t \bd_x^k) - \bxk \|_{\cM} \leq L_1 t \|\bd_x^k\|_{\bxk} \leq L_1 t
\end{equation}
for all $\bxk$ and $\bd_x^k$ for both the polar retraction and the QR retraction \cite{zhang2018cubic}. Therefore, if $t \leq \delta_x/L_1$, then the error bound condition \eqref{eq:error_x} shall hold making \eqref{eq:x_help2} a valid bound.  

Define $\bxk(t):= \Retr_{\cM}(\bxk,t \bd_x^k)$ and its projection onto the constraint set $\bx_p^k(t):=\Pi_{\cC_x(\byk)}\left(\bxk(t)\right)$. 
We now argue for the existence of a direction $\tilde{\bd}(t) \in T_{\bxk(t)}\cM$ such that 
\begin{equation}\label{eq:x_help23}
\bx_p^k(t)
= 
\text{Exp}\left(\bxk(t), \tilde{\bd}(t)\right)
\end{equation}
when $t$ is sufficiently small.
Since $\cM$ is either $\text{St}(n_1,m_1)$ or $\text{St}(n_1,m_1) \times \cdots \times \text{St}(n_p,m_p)$ for some  positive integers $n_i$ and $m_i$, it follows there is a positive global lower bound on the injectivity radius, $\kappa_{\cM}$, over $\cM$.
From \eqref{eq:x_help2}, \eqref{eq:retraction_cst}, and Assumption~\ref{a:norm_equivalence}, if additionally $t \leq \min\{\rho/L_1, \sqrt{\kappa_{\cM}/(\tau_x \Gamma \mathcal{G}_{c,x})}\}$ it follows 
\[
\text{\rm dist}(\bx_p^k(t), \bxk(t)) \leq \kappa_\cM/2
\]
and by Lemma~\ref{lem:retraction} there exists $\tilde{\bd}(t) \in \cT_{\bxk(t)}\cM$ with $\|\tilde{\bd}(t)\|_{\bxk(t)} < \kappa_\cM$ such that \eqref{eq:x_help23} holds. Furthermore, by Proposition 10.22 in \cite{boumal2023introduction}, Assumption~\ref{a:norm_equivalence}, and \eqref{eq:x_help2}, 
\begin{equation}\label{eq:x_help4}
\text{\rm dist}(\bx_p^k(t), \bxk(t)) = \|\tilde{\bd}(t)\|_{\bxk(t)} \leq \frac{\tau_x \Gamma}{2}\mathcal{G}_{c,x}t^2.
\end{equation}
%
%
Thus, by the gradient Lipschitz retraction condition of $f$ with respect to $\bx$
\begin{align}
f(\bx^k_p(t), \byk) &\overset{\eqref{eq:x_help23}}{\leq} f({\bx}^k(t), \byk) + \langle \grad_x f({\bx}^k(t), \byk), \tilde{\bd}(t) \rangle_{{\bx}^k(t)} + \frac{L_x}{2}\| \tilde{\bd}(t) \|^2_{{\bx}^k(t)}  \nonumber  \\
&\overset{\eqref{eq:Gfxnew}}{\leq} f({\bx}^k(t), \byk) + \mathcal{G}_{f,x} \|\tilde{\bd}(t)\|_{{\bx}^k(t)} + \frac{L_x}{2}\| \tilde{\bd}(t) \|^2_{{\bx}^k(t)}  \nonumber \\ 
&\overset{\eqref{eq:x_help4}}{\leq} f({\bx}^k(t), \byk) + \frac{1}{2}\mathcal{G}_{f,x}\mathcal{G}_{c,x}\Gamma\tau_x t^2  + \frac{L_x}{2}\left(\frac{1}{2}\mathcal{G}_{c,x} \Gamma \tau_x t^2\right)^2\nonumber \\ 
&\overset{\eqref{eq:x_help1}}{\leq} f(\bxk, \byk) -t m_x(\bxk,\byk) + \left(\frac{1}{2}L_x + \frac{1}{2} \mathcal{G}_{f,x}\mathcal{G}_{c,x}\Gamma \tau_x\right)t^2
+ \frac{1}{8} L_x \mathcal{G}_{c,x}^2\Gamma^2 \tau_x^2 t^4 \nonumber  \\
&\leq f(\bxk, \byk) -t m_x(\bxk,\byk) + \left(\frac{1}{2}L_x + \frac{1}{2} \mathcal{G}_{f,x}\mathcal{G}_{c,x}\Gamma \tau_x  + \frac{1}{8} L_x \mathcal{G}_{c,x}^2\Gamma^2 \tau_x^2   \right)t^2 \nonumber 
\end{align}
where the last inequality follows assuming $t\leq 1$.
Hence, if
\begin{equation}\label{eq:x_step1}
t \leq \min \left\{1, \frac{\delta_x}{L_1},\frac{\rho}{L_1}, \sqrt{\frac{\kappa_{\cM}}{\mathcal{G}_{c,x} \Gamma \tau_x}}, \frac{8(1-\alpha_2)m_x(\bxk,\byk)}{L_x \mathcal{G}_{c,x}^2 \Gamma^2 \tau_x^2 + 4(L_x + \mathcal{G}_{f,x} \mathcal{G}_{c,x} \Gamma \tau_x)} \right\}
\end{equation}
then the while loop in the x-updating phase shall terminate with a guaranteed descent.
Lastly, given \eqref{eq:x_step1} and the x-updating phase procedure it follows $t_k \geq t_{\text{low},x}$ for all $k\geq 0$ where 
\begin{equation}
t_{\text{low},x}  = \gamma \min \left\{1, \frac{\delta_x}{L_1},\frac{\rho}{L_1}, \sqrt{\frac{\kappa_{\cM}}{\mathcal{G}_{c,x} \Gamma \tau_x}}, \frac{8(1-\alpha_2)\epsilon_2}{L_x \mathcal{G}_{c,x}^2 \Gamma^2 \tau_x^2 + 4(L_x + \mathcal{G}_{f,x} \mathcal{G}_{c,x} \Gamma \tau_x)} \right\}. \nonumber 
\end{equation}
The result then follows as before via similar argumentation. 
\end{proof}

\subsection{Joint-Updating Phase}
We open our analysis of the joint-updating phase of Algorithm~\ref{alg:Staged_bcd} by proving $\barf$, as defined in Section~\ref{sec:manifold_reform}, has a gradient Lipschitz retraction over ${\cM}'$. 

\begin{lemma}\label{lem:f_grad_lip}
Under Assumptions~\ref{a:smooth2} and \ref{a:manifold}, $\barf:\cM' \rightarrow \R$ has a gradient Lipschitz retraction with the retraction over the product manifold, $\Retr_{\cM'}: \cT\cM' \rightarrow \cM'$ defined such that $(\bx, \eta_x) \times (\by, \eta_y) \mapsto (\Retr_{\cM}(\bx,\eta_x), \by + \eta_y)$, with constant $L_R = L_{xy}L_1 
 + \max\left(L_{x}, L_{y} \right)$. 
\end{lemma}
\begin{proof}
Let $\bz = (\bx,\by) \in \cM \times \R^n$ and $\eta_z = (\eta_x,\eta_y) \in \cT_{\bx}\cM \times \R^n$. Then, 
\begin{align}\label{eq:barf_grad_retr1}
&\barf(\Retr_{\cM'}(\bz,\eta_z))\nonumber \\
&\hspace{0.3in}= f(\Retr_{\cM}(\bx,\eta_x), \by + \eta_y) \nonumber \\ 
&\hspace{0.3in}\leq f(\Retr_{\cM}(\bx,\eta_x), \by) + \langle \nabla_y f(\Retr_{\cM}(\bx,\eta_x),\by),\eta_y \rangle + \frac{L_{y}}{2} \| \eta_y \|^2 \nonumber \\
&\hspace{0.3in}\leq f(\bx,\by) + \langle \grad f(\bx,\by), \eta_x \rangle_{\bx} + \langle \nabla_y f(\Retr_{\cM}(\bx,\eta_x),\by),\eta_y \rangle + \frac{L_{y}}{2} \| \eta_y \|^2 + \frac{L_{x}}{2}\| \eta_x\|_{\bx}^2 \nonumber \\
&\hspace{0.3in}= f(\bx,\by) + \langle \nabla_y f(\bx,\by), \eta_y \rangle + \langle \nabla_y f(\Retr_{\cM}(\bx,\eta_x),\by)- \nabla_y f(\bx,\by),\eta_y \rangle  \nonumber \\
&\hspace{1in} +  \langle \grad f(\bx,\by), \eta_x \rangle_{\bx}  + \frac{L_{y}}{2} \| \eta_y \|^2 + \frac{L_{x}}{2}\| \eta_x\|_{\bx}^2,
\end{align}
where both inequalities followed by the gradient Lipschitz conditions on $f$. 
Investigating the third term we see 
\begin{align}\label{eq:third_term1}
| \langle \nabla_y f(\Retr_{\cM}(\bx,\eta_x),\by)- \nabla_y f(\bx,\by),\eta_y \rangle | &\leq 
 \| \nabla_y f(\Retr_{\cM}(\bx,\eta_x),\by) - \nabla_y f(\bx,\by)\| \cdot \|\eta_y\| \nonumber \\
 &\overset{\eqref{eq:f_grad_xy}}{\leq} L_{xy} \|\Retr_{\cM}(\bx,\eta_x) - \bx\|_{\cM} \cdot \|\eta_y\| \nonumber \\
 &\overset{\eqref{eq:retraction_cst}}{\leq} L_{xy} L_1 \cdot \|\eta_x\|_{\bx}\cdot \|\eta_y\| \nonumber \\
 &\leq \frac{L_{xy}L_{1}}{2} \cdot \left(\|\eta_x\|_{\bx}^2 + \|\eta_y\|^2 \right),
\end{align}
where the mixed Lipschitz condition for $f$ was utilized to obtained the second inequality and the third followed by Assumption~\ref{a:manifold} and Proposition~\ref{prop:retraction_lip}.
Combining \eqref{eq:barf_grad_retr1} and \eqref{eq:third_term1}, we obtain our result. 
\end{proof}

\begin{remark}
Note, this lemma does depend on the retraction being applied. Due to Assumption~\ref{a:smooth2}, Lemma~\ref{lem:f_grad_lip} applies to both the exponential map and the chosen retraction utilized in Algorithm~\ref{alg:Staged_bcd}.
\end{remark}

We now present the improvement guarantee for the joint-updating phase if this part of the algorithm begins with a point which is not an $(\delta_2,\epsilon_3)$-approximate first-order KKT point of \eqref{eq:main_model}. 
The argument is similar to the analysis of both the $x$ and $y$-updating phases.  
\begin{lemma}\label{lem:joint_improvement}
Let $\delta_2 > 0$.
If $m^{\delta_2}_{KKT}(\bxk,\byk) > \epsilon_3$ and the joint-updating phase
is entered, then the update shall terminate with
\begin{equation}\label{eq:j_improve}
f(\bxkp,\bykp) \leq f(\bxk,\byk) - t_{\text{low},kkt}\cdot \alpha_3 \cdot \epsilon_3
\end{equation}
where 
\begin{multline}
t_{\text{low},\text{kkt}} :=  \gamma \min\bigg\{1, \frac{\delta_z}{L_1+1}, \frac{\rho}{L_1+1},  \sqrt{\frac{\kappa_{\cM}}{\tau_z \Gamma(\mathcal{G}_{\bar{c},z}+ \mathcal{H}_{g,y} )}}, \\ \frac{(1-\alpha_3)\epsilon_3}{ \frac{L_R}{2} + \frac{\tau_z}{2}\Gamma\mathcal{G}_{\bar{f},z}(\mathcal{G}_{\bar{c},z} + \mathcal{H}_{g,y})+ \frac{1}{8}L_R \tau_z^2 \Gamma^2 \left( \mathcal{G}_{\bar{c},z} + \mathcal{H}_{g,y} \right)^2 } \bigg\}.
\end{multline}
\end{lemma}
\begin{proof}
Let $\bm{z}^k = (\bxk,\byk)\in \cC$ satisfy $m^{\delta_1}_y(\bxk,\byk) \leq \epsilon_1$, $m_x(\bxk,\byk) \leq \epsilon_2$, and $m^{\delta_2}_{KKT}(\bxk,\byk) > \epsilon_3$. Let $\bd^k=(\bd_x^k,\bd_y^k) \in \mathcal{T}_{\bxk}\cM \times \R^n$ be a minimizer defining  $m^{\delta_2}_{KKT}(\bxk,\byk)$.
By  Lemma~\ref{lem:f_grad_lip}, we know $\barf$ has a gradient Lipschitz retraction. So, using the definition of the retraction on the product manifold $\cM'$ we have
\begin{align}\label{eq:f_improve}
f(\Retr_{\cM}(\bxk, t \bd_x^{k}), \byk + t \bd_y^{k}) & \leq f(\bxk,\byk) + t \langle \grad \barf(\bz^{k}), \bd^{k} \rangle_{\bz^{k}} + \frac{L_R t^2}{2} \| \bd^{k} \|_{\bz^{k}}^2 \nonumber \\ 
&\leq f(\bx^k,\by^k)  - t\cdot m^{\delta_2}_{KKT}(\bx^k,\by^k) + \frac{1}{2}L_R t^2.
\end{align}
This demonstrates a first-order improvement on the value of the objective function can be obtained. 
Next, we demonstrate only a second-order violation occurs in the constraints.  
Similar to the x-updating phase, for $i=1,\hdots,p$, define the functions $\bar{h}_i(t) = \bar{c}_i(\Retr_{\cM'}(\bz^k,t \bd^k))$. By Taylor's theorem,
\[
\bar{h}_i(t) = \bar{h}_i(0) + t\langle \grad \bar{c}_i(\bz^k), \bd^k\rangle_{\bz^k} + \frac{t^2}{2}\bar{h}_i''(a)
\]
for some $a \in (0,1)$. 
Therefore, it follows by the definition of $\bd^k$ and \eqref{eq:Gbarcix} that 
\[
|\bar{c}_i(\Retr_{\cM'}(\bz^k,t\bd^k))| \leq \frac{t^2}{2}\mathcal{G}_{\bar{c}_i,z} \text{ which implies } \|\bc(\Retr_{\cM}(\bxk, t \bd_x^k), \byk + t\bd_y^k)\| \leq \frac{t^2}{2}\mathcal{G}_{\bar{c},z}.
\]
Since the inequality constraints only depend upon $\by$, the prior argument to prove \eqref{eq:gy_bound} stands provided $t \leq \delta_2/\mathcal{G}_{g,y}$. 
Then, since our utilized retraction over $\cM$ satisfies \eqref{eq:retraction_cst}, if $t \leq \delta_z/(L_1 +1)$ we have
\begin{align}
\| \Retr_{\cM'}(\bz^k, t \bd^k) - \bz^k\|_{\cM'} &= \| \Retr_{\cM}(\bx^k, t \bd_x^k) - \bx^k\|_{\cM} + \| (\byk + t \bd_y^k) - \byk\|  \nonumber \\
&\leq t L_1 \| \bd_x^k\|_{\bxk}  + t \|\bd_y^k\|  \nonumber \\
&\leq \delta_z,
\end{align}
and the error bound condition \eqref{eq:error_z} holds.
Therefore,  
\begin{equation}\label{eq:joint_error}
\text{dist}\left( (\Retr_{\cM}(\bxk,t \bd_x^k), \byk + t \bd_y^k), \cC\right) \leq \frac{\tau_z}{2} \left(\mathcal{G}_{\bar{c},z} + \mathcal{H}_{g,y} \right) t^2.
\end{equation}

Define $\bz^k(t)= (\bx^k(t), \by^k(t)) := (\Retr_{\cM}(\bxk, t \bd_x^{k}), \byk + t \bd_y^{k})$ and $(\bx_p^k(t), \by_p^k(t)) = \Pi_{\cC}(\bx^k(t), \by^k(t))$. 
By Lemma~\ref{lem:retraction}, Assumption~\ref{a:norm_equivalence}, and the fact the exponential map over the product manifold $\cM'$ is the Cartesian product of the exponential maps over $\cM$ and $\R^n$, if 
\[
t \leq \min \{ \rho/(L_1 +1), \sqrt{{\kappa_\cM}/(\tau_z\Gamma(\mathcal{G}_{\bar{c},z} + \mathcal{H}_{g,y}))} \},
\]
then there exists $\tilde{\bd}(t) = (\tilde{\bd_x}(t), \tilde{\bd_y}(t)) \in \mathcal{T}_{(\bxk(t),\byk(t))}\cM'$ such that 
\[
(\bxk_p(t),\byk_p(t)) =  \left( \text{Exp}(\bxk(t),t\tilde{\bd_x}(t)), \byk + t\tilde{\bd_y}(t) \right).
\]
Furthermore, we have
\begin{equation}\label{eq:joint_1}
\text{\rm dist}((\bxk_p(t),\byk_p(t)), (\bxk(t),\byk(t))) = \|\tilde{\bd}(t)\|_{\bz^k(t)} \leq \frac{\tau_z \Gamma}{2}\left(\mathcal{G}_{\bar{c},z} + \mathcal{H}_{g,y}\right) t^2.
\end{equation}
Therefore, by the gradient Lipschitz retraction on $\bar{f}$ we have
\begin{align}
&f(\bxk_p(t),\byk_p(t))  \nonumber \\
&\leq f(\bxk(t),\byk(t)) + \langle \grad_{z} \bar{f}(\bz^k(t)), \tilde{\bd}(t)\rangle_{\bz^k(t)}  + \frac{L_R}{2}\|\tilde{\bd}(t) \|_{\bz^k(t)}^2 \nonumber  \\
&\leq f(\bxk(t),\byk(t)) + \mathcal{G}_{\bar{f},z} \|\tilde{\bd}(t)\|_{\bz^k(t)}  + \frac{L_R}{2}\|\tilde{\bd}(t) \|_{\bz^k(t)}^2 \nonumber  \\
&\overset{\eqref{eq:joint_1}}{\leq} f(\bxk(t),\byk(t)) + \mathcal{G}_{\bar{f},z}\left( \frac{\tau_z\Gamma}{2} \left(\mathcal{G}_{\bar{c},z} + \mathcal{H}_{g,y} \right) t^2\right)  + \frac{L_R}{2}\left( \frac{\tau_z\Gamma}{2} \left(\mathcal{G}_{\bar{c},z} + \mathcal{H}_{g,y} \right) t^2\right)^2 \nonumber \\
&\overset{\eqref{eq:f_improve}}{\leq} f(\bxk,\byk) - t m^{\delta_2}_{KKT}(\bxk,\byk) + \frac{1}{2}L_R t^2 \nonumber \\
&\hspace{2.0in}+ \mathcal{G}_{\bar{f},z}\left( \frac{\tau_z \Gamma}{2} \left(\mathcal{G}_{\bar{c},z} + \mathcal{H}_{g,y} \right) t^2\right)  + \frac{L_R}{2}\left( \frac{\tau_z\Gamma}{2} \left(\mathcal{G}_{\bar{c},z} + \mathcal{H}_{g,y} \right) t^2\right)^2 \nonumber \\
&\leq f(\bxk,\byk) - t m^{\delta_2}_{KKT}(\bxk,\byk) + \left(\frac{L_R}{2} + \frac{\tau_z\Gamma}{2} \mathcal{G}_{\bar{f},z}(\mathcal{G}_{\bar{c},z} + \mathcal{H}_{g,y})+ \frac{1}{8}L_R \Gamma^2\tau_z^2 \left( \mathcal{G}_{\bar{c},z} + \mathcal{H}_{g,y} \right)^2 \right)t^2, \nonumber 
\end{align}
where the last inequality follows assuming $t\leq 1$. Thus, a sufficient decrease shall be obtained provided 
\begin{multline}
t \leq \min\bigg\{1, \frac{\delta_z}{L_1+1}, \frac{\rho}{L_1+1}, \sqrt{\frac{\kappa_{\cM}}{\tau_z \Gamma(\mathcal{G}_{\bar{c},z}+ \mathcal{H}_{g,y} )}},\\ \frac{(1-\alpha_3)m^{\delta_2}_{KKT}(\bxk,\byk)}{ \frac{L_R}{2} + \frac{\tau_z}{2}\Gamma\mathcal{G}_{\bar{f},z}(\mathcal{G}_{\bar{c},z} + \mathcal{H}_{g,y})+ \frac{1}{8}L_R \tau_z^2 \Gamma^2 \left( \mathcal{G}_{\bar{c},z} + \mathcal{H}_{g,y} \right)^2 } \bigg\}. \nonumber
\end{multline}
Hence, as in the prior arguments, it follows the stepsize in the joint-updating phase is lower bounded by 
\begin{multline}
t_{\text{low},\text{kkt}} =  \gamma \min\bigg\{1, \frac{\delta_z}{L_1+1}, \frac{\rho}{L_1+1},  \sqrt{\frac{\kappa_{\cM}}{\tau_z \Gamma(\mathcal{G}_{\bar{c},z}+ \mathcal{H}_{g,y} )}}, \\ \frac{(1-\alpha_3)\epsilon_3}{ \frac{L_R}{2} + \frac{\tau_z}{2}\Gamma\mathcal{G}_{\bar{f},z}(\mathcal{G}_{\bar{c},z} + \mathcal{H}_{g,y})+ \frac{1}{8}L_R \tau_z^2 \Gamma^2 \left( \mathcal{G}_{\bar{c},z} + \mathcal{H}_{g,y} \right)^2 } \bigg\}  \nonumber 
\end{multline}
and the result follows. 
\end{proof}

\subsection{Overall Complexity}
Using the prior updating phases, we can derive the overall complexity of Algorithm~\ref{alg:Staged_bcd} for locating an approximate first-order KKT point of \eqref{eq:main_model}. 
\begin{theorem}\label{thm:complexity}
Assuming the initial level-set $L(\bx^0,\by^0)$ is bounded, their exists a finite lower bound $f^*$ such that $f^* \leq f(\bx,\by)$ for all $(\bx,\by) \in L(\bx^0,\by^0)$, and Assumptions 1-6 hold, then Algorithm~\ref{alg:Staged_bcd} shall locate an $(\epsilon,\epsilon)$-approximate first-order KKT point of \eqref{eq:main_model} in at most $\mathcal{O}(1/\epsilon^2)$ iterations.  
\end{theorem}
\begin{proof}
Assume $\delta_1 = \delta_2 = \epsilon_1 = \epsilon_2 = \epsilon_3 = \epsilon$ in Algorithm~\ref{alg:Staged_bcd}. 
For the sake of obtaining a contradiction, assume additionally the algorithm never terminates. 
By Lemmas~\ref{lem:y_improvement}, \ref{lem:x_improvement}, and \ref{lem:joint_improvement}, it follows 
\[
f(\bxkp,\bykp) \leq f(\bxk,\byk) - \left(\min_{1\leq i \leq 3} \alpha_i\right) \cdot \min\{t_{\text{low},y}, t_{\text{low},x}, t_{\text{low,kkt}}\}\cdot \epsilon
\]
 for all $k\geq 0$. For $\epsilon$ small enough there exist positive constants $C_1, C_2$, and $C_3$ such that 
 \[
t_{\text{low},y} = C_1 \epsilon,\;\; t_{\text{low},x} = C_2 \epsilon,\;\; \text{and}\;\; t_{\text{low},kkt} = C_3 \epsilon.
\]
Hence, letting $\Theta:= \left(\min_{1\leq i \leq 3} \alpha_i \right)\cdot \left(\min_{1\leq i \leq 3} C_i \right)$ we have 
 $
 f(\bxkp,\bykp) \leq f(\bxk,\byk) - \Theta \epsilon^2
 $
 for all $k\geq 0$.
Since the algorithm does not terminate, by repeated application of the above inequality we have
 \[
 f(\bxkp,\bykp) \leq f(\bx,^0\by^0) - \Theta (k+1)\epsilon^2.
 \]
So, if $k > \lceil \frac{f(\bx^0,\by^0)-f^*}{\Theta \epsilon^2} + 1\rceil$, then $f(\bxkp,\bykp)< f^*$  which violates our assumptions. 
Therefore, the algorithm must terminate in at most $\mathcal{O}(1/\epsilon^2)$ iterations.
\end{proof}

\section{Numerical Experiments}\label{sec:experiments}
The purpose of this section is twofold. First, we want to demonstrate the efficacy of the modeling framework given by \eqref{eqn:gen_spec_coord}, and, second, we want to exhibit the theory and applicability of Algorithm~\ref{alg:Staged_bcd}. 

\subsection{Generalized Semidefinite Programs}
To demonstrate the novelty of our modeling framework and the ability of our method to solve general matrix optimization models with nonconvex constraints, we present a set of numerical tests on generalized semidefinite programs.
The instance of \eqref{eq:gen_SDP} we considered in our experiments was
\begin{align}\label{eq:gen_SDP_test}
\min&\;\; \langle -\bm{I}, \bm{X} \rangle \\
\text{s.t.}&\;\; \langle \bA_i, \bX \rangle = \ell_i, \;\; i=1,\hdots, s, \nonumber \\
           &\;\; \lambda_n(\bX) \leq b_1, \nonumber \\
           &\;\; \lambda_n(\bX) + \lambda_{n-1}(\bX) \leq b_2, \nonumber \\ 
           &\;\;\hspace{0.5in}\vdots \nonumber \\
           &\;\; \lambda_n(\bX) + \lambda_{n-1}(\bX) + \cdots + \lambda_1(\bX) \leq b_n, \nonumber \\
           &\;\; \lambda_n(\bX) \geq 0, \nonumber \\
           &\;\; \bX \in \Symn, \nonumber  
\end{align}
where we designed the problem parameters such that a feasible point existed. 
A few observations are notable about \eqref{eq:gen_SDP_test}.
First, the objective function is lower bounded by $-b_n$ if we look at the second to last inequality constraint in the model. 
Thus, we have a simple way of measuring if our algorithm is able to compute a global minimizer. 
Second, the eigenvalue constraint in the model is nonconvex since the left-hand side of each of the inequality constraints is a concave function. 

We generated our test problems by randomly constructing a positive definite matrix $\bC$, setting $\ell_i = \langle \bA_i, \bC \rangle$ for $i=1, \hdots, s$, with $\bA_i$ a randomly generated symmetric matrix, and setting $b_i$ equal to the sum of the $i$ smallest eigenvalues of $\bC$. 
Hence, $\bC$ is a feasible solution to \eqref{eq:gen_SDP_test} and because it makes the second to last inequality tight it yields a globally optimal solution.  

We conducted ten randomized tests for varying problems sizes of \eqref{eq:gen_SDP_test}. 
The results are displayed in Table~\ref{tab:SDP_table}.
For each randomized test we computed the following values: 
\begin{itemize}
\item the absolute difference between the objective value of the final iterate generated by Algorithm~\ref{alg:Staged_bcd} and the optimal objective value, i.e., $|f(\bX^{\text{final}}) - f^*|$
\item the feasibility error of the final iterate computed by Algorithm~\ref{alg:Staged_bcd} with respect to the equality constraints, i.e., $\|\bm{G}(\bX^{\text{final}})\|$
\item the feasibility error of the final iterate produced by Algorithm~\ref{alg:Staged_bcd} with respect to the inequality constraints, i.e., $\| \max\left(\bm{g}(\blambda(\bX^{\text{final}})),\bm{0}\right)\|$
\end{itemize}
%
Table~\ref{tab:SDP_table} also tracks the minimum, median, and maximum values of these three quantities across the ten randomized experiments for each problem size.
It also counts the total number of problems solved for each dimension; 
we consider a problem solved if the difference between the objective value computed when the algorithm terminated was within $10^{-6}$ of the optimal value, and the point located had a feasibility error less than $10^{-6}$ in both the equality and inequality constraints. 

\begin{table}[H]
\centering
\begin{tabular}{|c|c|cll|cc|}
\hline
\multicolumn{1}{|l|}{\textbf{Dim.}} & \textbf{Solved}  & \multicolumn{3}{l|}{\textbf{Dist. to Opt. Value}}     & \multicolumn{2}{c|}{\textbf{Feasibility Error}}              \\ \cline{3-7} 
\multicolumn{1}{|l|}{}                       &                          & \multicolumn{3}{c|}{( $|f(\bX^{\text{final}}) - f^*|$ ) }   & 
\multicolumn{1}{c|}{\textbf{Equality}} 
& 
\textbf{Inequality} 
\\ \hline
\textbf{n=5}    & 9  & \multicolumn{3}{c|}{[0, 1.48e-11, 1.22]}        & \multicolumn{1}{c|}{[2.22e-16, 1.12e-08, 1.78e-07]}   & [0, 4.86e-7, 9.99e-7]     \\ \hline
\textbf{n=10}   & 8  & \multicolumn{3}{c|}{[0, 1.31e-08, 3.68]}        & \multicolumn{1}{c|}{[7.11e-15, 2.51e-14, 2.57e-07]}   
& [0, 0, 9.53e-7]     
\\ \hline
\textbf{n=25}   & 10 & \multicolumn{3}{c|}{[0, 2.27e-13, 5.89e-8]}     & \multicolumn{1}{c|}{[7.69e-14, 1.42e-13, 2.06e-07]}   
& [0, 0, 9.69e-7]     
\\ \hline
\textbf{n=50}   & 10 & \multicolumn{3}{c|}{[0, 1.59e-12, 3.44e-7]}     & \multicolumn{1}{c|}{[6.96e-13, 4.40e-09, 1.30e-07]}   & [0, 4.54e-7, 9.69e-7]     \\ \hline
\textbf{n=100}  & 8  & \multicolumn{3}{c|}{[1.82e-12, 5.46e-12, 4.72]}& \multicolumn{1}{c|}{[4.33e-12, 3.22e-10, 1.89e-08]}   & 
[0, 4.47e-7, 9.93e-7]     
\\ \hline
\end{tabular}
\caption{Displays the results of the randomized tests conducted on \eqref{eq:gen_SDP_test}. 
Each entry in the columns provides the minimum, median, and maximum value of the respective quantity for the ten randomized experiments conducted on each problem size $n$.
Any numerical value obtained below $10^{-20}$ was set to zero in the table. 
}
\label{tab:SDP_table}
\end{table}

From the table we notice Algorithm~\ref{alg:Staged_bcd} was able to find near globally optimal solutions to a vast majority of the problems. 
Even though \eqref{eq:gen_SDP_test} is a nonconvex problem, our method successfully solve 45 of the 50 problems instances. 
Furthermore, our method seems to be independent of the problem size, at least for the subset of generalized SDPs represented by \eqref{eq:gen_SDP_test}. 
A question worth pursuing in the future is applications of generalized SDPs, especially as an approach to relax problems which have classically been rewritten as semidefinite programs. 

\subsection{Quadratically Constrained Quadratic Programs}\label{sec:QCQP_exp}
An important instance of a QCQP is 
\begin{align}\label{eq:QCQP_test}
\min&\; \bx^\top \bm{C} \bx \\
\text{s.t.}&\; \bx^\top \bm{A}_i \bx \geq 1, \; i=1,\hdots, m, \nonumber \\
           &\; \bx \in \R^n, \nonumber 
\end{align}
where $\bm{C}$ and $\bm{A}_i$ are positive semidefinite. 
Though the objective function in \eqref{eq:QCQP_test} is convex, the constraint is nonconvex since it is the intersection of the exterior of $m$ ellipsoids. 
This model is well-known to be NP-Hard and in-general more difficult than Max-cut, 
a comparison of the approximation ratios between the two problems supports this position;
\eqref{eq:QCQP_test} has applications in communications through the multicast downlink transmit beamforming problem \cite{alex2010convex}.

A popular and powerful approach to tackle \eqref{eq:QCQP_test} follows by relaxing the problem to a semidefinite program and then applying a randomization procedure to compute feasible solutions to the problem. 
Namely, \eqref{eq:QCQP_test} is rewritten as the equivalent matrix optimization model,
\begin{align}\label{eq:QCQP_mat}
\min&\; \langle \bm{C}, \bX \rangle  \\
\text{s.t.}&\; \langle \bm{A}_i, \bX \rangle \geq 1, \; i=1,\hdots, m, \nonumber \\
           &\; \bX \succeq \bm{0}, \; \text{rank}(\bX) = 1, \nonumber 
\end{align}
and then relaxed to the semidefinite program 
\begin{align}\label{eq:QCQP_SDR}
\min&\; \langle \bm{C}, \bX \rangle \\
\text{s.t.}&\; \langle -\bA_i, \bX \rangle \leq -1, \; i=1,\hdots, m, \nonumber \\
           &\; \bX \succeq \bm{0}. \nonumber 
\end{align}
A global min $\bX^*_{sdp}$ is then computed to \eqref{eq:QCQP_SDR} and is utilized to define Gaussian random variables $\bm{\xi}_\ell \sim \mathcal{N}(\bm{0}, \bX_{sdp}^*)$ for $\ell = 1, \hdots, L$.  
From these random variables, feasible solutions to \eqref{eq:QCQP_test} are formed, $\bx(\bm{\xi}_\ell)$, such that 
\begin{equation}\label{eq:rand_QCQP}
\bx(\bm{\xi}_\ell)  = \frac{\bm{\xi}_\ell}{\sqrt{\min_{1\leq i \leq m} \bm{\xi}_\ell^\top \bA_i \bm{\xi}_\ell}}
\end{equation}
for $\ell=1,\hdots, L$.
Randomization procedures such as this combined with semidefinite program relaxations have proven extremely effective at locating approximate solutions to challenging QCQPs;
see \cite{luo2010semidefinite} for a survey of semidefinite relaxation and randomization applied to QCQPs. 

The model and algorithm we have developed present new approaches for solving \eqref{eq:QCQP_test} 
by enabling novel nonconvex relaxations of the rank-1 constraint in the equivalent model \eqref{eq:QCQP_mat}. 
Rather than removing the rank condition completely to obtain a convex problem, we instead solve 
\begin{align}\label{eq:QCQP_SCO}
\min&\; \langle \bm{C}, \bX \rangle \\
\text{s.t.}&\; \langle -\bA_i, \bX \rangle \leq -1, \; i=1,\hdots, m, \nonumber \\
           &\; \lambda_1(\bX) \geq \delta, \; \lambda_i(\bX) \in [0,\delta],\; i=2,\hdots, n, \nonumber \\
           &\; \bX \in \Symn, \nonumber 
\end{align}
where the parameter $\delta \geq 0$ allows us to alter the tightness of our approximation.
If $\delta = 0$, then \eqref{eq:QCQP_SCO} is equivalent to \eqref{eq:QCQP_mat}.
For small positive values of $\delta$, the eigenvalue constraint enforces a near rank-1 constraint, as the bottom $n-1$ eigenvalues of $\bX$ are then forced to be within $\delta$ of zero. 
Once an approximate solution to \eqref{eq:QCQP_SCO} is computed using Algorithm~\ref{alg:Staged_bcd},  we form feasible solutions to \eqref{eq:QCQP_test} in two different fashions:
\begin{enumerate}
\item[(A1.)] {\bf Projection Approach:} After computing a stationary point $\bX^*$ to \eqref{eq:QCQP_SCO}, the top eigenvalue $\lambda_1(\bX^*)$ and corresponding eigenvector $\bm{v}_1$ of $\bX^*$ are computed and the vector $\bm{\xi}_* = \sqrt{\lambda_1(\bX^*)} \bm{v}_1$ is formed. 
Then, using \eqref{eq:rand_QCQP}, the vector $\bm{\xi}_*$ is scaled so that it satisfies the quadratic inequality constraints in \eqref{eq:QCQP_test}.
This scaled version of $\bm{\xi}_*$ is then used as an approximate solution to \eqref{eq:QCQP_test}.
\item[(A2).] {\bf Randomization Approach:} After computing a stationary point to \eqref{eq:QCQP_SCO}, the stationary point is utilized to generate randomized solutions in the same manner as done in the semidefinite relaxation and randomization approach described above.
\end{enumerate}

\begin{figure}[t]
    \centering
    \includegraphics[width=0.7\linewidth]{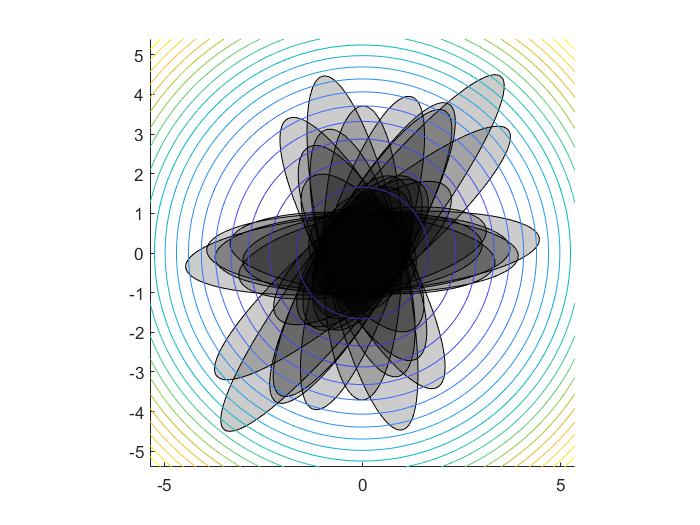}
    \caption{
    Displays the feasible region and level sets for an instance of \eqref{eq:QCQP_test_instances} with $m=25$. The white region represents all of the feasible points outside the interior of the $m$ ellipsoids while the concentric circles centered at the origin are the level sets of the norm-squared objective function.  
    }
    \label{fig:feasible_region}
\end{figure}

To compare our model \eqref{eq:QCQP_SCO} and two approximation approaches detailed above to the semidefinite program relaxation (SDR) approach with randomization, we conducted forty numerical experiments. 
In-order to check if our method was able to compute global minimums, we conducted our numerical experiments with $n=2$ so that a brute force grid search was able to compute a global optimum of \eqref{eq:QCQP_test} in a less than prohibitively long time. 
The form of \eqref{eq:QCQP_test} we considered in our experiments was 
\begin{align}\label{eq:QCQP_test_instances}
\min&\; \|\bx\|^2 \\
\text{s.t.}&\; \bx^\top \bm{A}_i \bx \geq 1, \; i=1,\hdots, m, \nonumber \\
           &\; \bx \in \R^2, \nonumber 
\end{align}
where we randomly generated positive definite matrices $\bA_i$ for $i=1, \hdots, m$. 
An example of the feasible region and level sets of the objective can be seen in Figure~\ref{fig:feasible_region}. 
In our experiments, we conducted ten randomized tests for $m=5, 10, 25$, and $50$, and solved \eqref{eq:QCQP_SCO} with $\delta=10^{-1}, 10^{-3}$, and $10^{-6}$ for all settings.
When generating points via randomization, we sampled 20 points from the Gaussian distributions. 
Since our model is nonconvex, we can benefit from initializing our algorithm multiple times and resolving, whereas the SDR approach cannot. 
Thus, for these tests, we solved each instance of \eqref{eq:QCQP_SCO} three times. One initialization of our algorithm was formed by taking the solution to \eqref{eq:QCQP_SDR} and projecting it onto the constraint set in \eqref{eq:QCQP_SCO}; the other two came from randomly generating a symmetric matrix and projecting onto the constraint set. 
Tables~\ref{tab:QCQP_tab1} and \ref{tab:QCQP_tab2} display the numerical results of the experiments. 

\begin{landscape}
\begin{table}[H]
\centering
\begin{tabular}{|cc|c|cc|cccllllll|}
\hline
\multicolumn{1}{|l}{} & \multicolumn{1}{l|}{} & \multicolumn{1}{l|}{}  & \multicolumn{2}{c|}{\textbf{SDR}}            & \multicolumn{9}{c|}{\textbf{Our Model}}                                                                                                                                                                                                                                                                                                                       \\ \cline{6-14}
\multicolumn{1}{|l}{} & \multicolumn{1}{l|}{} & \multicolumn{1}{l|}{}  & \multicolumn{1}{l}{} & \multicolumn{1}{l|}{} & \multicolumn{1}{l}{} & \multicolumn{1}{l}{$\delta = 10^{-1}$} & \multicolumn{1}{l|}{}                    &                                       & $\delta = 10^{-3}$                           & \multicolumn{1}{l|}{}                    &                                       & $\delta = 10^{-6}$                           &                                          \\ \cline{1-14}
\textbf{m}            & \textbf{Test}         & \textbf{Optimal} & \textbf{Orig.}    & \textbf{Random}   & \multicolumn{1}{c|}{\textbf{Orig.}} & \textbf{Random}               & \multicolumn{1}{c|}{\textbf{Project}} & \multicolumn{1}{c|}{\textbf{Orig.}} & \multicolumn{1}{c}{\textbf{Random}} & \multicolumn{1}{c|}{\textbf{Project}} & \multicolumn{1}{c|}{\textbf{Orig.}} & \multicolumn{1}{c}{\textbf{Random}} & \multicolumn{1}{c|}{\textbf{Project}} \\ \hline
  5 &  1 &  1.10 &  1.10 & \textbf{\color{blue} 1.10} & \multicolumn{1}{c|}{ 1.10} & \textbf{\color{blue} 1.10} & \multicolumn{1}{c|}{\textbf{\color{blue} 1.10}} & \multicolumn{1}{c|}{ 1.10} & \multicolumn{1}{c}{\textbf{\color{blue} 1.10}} & \multicolumn{1}{c|}{\textbf{\color{blue} 1.10}} & \multicolumn{1}{c|}{ 1.10} & \multicolumn{1}{c}{\textbf{\color{blue} 1.10}} & \multicolumn{1}{c|}{\textbf{\color{blue} 1.10}} \\
 &  2 &  1.49 &  1.49 & \textbf{\color{blue} 1.49} & \multicolumn{1}{c|}{ 1.49} & \textbf{\color{blue} 1.49} & \multicolumn{1}{c|}{\textbf{\color{blue} 1.49}} & \multicolumn{1}{c|}{ 1.49} & \multicolumn{1}{c}{\textbf{\color{blue} 1.49}} & \multicolumn{1}{c|}{\textbf{\color{blue} 1.49}} & \multicolumn{1}{c|}{ 1.49} & \multicolumn{1}{c}{\textbf{\color{blue} 1.49}} & \multicolumn{1}{c|}{\textbf{\color{blue} 1.49}} \\
 &  3 &  2.86 &  1.91 &  3.17 & \multicolumn{1}{c|}{ 2.74} &  2.94 & \multicolumn{1}{c|}{\textbf{\color{blue} 2.87}} & \multicolumn{1}{c|}{ 2.85} & \multicolumn{1}{c}{\textbf{\color{blue} 2.86}} & \multicolumn{1}{c|}{\textbf{\color{blue} 2.85}} & \multicolumn{1}{c|}{ 2.85} & \multicolumn{1}{c}{\textbf{\color{blue} 2.85}} & \multicolumn{1}{c|}{\textbf{\color{blue} 2.85}} \\
 &  4 &  2.33 &  1.77 &  2.41 & \multicolumn{1}{c|}{ 2.25} & \textbf{\color{blue} 2.33} & \multicolumn{1}{c|}{\textbf{\color{blue} 2.34}} & \multicolumn{1}{c|}{ 2.33} & \multicolumn{1}{c}{\textbf{\color{blue} 2.33}} & \multicolumn{1}{c|}{\textbf{\color{blue} 2.33}} & \multicolumn{1}{c|}{ 2.33} & \multicolumn{1}{c}{\textbf{\color{blue} 2.33}} & \multicolumn{1}{c|}{\textbf{\color{blue} 2.33}} \\
 &  5 &  1.67 &  1.66 &  1.83 & \multicolumn{1}{c|}{ 1.67} & \textbf{\color{blue} 1.68} & \multicolumn{1}{c|}{\textbf{\color{blue} 1.68}} & \multicolumn{1}{c|}{ 1.67} & \multicolumn{1}{c}{\textbf{\color{blue} 1.67}} & \multicolumn{1}{c|}{\textbf{\color{blue} 1.67}} & \multicolumn{1}{c|}{ 1.67} & \multicolumn{1}{c}{\textbf{\color{blue} 1.67}} & \multicolumn{1}{c|}{\textbf{\color{blue} 1.67}} \\
 &  6 &  2.26 &  1.76 &  2.59 & \multicolumn{1}{c|}{ 2.20} &  2.27 & \multicolumn{1}{c|}{ 2.28} & \multicolumn{1}{c|}{ 2.26} & \multicolumn{1}{c}{\textbf{\color{blue} 2.26}} & \multicolumn{1}{c|}{\textbf{\color{blue} 2.26}} & \multicolumn{1}{c|}{ 2.26} & \multicolumn{1}{c}{\textbf{\color{blue} 2.26}} & \multicolumn{1}{c|}{\textbf{\color{blue} 2.26}} \\
 &  7 &  1.91 &  1.67 &  1.94 & \multicolumn{1}{c|}{ 1.87} & \textbf{\color{blue} 1.91} & \multicolumn{1}{c|}{\textbf{\color{blue} 1.92}} & \multicolumn{1}{c|}{ 1.91} & \multicolumn{1}{c}{\textbf{\color{blue} 1.91}} & \multicolumn{1}{c|}{\textbf{\color{blue} 1.91}} & \multicolumn{1}{c|}{ 1.91} & \multicolumn{1}{c}{\textbf{\color{blue} 1.91}} & \multicolumn{1}{c|}{\textbf{\color{blue} 1.91}} \\
 &  8 &  1.13 &  1.13 & \textbf{\color{blue} 1.13} & \multicolumn{1}{c|}{ 1.13} & \textbf{\color{blue} 1.13} & \multicolumn{1}{c|}{\textbf{\color{blue} 1.13}} & \multicolumn{1}{c|}{ 1.13} & \multicolumn{1}{c}{\textbf{\color{blue} 1.13}} & \multicolumn{1}{c|}{\textbf{\color{blue} 1.13}} & \multicolumn{1}{c|}{ 1.13} & \multicolumn{1}{c}{\textbf{\color{blue} 1.13}} & \multicolumn{1}{c|}{\textbf{\color{blue} 1.13}} \\
 &  9 &  2.57 &  1.93 &  2.59 & \multicolumn{1}{c|}{ 2.49} &  2.64 & \multicolumn{1}{c|}{ 2.58} & \multicolumn{1}{c|}{ 2.57} & \multicolumn{1}{c}{\textbf{\color{blue} 2.57}} & \multicolumn{1}{c|}{\textbf{\color{blue} 2.57}} & \multicolumn{1}{c|}{ 2.57} & \multicolumn{1}{c}{\textbf{\color{blue} 2.57}} & \multicolumn{1}{c|}{\textbf{\color{blue} 2.57}} \\
 & 10 &  2.18 &  1.84 &  2.35 & \multicolumn{1}{c|}{ 2.14} &  2.19 & \multicolumn{1}{c|}{ 2.19} & \multicolumn{1}{c|}{ 2.17} & \multicolumn{1}{c}{\textbf{\color{blue} 2.18}} & \multicolumn{1}{c|}{\textbf{\color{blue} 2.17}} & \multicolumn{1}{c|}{ 2.17} & \multicolumn{1}{c}{\textbf{\color{blue} 2.17}} & \multicolumn{1}{c|}{\textbf{\color{blue} 2.17}} \\ \hline
 10 &  1 &  5.41 &  1.95 &  9.05 & \multicolumn{1}{c|}{ 5.19} &  5.68 & \multicolumn{1}{c|}{ 5.57} & \multicolumn{1}{c|}{ 5.56} & \multicolumn{1}{c}{ 5.57} & \multicolumn{1}{c|}{ 5.56} & \multicolumn{1}{c|}{ 5.56} & \multicolumn{1}{c}{ 5.56} & \multicolumn{1}{c|}{ 5.56} \\
 &  2 &  3.91 &  1.89 &  5.05 & \multicolumn{1}{c|}{ 3.68} &  4.41 & \multicolumn{1}{c|}{\textbf{\color{blue} 3.92}} & \multicolumn{1}{c|}{ 3.91} & \multicolumn{1}{c}{\textbf{\color{blue} 3.92}} & \multicolumn{1}{c|}{\textbf{\color{blue} 3.91}} & \multicolumn{1}{c|}{ 3.91} & \multicolumn{1}{c}{\textbf{\color{blue} 3.91}} & \multicolumn{1}{c|}{\textbf{\color{blue} 3.91}} \\
 &  3 &  2.37 &  1.88 &  2.60 & \multicolumn{1}{c|}{ 2.31} &  2.42 & \multicolumn{1}{c|}{\textbf{\color{blue} 2.36}} & \multicolumn{1}{c|}{ 2.36} & \multicolumn{1}{c}{\textbf{\color{blue} 2.36}} & \multicolumn{1}{c|}{\textbf{\color{blue} 2.36}} & \multicolumn{1}{c|}{ 2.36} & \multicolumn{1}{c}{\textbf{\color{blue} 2.36}} & \multicolumn{1}{c|}{\textbf{\color{blue} 2.36}} \\
 &  4 &  1.78 &  1.72 &  2.11 & \multicolumn{1}{c|}{ 1.77} & \textbf{\color{blue} 1.79} & \multicolumn{1}{c|}{ 1.80} & \multicolumn{1}{c|}{ 1.78} & \multicolumn{1}{c}{\textbf{\color{blue} 1.78}} & \multicolumn{1}{c|}{\textbf{\color{blue} 1.78}} & \multicolumn{1}{c|}{ 1.78} & \multicolumn{1}{c}{\textbf{\color{blue} 1.78}} & \multicolumn{1}{c|}{\textbf{\color{blue} 1.78}} \\
 &  5 &  2.98 &  1.92 &  3.18 & \multicolumn{1}{c|}{ 2.87} &  3.01 & \multicolumn{1}{c|}{\textbf{\color{blue} 2.98}} & \multicolumn{1}{c|}{ 2.98} & \multicolumn{1}{c}{\textbf{\color{blue} 2.98}} & \multicolumn{1}{c|}{\textbf{\color{blue} 2.98}} & \multicolumn{1}{c|}{ 2.98} & \multicolumn{1}{c}{\textbf{\color{blue} 2.98}} & \multicolumn{1}{c|}{\textbf{\color{blue} 2.98}} \\
 &  6 &  4.49 &  1.89 &  5.34 & \multicolumn{1}{c|}{ 4.21} &  4.59 & \multicolumn{1}{c|}{\textbf{\color{blue} 4.49}} & \multicolumn{1}{c|}{ 4.95} & \multicolumn{1}{c}{ 4.97} & \multicolumn{1}{c|}{ 4.95} & \multicolumn{1}{c|}{ 4.48} & \multicolumn{1}{c}{\textbf{\color{blue} 4.48}} & \multicolumn{1}{c|}{\textbf{\color{blue} 4.48}} \\
 &  7 &  3.45 &  1.87 &  3.74 & \multicolumn{1}{c|}{ 3.26} &  3.49 & \multicolumn{1}{c|}{ 3.47} & \multicolumn{1}{c|}{ 3.45} & \multicolumn{1}{c}{\textbf{\color{blue} 3.45}} & \multicolumn{1}{c|}{\textbf{\color{blue} 3.45}} & \multicolumn{1}{c|}{ 3.45} & \multicolumn{1}{c}{\textbf{\color{blue} 3.45}} & \multicolumn{1}{c|}{\textbf{\color{blue} 3.45}} \\
 &  8 &  3.85 &  1.89 &  5.79 & \multicolumn{1}{c|}{ 3.64} &  3.99 & \multicolumn{1}{c|}{\textbf{\color{blue} 3.85}} & \multicolumn{1}{c|}{ 3.85} & \multicolumn{1}{c}{\textbf{\color{blue} 3.85}} & \multicolumn{1}{c|}{\textbf{\color{blue} 3.85}} & \multicolumn{1}{c|}{ 3.85} & \multicolumn{1}{c}{\textbf{\color{blue} 3.85}} & \multicolumn{1}{c|}{\textbf{\color{blue} 3.85}} \\
 &  9 &  3.06 &  1.88 &  3.27 & \multicolumn{1}{c|}{ 2.90} &  3.09 & \multicolumn{1}{c|}{\textbf{\color{blue} 3.06}} & \multicolumn{1}{c|}{ 3.05} & \multicolumn{1}{c}{\textbf{\color{blue} 3.06}} & \multicolumn{1}{c|}{\textbf{\color{blue} 3.05}} & \multicolumn{1}{c|}{ 3.05} & \multicolumn{1}{c}{\textbf{\color{blue} 3.05}} & \multicolumn{1}{c|}{\textbf{\color{blue} 3.05}} \\
 & 10 &  2.97 &  1.94 &  3.00 & \multicolumn{1}{c|}{ 2.87} & \textbf{\color{blue} 2.98} & \multicolumn{1}{c|}{\textbf{\color{blue} 2.97}} & \multicolumn{1}{c|}{ 2.97} & \multicolumn{1}{c}{\textbf{\color{blue} 2.98}} & \multicolumn{1}{c|}{\textbf{\color{blue} 2.97}} & \multicolumn{1}{c|}{ 2.97} & \multicolumn{1}{c}{\textbf{\color{blue} 2.97}} & \multicolumn{1}{c|}{\textbf{\color{blue} 2.97}} \\ \hline
\end{tabular}
\caption{\footnotesize Showcases the results for the 20 tests conducted on \eqref{eq:QCQP_test_instances} with $m=5$ and $10$. 
The column titled, {\it Optimal}, estimates the opt. value of \eqref{eq:QCQP_test_instances} within $\pm 0.013$ of the exact value. 
Under {\it SDR}, the values in column, {\it Orig.}, are the opt. obj. values of \eqref{eq:QCQP_SDR} while column {\it Random} provides the best obj. value obtained via the randomization process \eqref{eq:rand_QCQP} using twenty samples. 
The header, {\it Our Model}, presents the best obj. values obtained via \eqref{eq:QCQP_SCO} and our approximation procedures for different values of $\delta$; 
{\it Orig.} states the best obj. value of \eqref{eq:QCQP_SCO} obtained from three initializations of Algorithm~\ref{alg:Staged_bcd}; {\it Random} gives the best randomized solution obtained from sampling based on our three solutions to \eqref{eq:QCQP_SCO}, and {\it Project} gives the best value obtained from our three projected solutions obtained by solving \eqref{eq:QCQP_SCO} and projecting as described in (A1.).
Values in blue denote a global opt. solution, within a tolerance.
}
\label{tab:QCQP_tab1}
\end{table}

 \begin{table}[H]
\centering
\begin{tabular}{|cc|c|cc|cccllllll|}
\hline
\multicolumn{1}{|l}{} & \multicolumn{1}{l|}{} & \multicolumn{1}{l|}{}  & \multicolumn{2}{c|}{\textbf{SDR}}            & \multicolumn{9}{c|}{\textbf{Our Model}}                                                                                                                                                                                                                                                                                                                       \\ \cline{6-14}
\multicolumn{1}{|l}{} & \multicolumn{1}{l|}{} & \multicolumn{1}{l|}{}  & \multicolumn{1}{l}{} & \multicolumn{1}{l|}{} & \multicolumn{1}{l}{} & \multicolumn{1}{l}{$\delta = 10^{-1}$} & \multicolumn{1}{l|}{}                    &                                       & $\delta = 10^{-3}$                           & \multicolumn{1}{l|}{}                    &                                       & $\delta = 10^{-6}$                           &                                          \\ \cline{1-14}
\textbf{m}            & \textbf{Test}         & \textbf{Optimal} & \textbf{Orig.}    & \textbf{Random}   & \multicolumn{1}{c|}{\textbf{Orig.}} & \textbf{Random}               & \multicolumn{1}{c|}{\textbf{Project}} & \multicolumn{1}{c|}{\textbf{Orig.}} & \multicolumn{1}{c}{\textbf{Random}} & \multicolumn{1}{c|}{\textbf{Project}} & \multicolumn{1}{c|}{\textbf{Orig.}} & \multicolumn{1}{c}{\textbf{Random}} & \multicolumn{1}{c|}{\textbf{Project}} \\ \hline
 25 &  1 &  6.05 &  1.91 &  6.23 & \multicolumn{1}{c|}{ 5.58} &  6.07 & \multicolumn{1}{c|}{\textbf{\color{blue} 6.06}} & \multicolumn{1}{c|}{ 6.04} & \multicolumn{1}{c}{\textbf{\color{blue} 6.05}} & \multicolumn{1}{c|}{\textbf{\color{blue} 6.05}} & \multicolumn{1}{c|}{ 6.05} & \multicolumn{1}{c}{\textbf{\color{blue} 6.05}} & \multicolumn{1}{c|}{\textbf{\color{blue} 6.05}} \\
 &  2 &  4.84 &  1.96 &  6.06 & \multicolumn{1}{c|}{ 4.53} &  4.92 & \multicolumn{1}{c|}{ 4.85} & \multicolumn{1}{c|}{ 4.83} & \multicolumn{1}{c}{\textbf{\color{blue} 4.84}} & \multicolumn{1}{c|}{\textbf{\color{blue} 4.83}} & \multicolumn{1}{c|}{ 4.83} & \multicolumn{1}{c}{\textbf{\color{blue} 4.83}} & \multicolumn{1}{c|}{\textbf{\color{blue} 4.83}} \\
 &  3 &  4.14 &  1.95 &  4.29 & \multicolumn{1}{c|}{ 3.89} &  4.24 & \multicolumn{1}{c|}{ 4.16} & \multicolumn{1}{c|}{ 4.14} & \multicolumn{1}{c}{\textbf{\color{blue} 4.15}} & \multicolumn{1}{c|}{\textbf{\color{blue} 4.14}} & \multicolumn{1}{c|}{ 4.14} & \multicolumn{1}{c}{\textbf{\color{blue} 4.14}} & \multicolumn{1}{c|}{\textbf{\color{blue} 4.14}} \\
 &  4 &  7.67 &  1.94 &  9.62 & \multicolumn{1}{c|}{ 7.03} &  7.84 & \multicolumn{1}{c|}{\textbf{\color{blue} 7.68}} & \multicolumn{1}{c|}{ 7.66} & \multicolumn{1}{c}{\textbf{\color{blue} 7.67}} & \multicolumn{1}{c|}{\textbf{\color{blue} 7.67}} & \multicolumn{1}{c|}{ 7.67} & \multicolumn{1}{c}{\textbf{\color{blue} 7.67}} & \multicolumn{1}{c|}{\textbf{\color{blue} 7.67}} \\
 &  5 & 10.06 &  1.96 & 13.55 & \multicolumn{1}{c|}{ 9.20} & 10.08 & \multicolumn{1}{c|}{10.07} & \multicolumn{1}{c|}{10.05} & \multicolumn{1}{c}{10.09} & \multicolumn{1}{c|}{\textbf{\color{blue}10.06}} & \multicolumn{1}{c|}{10.06} & \multicolumn{1}{c}{\textbf{\color{blue}10.06}} & \multicolumn{1}{c|}{\textbf{\color{blue}10.06}} \\
 &  6 &  3.53 &  1.91 &  4.32 & \multicolumn{1}{c|}{ 3.35} & \textbf{\color{blue} 3.53} & \multicolumn{1}{c|}{\textbf{\color{blue} 3.53}} & \multicolumn{1}{c|}{ 3.53} & \multicolumn{1}{c}{\textbf{\color{blue} 3.53}} & \multicolumn{1}{c|}{\textbf{\color{blue} 3.53}} & \multicolumn{1}{c|}{ 3.53} & \multicolumn{1}{c}{\textbf{\color{blue} 3.53}} & \multicolumn{1}{c|}{\textbf{\color{blue} 3.53}} \\
 &  7 &  8.07 &  1.96 &  9.13 & \multicolumn{1}{c|}{ 7.43} &  8.21 & \multicolumn{1}{c|}{ 8.08} & \multicolumn{1}{c|}{ 8.06} & \multicolumn{1}{c}{\textbf{\color{blue} 8.08}} & \multicolumn{1}{c|}{\textbf{\color{blue} 8.07}} & \multicolumn{1}{c|}{ 8.07} & \multicolumn{1}{c}{\textbf{\color{blue} 8.07}} & \multicolumn{1}{c|}{\textbf{\color{blue} 8.07}} \\
 &  8 &  4.70 &  1.95 &  7.12 & \multicolumn{1}{c|}{ 4.39} &  4.72 & \multicolumn{1}{c|}{ 4.71} & \multicolumn{1}{c|}{ 4.69} & \multicolumn{1}{c}{\textbf{\color{blue} 4.70}} & \multicolumn{1}{c|}{\textbf{\color{blue} 4.70}} & \multicolumn{1}{c|}{ 4.70} & \multicolumn{1}{c}{\textbf{\color{blue} 4.70}} & \multicolumn{1}{c|}{\textbf{\color{blue} 4.70}} \\
 &  9 &  7.69 &  1.97 &  7.76 & \multicolumn{1}{c|}{ 7.09} &  7.83 & \multicolumn{1}{c|}{\textbf{\color{blue} 7.70}} & \multicolumn{1}{c|}{ 7.68} & \multicolumn{1}{c}{\textbf{\color{blue} 7.69}} & \multicolumn{1}{c|}{\textbf{\color{blue} 7.69}} & \multicolumn{1}{c|}{ 7.69} & \multicolumn{1}{c}{\textbf{\color{blue} 7.69}} & \multicolumn{1}{c|}{\textbf{\color{blue} 7.69}} \\
 & 10 &  4.91 &  1.93 &  6.83 & \multicolumn{1}{c|}{ 4.59} &  4.95 & \multicolumn{1}{c|}{\textbf{\color{blue} 4.91}} & \multicolumn{1}{c|}{ 4.90} & \multicolumn{1}{c}{\textbf{\color{blue} 4.92}} & \multicolumn{1}{c|}{\textbf{\color{blue} 4.91}} & \multicolumn{1}{c|}{ 4.91} & \multicolumn{1}{c}{\textbf{\color{blue} 4.91}} & \multicolumn{1}{c|}{\textbf{\color{blue} 4.91}} \\ \hline
 50 &  1 &  5.31 &  1.95 &  5.67 & \multicolumn{1}{c|}{ 4.94} &  5.34 & \multicolumn{1}{c|}{ 5.33} & \multicolumn{1}{c|}{ 5.31} & \multicolumn{1}{c}{\textbf{\color{blue} 5.31}} & \multicolumn{1}{c|}{\textbf{\color{blue} 5.31}} & \multicolumn{1}{c|}{ 5.31} & \multicolumn{1}{c}{\textbf{\color{blue} 5.31}} & \multicolumn{1}{c|}{\textbf{\color{blue} 5.31}} \\
 &  2 &  9.35 &  1.97 & 10.14 & \multicolumn{1}{c|}{ 8.54} &  9.40 & \multicolumn{1}{c|}{\textbf{\color{blue} 9.35}} & \multicolumn{1}{c|}{ 9.34} & \multicolumn{1}{c}{\textbf{\color{blue} 9.35}} & \multicolumn{1}{c|}{\textbf{\color{blue} 9.35}} & \multicolumn{1}{c|}{ 9.35} & \multicolumn{1}{c}{\textbf{\color{blue} 9.35}} & \multicolumn{1}{c|}{\textbf{\color{blue} 9.35}} \\
 &  3 &  7.25 &  1.95 &  7.34 & \multicolumn{1}{c|}{ 6.66} &  7.39 & \multicolumn{1}{c|}{ 7.27} & \multicolumn{1}{c|}{ 7.24} & \multicolumn{1}{c}{\textbf{\color{blue} 7.25}} & \multicolumn{1}{c|}{\textbf{\color{blue} 7.25}} & \multicolumn{1}{c|}{ 7.25} & \multicolumn{1}{c}{\textbf{\color{blue} 7.25}} & \multicolumn{1}{c|}{\textbf{\color{blue} 7.25}} \\
 &  4 &  9.49 &  1.96 & 11.01 & \multicolumn{1}{c|}{ 8.66} &  9.54 & \multicolumn{1}{c|}{\textbf{\color{blue} 9.49}} & \multicolumn{1}{c|}{ 9.47} & \multicolumn{1}{c}{ 9.50} & \multicolumn{1}{c|}{\textbf{\color{blue} 9.48}} & \multicolumn{1}{c|}{ 9.48} & \multicolumn{1}{c}{\textbf{\color{blue} 9.48}} & \multicolumn{1}{c|}{\textbf{\color{blue} 9.48}} \\
 &  5 &  9.13 &  1.95 & 10.44 & \multicolumn{1}{c|}{ 8.33} &  9.20 & \multicolumn{1}{c|}{\textbf{\color{blue} 9.13}} & \multicolumn{1}{c|}{ 9.12} & \multicolumn{1}{c}{\textbf{\color{blue} 9.13}} & \multicolumn{1}{c|}{\textbf{\color{blue} 9.12}} & \multicolumn{1}{c|}{ 9.12} & \multicolumn{1}{c}{\textbf{\color{blue} 9.12}} & \multicolumn{1}{c|}{\textbf{\color{blue} 9.12}} \\
 &  6 &  7.69 &  1.96 &  8.07 & \multicolumn{1}{c|}{ 7.03} &  7.70 & \multicolumn{1}{c|}{\textbf{\color{blue} 7.69}} & \multicolumn{1}{c|}{ 7.68} & \multicolumn{1}{c}{\textbf{\color{blue} 7.69}} & \multicolumn{1}{c|}{\textbf{\color{blue} 7.69}} & \multicolumn{1}{c|}{ 7.69} & \multicolumn{1}{c}{\textbf{\color{blue} 7.69}} & \multicolumn{1}{c|}{\textbf{\color{blue} 7.69}} \\
 &  7 & 14.93 &  1.97 & 17.22 & \multicolumn{1}{c|}{13.58} & 16.66 & \multicolumn{1}{c|}{\textbf{\color{blue}14.94}} & \multicolumn{1}{c|}{17.17} & \multicolumn{1}{c}{17.19} & \multicolumn{1}{c|}{17.18} & \multicolumn{1}{c|}{14.92} & \multicolumn{1}{c}{\textbf{\color{blue}14.92}} & \multicolumn{1}{c|}{\textbf{\color{blue}14.92}} \\
 &  8 & 11.26 &  1.95 & 13.43 & \multicolumn{1}{c|}{11.36} & 12.57 & \multicolumn{1}{c|}{12.50} & \multicolumn{1}{c|}{12.48} & \multicolumn{1}{c}{12.50} & \multicolumn{1}{c|}{12.49} & \multicolumn{1}{c|}{11.26} & \multicolumn{1}{c}{\textbf{\color{blue}11.26}} & \multicolumn{1}{c|}{\textbf{\color{blue}11.26}} \\
 &  9 &  5.82 &  1.96 &  6.39 & \multicolumn{1}{c|}{ 5.40} &  5.91 & \multicolumn{1}{c|}{ 5.84} & \multicolumn{1}{c|}{ 5.82} & \multicolumn{1}{c}{\textbf{\color{blue} 5.83}} & \multicolumn{1}{c|}{\textbf{\color{blue} 5.82}} & \multicolumn{1}{c|}{ 5.82} & \multicolumn{1}{c}{\textbf{\color{blue} 5.82}} & \multicolumn{1}{c|}{\textbf{\color{blue} 5.82}} \\
 & 10 & 11.79 &  1.97 & 12.47 & \multicolumn{1}{c|}{10.90} & 12.10 & \multicolumn{1}{c|}{11.98} & \multicolumn{1}{c|}{11.95} & \multicolumn{1}{c}{11.97} & \multicolumn{1}{c|}{11.96} & \multicolumn{1}{c|}{11.96} & \multicolumn{1}{c}{11.96} & \multicolumn{1}{c|}{11.96} \\ \hline
\end{tabular}
\caption{Showcases the results for the 20 tests conducted on \eqref{eq:QCQP_test_instances} with $m=25$ and $50$.
The description of the entries in this table are identical to Table~\ref{tab:QCQP_tab1}.
}
\label{tab:QCQP_tab2}
\end{table}
\end{landscape}

\begin{figure}[H]
    \hspace{-1.30in}
    \includegraphics[width=1.35\textwidth]{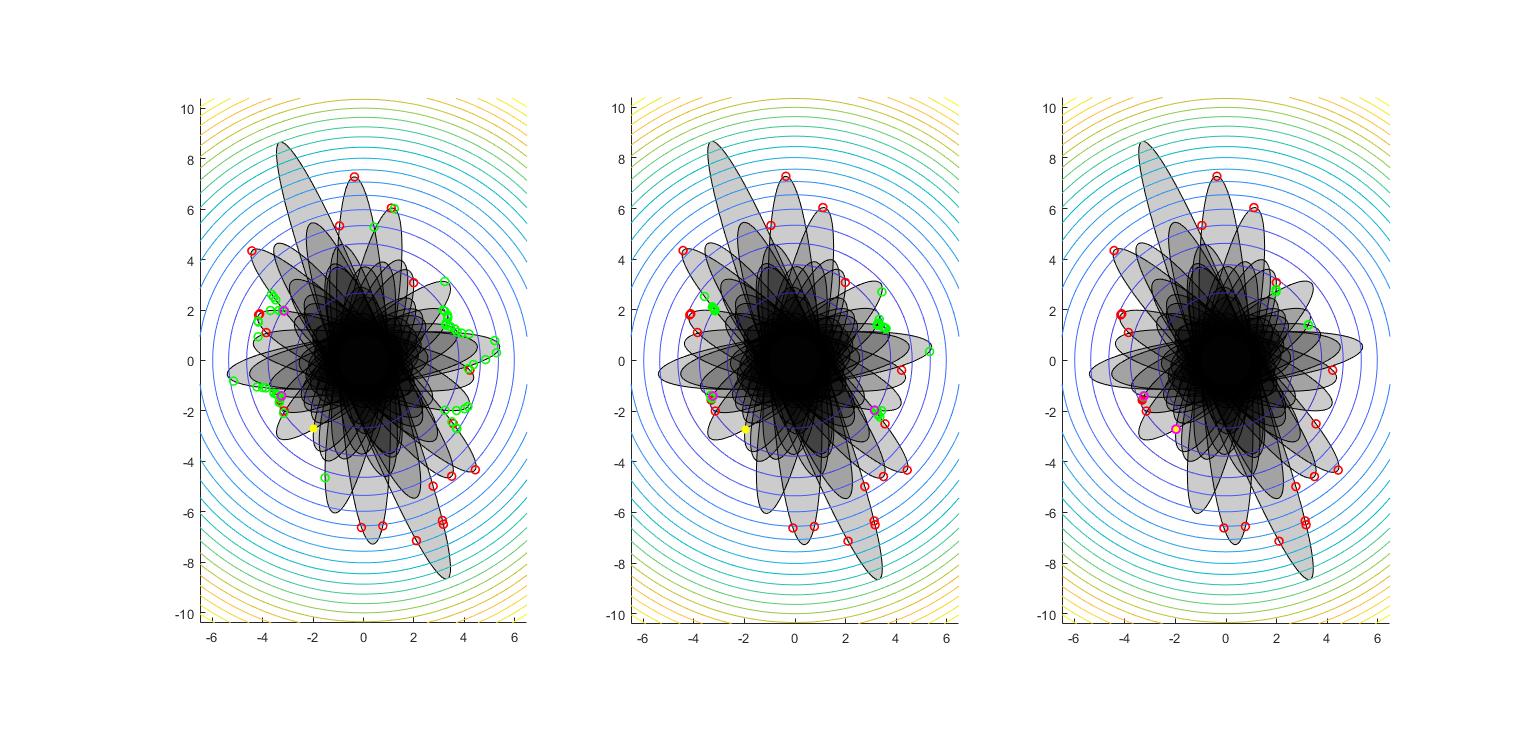}
    \caption{Displays the feasible points computed by each method for Test 8 with $m=50$. The yellow star is an estimated global optimal solution, the red circles are the randomized solutions generated from the SDR approach, the green circles are the candidate points generated from our method via (A2.), and magenta circles are the projected solutions obtained via (A1.). 
    From left to right, the three panels display the results for \eqref{eq:QCQP_SCO} with $\delta=10^{-1}$, $\delta=10^{-3}$, and  $\delta=10^{-6}$.}
    \label{fig:QCQP_sampled_solution}
\end{figure}

The tables showcase our method was able to drastically outperform the standard SDR approach for solving \eqref{eq:QCQP_test_instances}. 
Each of the blue values in Tables~\ref{tab:QCQP_tab1} and \ref{tab:QCQP_tab2} denote an objective value obtained which was within $0.09 - 0.0126$ of the global optimal value estimated via grid search. 
In the 40 numerical tests, the SDR approach was only able to compute a near optimal solution 3 times while our approach with $\delta=10^{-6}$, with both approaches (A1.) and (A2.), was able to solve 38 for the 40 tests. 
Similarly, with $\delta = 10^{-1}$ and $10^{-3}$, the projection approach was able to compute near optimal solutions in 24 and 35 of the experiments respectively. 
Furthermore, our projection approach, (A1.), was able to successfully compute near optimal solutions without the need for randomization.
Figure~\ref{fig:QCQP_sampled_solution} displays an example of the feasible points generated by the different methods. 

A few observations can be made from Figure~\ref{fig:QCQP_sampled_solution}. 
First, the randomized solutions generated from the semidefinite programming approach spread themselves out across the feasible region greatly while the general spread of the randomized solutions via our approach is highly dependent upon the value of $\delta$.
For the green circles in the far left panel of Figure~\ref{fig:QCQP_sampled_solution}, we see the spread is less than that of the red circles, and we note the sampled distribution of points collapses around about four local minimums when $\delta$ decreases to $10^{-3}$ in the middel panel.
So, while the randomized points computed via \eqref{eq:QCQP_SCO} seem to cluster around local minimums, the randomized points computed from SDR do not seem to exhibit any such behavior. 
Second, the projected solutions obtain via (A1.) seem to locate local minimums in this example without the need for randomization. Note, they are the magenta circles in Figure~\ref{fig:QCQP_sampled_solution}. 
This numerical experiment was an interesting example because it was the only instance where our method with the two larger $\delta$ values was unable to compute a global solution while the smallest value was able to do so. The far right panel in Figure~\ref{fig:QCQP_sampled_solution} displays the results for $\delta=10^{-6}$ on this problem. 
From the figure we can see that one of the projected solutions located the global optimal value, and furthermore all of the randomized points generated by out method collapsed very tightly about four different local minimum, including the global minimums.

Overall, the numerical results demonstrate our method presents a novel way of approaching a difficult NP-Hard problem and can best the classical and effective SDR approach with randomization.
The general behavior of our approach is worth further investigation as well from the observations stated here.
Also, it should noted the procedure presented in this section extends over to any problem, not just QCQPs, where rank constraints are removed in-order to obtain easier and sometimes convex problems. 
Hence, our model could greatly impact the way numerous problems are solved.  

\section{Conclusion}\label{sec:conclusion}
We have set out in this paper to present the first general matrix optimization framework which enables the modeling of both coordinate and spectral constraints. Further, we developed and theoretically guaranteed the performance of an algorithm
which computes stationary points for our model. 
Under standard assumptions, we proved our algorithm converges to an $(\epsilon,\epsilon)$-approximate first-order KKT point in $\mathcal{O}(1/\epsilon^2)$ iterations, and we demonstrated our approach could compute global minimums to challenging nonconvex problems.
In our numerical experiments, we demonstrated a novel nonconvex relaxation of a quadratically constrained quadratic program was solvable by our method to global optimality and significantly outperformed the state-of-the-art semidefinite relaxation approach with randomization. 
We also presented a generalized form of semidefinite programming which extends a current essential paradigm. 

The results of this paper are exciting. 
We have extended matrix optimization into a new arena and demonstrated novel nonconvex relaxations are possible which can outperform current frameworks.
However, there is still much work to be done on this research program. 
For starters, though our algorithm performed well in our numerical experiments, the method is not computationally efficient enough to be applied to large-scale problems. 
The development of more efficient methods for solving \eqref{eqn:gen_spec_coord} is of first importance moving forward, and we believe there are several ways to utilize developments in nonconvex optimization to these ends. 
Also, investigating theoretical guarantees for approximation methods resulting from our nonconvex relaxations of rank-one constraints would be an intriguing avenue for future work. 



\appendix
\section{Gradient Lipschitz Properties of the Reformulation}\label{sec:appendix_lip}

\begin{theorem}\label{thm:lip_thm}
Assume $F:\Symn \rightarrow \R$ is gradient Lipschitz with parameter $L_F >0$, i.e., 
\[
\| \nabla F(\bX_1) - \bnabla F(\bX_2) \|_F \leq L_F \| \bX_1 - \bX_2\|_F, \; \forall \bX_1, \bX_2 \in \Symn, 
\]
and define $f:\mathcal{O}(n,n) \times \R^n \rightarrow \R$ using the spectral decomposition of $\bX \in \Symn$ such that,
\[
f(\bQ, \blambda) := F(\bQ \text{\rm Diag}(\blambda) \bQ^\top).
\]
If we restrict the domain of $f$ to $\mathcal{O}(n,n) \times \mathcal{C}$ where $\cC$ is a bounded subset of $\R^n$ such that $\|\blambda\| \leq D$ for all $\blambda \in \cC$, then $f$ satisfies the Lipschitz conditions in Assumption 2. 
\end{theorem}

\begin{proof}   
Let $\bQ=[ \bq_1 \; | \; \hdots \; | \bq_n ]\in \mathcal{O}(n,n)$ be fixed. Define the function $H:\R^n \rightarrow \R$ such that,
$
H(\blambda; \bQ) := F(\bQ \text{Diag}(\blambda)\bQ^\top). 
$
Then, 
\[
H(\blambda;\bQ) = F(\bQ \text{Diag}(\blambda) \bQ^\top) = F\left(\sum_{i=1}^{n} \lambda_i \bq_i \bq_i^\top\right), 
\]
and computing the partial derivatives of $H$ we have,
\begin{align}
\frac{\partial H(\blambda;\bQ)}{\partial \lambda_i}  &= \frac{d}{dt}\left( F\left(\lambda_1 \bq_1 \bq_1^\top + \cdots +  \lambda_{i-1} \bq_{i-1} \bq_{i-1}^\top + (\lambda_i+t) \bq_i \bq_i^\top + \cdots + \lambda_n \bq_n \bq_n^\top \right) \right)\bigg|_{t=0} \nonumber \\
&= \frac{d}{dt}\left( F\left( \bQ \text{Diag}(\blambda) \bQ^\top +   t\bq_i \bq_i^\top \right) \right)\bigg|_{t=0} \nonumber \\  
&= \langle \nabla F(\bQ \text{Diag}(\blambda) \bQ^\top), \bq_i \bq_i^\top \rangle \nonumber.  
\end{align}
Then, 
\begin{align}
\| \nabla_{\blambda} H(\blambda_1;\bQ) - \nabla_{\blambda} H(\blambda_2;\bQ) \|^2 &= \sum_{i=1}^{n} \left[  \langle \nabla F(\bQ \text{Diag}(\blambda_1) \bQ^\top) -  \nabla F(\bQ \text{Diag}(\blambda_2) \bQ^\top), \bq_i \bq_i^\top \rangle  \right]^2 \nonumber  \\ 
&\leq  \sum_{i=1}^{n} \| \nabla F(\bQ \text{Diag}(\blambda_1) \bQ^\top) -  \nabla F(\bQ \text{Diag}(\blambda_2) \bQ^\top) \|_F^2 \| \bq_i \bq_i^\top \|_F^2 \nonumber \\
&\leq  \sum_{i=1}^{n} L_F^2 \|\bQ \text{Diag}(\blambda_1) \bQ^\top - \bQ \text{Diag}(\blambda_2) \bQ^\top \|_F^2 \nonumber \\
&\leq  \sum_{i=1}^{n} L_F^2 \|\bQ(\text{Diag}(\blambda_1) - \text{Diag}(\blambda_2) )\bQ^\top\|_F^2 \nonumber \\
&= n L_F^2 \| \blambda_1  - \blambda_2\|^2, \nonumber 
\end{align}
where the second inequality follows from the gradient Lipschitz condition on $F$. Thus, 
\[
\| \nabla_{\blambda} H(\blambda_1;\bQ) - \nabla_{\blambda} H(\blambda_2;\bQ) \| \leq \sqrt{n} L_F \| \blambda_1 - \blambda_2 \|, \;\; \forall  \blambda_1, \blambda_2 \in \R^n. 
\]

For fixed $\blambda \in \R^n$, define the function $G:\R^{n\times n} \rightarrow \R$ such that, 
\[
G(\bm{U};\blambda) =  F(\bm{U} \text{Diag}(\blambda)\bm{U}^\top).
\]
We next compute the gradient of $G$. Let $\bX=\bm{U}\text{Diag}(\blambda) \bm{U}^\top$ and $\bm{E} \in \R^{n\times n}$. Then, 
\begin{align}
\frac{d}{dt}\left(G(\bm{U}+t \bm{E})\right)\bigg|_{t=0} &=  \frac{d}{dt}\left( F( \left[\bm{U}+t\bm{E}\right]\text{Diag}(\blambda) \left[\bm{U}+t\bm{E}\right]^\top)  \right)\bigg|_{t=0} \nonumber \\
&= \frac{d}{dt}\left( F ( \bX + t( \bm{E}\text{Diag}(\blambda) \bm{U}^\top + \bm{U} \text{Diag}(\blambda) \bm{E}^\top) + \mathcal{O}(t^2) \right)\bigg|_{t=0} \nonumber \\
&= \langle \nabla F(\bX),\; \bm{E}\text{Diag}(\blambda) \bm{U}^\top + \bm{U} \text{Diag}(\blambda) \bm{E}^\top \rangle \nonumber \\
&= \langle \nabla F(\bX) \bm{U} \text{Diag}(\blambda) + \nabla F(\bX)^\top \bm{U} \text{Diag}(\blambda), \; \bm{E} \rangle \nonumber \\
&= \langle \nabla G(\bm{U}), \bm{E} \rangle. \nonumber 
\end{align}
Thus, 
\[
\nabla G(\bm{U}) = \nabla F(\bX) \bm{U} \text{Diag}(\blambda) + \nabla F(\bX)^\top \bm{U} \text{Diag}(\blambda). 
\]
We leverage Lemma 2.7 of \cite{boumal2019global} to prove $G$ has a gradient Lipschitz retraction over the manifold of orthogonal matrices, i.e., we shall show $G$ is gradient Lipschitz over the convex hull of $\mathcal{O}(n,n)$.
Let $\bm{U}, \bm{V} \in \R^{n\times n}$ be contained in the convex hull of $\mathcal{O}(n,n)$, $\blambda \in \R^n$ be fixed, $\bX = \bm{U}\text{Diag}(\blambda) \bm{U}^\top$, and $\bY = \bm{V} \text{Diag}(\blambda) \bm{V}^\top$. Then, 
\begin{align}\label{eqn:G_diff}
\|\nabla G(\bm{U}) - \nabla G(\bm{V})\|_F &= \| \overbrace{\left( \nabla F(\bX) + \nabla F(\bX)^\top \right)}^{\bm{\Sigma_x}:=} \bm{U} \text{Diag}(\blambda) -  \overbrace{\left( \nabla F(\bY) + \nabla F(\bY)^\top \right)}^{\bm{\Sigma}_Y:=} \bm{V} \text{Diag}(\blambda)\|_F \nonumber  \\
&\leq\|  \bm{\Sigma}_x \bm{U} - \bm{\Sigma}_x \bm{V} + \bm{\Sigma}_x \bm{V} - \bm{\Sigma}_y \bm{V} \|_F \cdot \| \blambda\| \nonumber \\
&\leq D \| \bm{\Sigma}_x \left(\bm{U} - \bm{V} \right) + \left(\bm{\Sigma}_x - \bm{\Sigma}_y\right) \bm{V} \|_F \nonumber \\
&\leq D \|\bm{\Sigma}_x\|_F \| \bm{U} - \bm{V}\|_F + D \| \bV\|_F \|\bSigma_x - \bSigma_y\|_F. 
\end{align}
Since we are assuming $F$ is gradient Lipschitz over the set of symmetric matrices, $\|\blambda\|$ is bounded, and $\bU$ and $\bV$ are contained in the convex hull of $\mathcal{O}(n,n)$, which is a bounded subset of $\R^{n\times n}$, it follows there exists $M_F > 0$ such that, 
\begin{equation}\label{eqn:Fgrad_bound}
\| \nabla F(\bX) \|_F \leq M_F, \; \forall \bX = \bU \Diag(\blambda) \bU^\top, 
\end{equation}
with $\bU$ and $\blambda$ as defined above. 
Additionally, be virtue of $\bU$ being in the convex hull of $\cO(n,n)$ it follows $\bU$ is the convex combination of orthogonal matrices $\bA_1, \hdots, \bA_p$.  From this a simple bound on the norm of $\bU$ is achieved. Assume $\bU = \sum_{i}^{p} \alpha_i \bA_i$ with $\alpha_i > 0$ and $\sum_{i=1}^{p} \alpha_i =1$. Then,  
\begin{equation}\label{eqn:U_bound}
\|\bU\|_F = \| \sum_{i}^{p} \alpha_i \bA_i \|_F \leq \sum_{i}^{p} \alpha_i \|\bA_i \|_F \leq \sum_{i=1}^{p} \alpha_i \sqrt{n} = \sqrt{n}.  
\end{equation}
Next, we bound the last term in \eqref{eqn:G_diff}. So, 
\begin{align}\label{eqn:F_diff_bound}
\| \bSigma_x - \bSigma_y \|_F &\leq 2 \| \nabla F(\bX) - \nabla F(\bY)\|_F \nonumber\\
&\leq 2L_F \| \bU \Diag(\blambda) \bU^\top - \bV \Diag(\blambda) \bV^\top \|_F \nonumber \\
&= 2 L_F \| \bU \Diag(\blambda) \bU^\top - \bU \Diag(\blambda) \bV^\top + \bU \Diag(\blambda) \bV^\top - \bV \Diag(\blambda) \bV^\top\|_F \nonumber \\
&=2 L_F \| \bU \Diag(\blambda) \left( \bU - \bV\right)^\top + \left(\bU-\bV\right) \Diag(\blambda) \bV^\top \|_F \nonumber\\
&\leq 2 L_F \left( \|\bU \Diag(\blambda)\|_F + \|\Diag(\blambda)\bV^\top\|_F\right) \| \bU - \bV \|_F \nonumber \\
&\leq 4L_F D \sqrt{n} \|\bU - \bV\|_F.
\end{align}
Combining \eqref{eqn:Fgrad_bound}, \eqref{eqn:U_bound} and  \eqref{eqn:F_diff_bound} with \eqref{eqn:G_diff} we obtain, 
\begin{align}
\| \nabla G(\bU) - \nabla G(\bV)\|_F &\leq 2DM_F \|\bU-\bV\|_F + 4L_FD^2n \|\bU-\bV\|_F \nonumber \\
&=  2D\left(M_F + 2L_FDn \right)\|\bU - \bV\|_F, \nonumber
\end{align}
and since this holds for any $\bU$ and $\bV$ in the convex hull of $\mathcal{O}(n,n)$ it follows $G$ is gradient Lipschitz over the convex hull of the orthogonal matrices. Therefore, by Lemma 2.7 in \cite{boumal2019global}, $G$ restricted to $\mathcal{O}(n,n)$ has a gradient Lipschitz retraction and this holds for all $\blambda \in \cC$.  

Lastly, let $\bU, \bV \in \cO(n,n)$ with $\bX$ and $\bY$ defined as before, $\bm{u}_i$ and $\bm{v}_i$ be the $i$-th columns of $\bU$ and $\bV$ respectively, and $\blambda \in \R^n$. Then, 
\begin{align}\label{eqn:bound_on_H}
\| \nabla_{\lambda} H(\blambda; \bU) - \nabla_{\lambda} H(\blambda; \bV) \|^2 &= \sum_{i=1}^{n} \left( \langle \nabla F(\bX), \bm{u}_i \bm{u}_i^\top \rangle - \langle \nabla F(\bY), \bm{v}_i \bm{v}_i^\top \rangle    \right)^2  \nonumber \\
&= \sum_{i=1}^{n} \left( \langle \nabla F(\bX) - \nabla F(\bY), \bm{u}_i \bm{u}_i^\top \rangle - \langle \nabla F(\bY), \bm{v}_i \bm{v}_i^\top - \bm{u}_i \bm{u}_i^\top \rangle    \right)^2  \nonumber \\
&\leq \sum_{i=1}^{n} \left( \| \nabla F(\bX) - \nabla F(\bY)\| \cdot \|\bm{u}_i \bm{u}_i^\top\| + \|\nabla F(\bY)\|\cdot \|\bm{v}_i \bm{v}_i^\top - \bm{u}_i \bm{u}_i^\top \| \right)^2  \nonumber \\
&\leq \sum_{i=1}^{n} \left(2L_F D \sqrt{n} \|\bU - \bV\|_F + M_F\|\bm{v}_i \bm{v}_i^\top - \bm{u}_i \bm{u}_i^\top \| \right)^2,   
\end{align}
where we used the bounds \eqref{eqn:Fgrad_bound} and \eqref{eqn:F_diff_bound}. Focusing on the last term in \eqref{eqn:bound_on_H}, 
\begin{align}\label{eqn:uuvv_bound}
\|\bm{v}_i \bm{v}_i^\top - \bm{u}_i \bm{u}_i^\top \|_F &= \| \bV \be_i \be_i^\top \bV^\top - \bU \be_i \be_i^\top \bU^\top\|_F \nonumber \\
&= \| \bV \be_i \be_i^\top \bV^\top -\bV \be_i \be_i^\top \bU^\top + \bV \be_i \be_i^\top \bU^\top  - \bU \be_i \be_i^\top \bU^\top \|_F \nonumber \\
&= \| \bV \be_i \be_i^\top \left(\bV - \bU \right)^\top + \left(\bV-\bU\right) \be_i \be_i^\top \bU^\top \|_F \nonumber \\
&\leq \left(\|\be_i \be_i^\top\|_F \cdot (\|\bV\|_F + \|\bU\|_F) \right) \|\bU-\bV\|_F\nonumber \\ 
&= 2\sqrt{n} \|\bU-\bV\|_F. 
\end{align}
Combining \eqref{eqn:bound_on_H} and \eqref{eqn:uuvv_bound},
\begin{align}
\| \nabla_{\lambda} H(\blambda; \bU) - \nabla_{\lambda} H(\blambda; \bV) \|^2 &\leq \sum_{i=1}^{n}\left(2L_F D \sqrt{n} \|\bU - \bV\|_F + 2\sqrt{n}M_F\|\bU - \bV \|_F \right)^2, \nonumber \\
&= \sum_{i=1}^{n}\left(2\sqrt{n} (L_F D + M_F )\|\bU - \bV \|_F \right)^2, \nonumber \\
&= 4n^2 \left( L_F D + M_F \right)^2 \|\bU - \bV\|_F^2. \nonumber
\end{align}
\end{proof}

\end{document}